\documentclass[oneside,reqno,11pt]{amsart}

\usepackage{graphicx}
\usepackage{epstopdf}
\usepackage{amsmath}
\usepackage{amssymb}
\usepackage[top=1in, bottom=1.25in, left=1.25in, right=1.25in]{geometry}
\usepackage{textcomp}
\usepackage{mathrsfs}
\usepackage{tikz}
\usepackage{tikz-cd}
\usepackage[utf8]{inputenc}
\usepackage{mathabx}
\usepackage{bm}
\usepackage{amsthm}
\usepackage{amsfonts}
\usepackage[colorlinks=true,linkcolor=blue]{hyperref}
\usepackage{pb-diagram}
\usepackage{lipsum}
\usepackage{secdot}
\usepackage{color}
\usepackage[normalem]{ulem}
\usepackage{enumitem}
\usepackage{dsfont} 
\usepackage{appendix}
\usepackage{color}
\usepackage{calc}
\usepackage{slantsc}
\usepackage[sc]{mathpazo}
\usepackage{indentfirst}
\usepackage{mathtools}
\usepackage{hyperref}

\DeclarePairedDelimiter\floor{\lfloor}{\rfloor}

\theoremstyle{plain}
\newtheorem{theorem}{Theorem}[section]
\newtheorem{lemma}[theorem]{Lemma}
\newtheorem{corollary}[theorem]{Corollary}

\newtheorem{prop}[theorem]{Proposition}
\newtheorem{defn}[theorem]{Definition}

\newtheorem{remark}[theorem]{Remark}

\numberwithin{equation}{section}

\newcommand*{\medcup}{\mathbin{\scalebox{1.5}{\ensuremath{\cup}}}}

\newcommand{\twobar}{/\kern-0.5em/}
\newcommand{\threebar}{/\kern-0.2em/\kern-0.2em/}

\newcommand\norm[1]{\left\lVert#1\right\rVert}

\newcommand{\tr}{\mathrm{tr} \,}

\newcommand{\ev}{\mathrm{ev}}

\newcommand{\pt}{\mathrm{pt}}

\newcommand{\bN}{\mathbb{N}}
\newcommand{\bP}{\mathbb{P}}
\newcommand{\bS}{\mathbb{S}}
\newcommand{\bT}{\mathbf{T}}
\newcommand{\Z}{\mathbb{Z}}
\newcommand{\Q}{\mathbb{Q}}
\newcommand{\R}{\mathbb{R}}
\newcommand{\C}{\mathbb{C}}

\newcommand{\fm}{\mathfrak{m}}
\newcommand{\one}{\mathbf{1}}
\newcommand{\fd}{\mathfrak{d}}
\newcommand{\fs}{\mathfrak{s}}
\newcommand{\ff}{\mathfrak{f}}

\newcommand{\forget}{\mathfrak{forget}}

\newcommand{\val}{\mathrm{val}}
\newcommand{\Crit}{\mathrm{Crit}}

\newcommand{\into}{\hookrightarrow}

\newcommand{\Image}{\mathrm{Im}}

\newcommand{\cM}{\mathcal{M}}

\newcommand{\cI}{\mathcal{I}}

\newcommand{\Fuk}{\mathrm{Fuk}}

\newcommand{\wT}{\triangledown}
\newcommand{\blT}{\blacktriangledown}

\newcommand{\wunit}{\bm{1}^{\triangledown}}
\newcommand{\bunit}{\bm{1}^{\blacktriangledown}}
\newcommand{\gunit}{\bm{1}^{\textcolor{gray}{\blacktriangledown}}}
\newcommand{\wlambda}{\bm{\lambda}^{\triangledown}}
\newcommand{\blambda}{\bm{\lambda}^{\blacktriangledown}}
\newcommand{\glambda}{\bm{\lambda}^{\textcolor{gray}{\blacktriangledown}}}

\newcommand{\Spec}{\operatorname{Spec}}

\newcommand{\re}{\operatorname{Re}}

\makeatletter
\newcommand*\bigcdot{\mathpalette\bigcdot@{.5}}
\newcommand*\bigcdot@[2]{\mathbin{\vcenter{\hbox{\scalebox{#2}{$\m@th#1\bullet$}}}}}
\makeatother

\usetikzlibrary{decorations.pathreplacing}

\makeatletter
\def\underbrace#1{\@ifnextchar_{\tikz@@underbrace{#1}}{\tikz@@underbrace{#1}_{}}}
\def\tikz@@underbrace#1_#2{\tikz[baseline=(a.base)] {\node (a) {\(#1\)}; \draw[ultra thick,line cap=round,decorate,decoration={brace,amplitude=5pt}] (a.south east) -- node[below,inner sep=7pt] {\(\scriptstyle #2\)} (a.south west);}}
\makeatother

\begin{document}
	\title{$T$-equivariant disc potential and SYZ mirror construction}

\author{Yoosik Kim}
\address{Department of Mathematics, Pusan National University, Busan, South Korea}
\email{yoosik@pusan.ac.kr}

\author{Siu-Cheong Lau}
\address{Department of Mathematics and Statistics, Boston University, 111 Cummington Mall, Boston MA 02215, USA}
\email{lau@math.bu.edu}

\author{Xiao Zheng}
\address{The Institute of Mathematical Sciences, The Chinese university of Hong Kong, Shatin, Hong Kong}
\email{xiaozh259@gmail.com}
	
	\begin{abstract}
		We develop a $G$-equivariant Lagrangian Floer theory and obtain a curved $A_\infty$ algebra, and in particular a $G$-equivariant disc potential.  We construct a Morse model, which counts pearly trees in the Borel construction $L_G$.  When applied to a smooth moment map fiber of a semi-Fano toric manifold, our construction recovers the $T$-equivariant toric Landau-Ginzburg mirror of Givental. We also study the $\bS^1$-equivariant Floer theory of a typical singular fiber of a Lagrangian torus fibration (i.e. a pinched torus) and compute its $\bS^1$-equivariant disc potential via the gluing technique developed in \cite{CHL18,HKL}. 
	\end{abstract}
	
	\maketitle
	\tableofcontents
	
	\section{Introduction}
	\label{sec:intro}


In this paper, we develop an equivariant version of the SYZ mirror construction using equivariant Lagrangian Floer theory. For a $G$-inavariant Lagranguan submanifold of a Hamiltonian $G$-manifold $X$, and $T\subset G$ the maximal torus, the $T$-equivariant disc potential gives a holomorphic map from the mirror $X^{\vee}$ to the complexified dual torus $T^{\vee}_{\C}$. 

Let us consider the well-known Landau-Ginzburg mirror of a compact Fano toric $d$-fold $X$ \cite{HV,givental_icm, Givental,LLY3}. It is given by a pair $((\C^{\times})^d,W)$, where the superpotential $W \colon (\C^{\times})^d\to\C$ is a holomorphic function of the form
	\[
	W(z)=\sum_{i=1}^m q^{\alpha_i}\vec{z}^{v_i},
	\]        
	where $v_1,\ldots,v_m\in\bm{N}\cong \Z^d$ are primitive generators of the one-dimensional cones of the fan $\Sigma$ defining $X$, $\vec{z}^{v_i}$ denote the corresponding Laurent monomials, and $q^{\alpha_1},\ldots,q^{\alpha_m}$ are Kähler parameters associated to effective curves classes 
	$\alpha_i\in H_2(X;\Z)$. 
	
In \cite{Givental},  Giventa showed that the equivariant quantum cohomology $QH_{T}^{\bullet}(X)$ of $X$ is mirror to the \emph{equivariant superpotential}
	\[
	W_{\lambda}(z)=W(z)+\sum_{i=1}^d \lambda_i \log{z_i},
	\]   
	where $\lambda_1,\ldots,\lambda_d$ are the equivariant parameters for the torus $T$ acting on $X$. Iritani \cite{IR17,Iritani-big} further generalized this to a mirror correspondence between big equivariant quantum cohomology and a universal Landau-Ginzburg potential which has both equivariant and bulk-deformation parameters. Cho-Oh \cite{CO} (in the Fano case) and Fukaya-Oh-Ohta-Ono \cite{FOOO-T} (in the general compact non-Fano setting) used Lagrangian Floer theory to define the disc potential of a smooth fiber of the toric moment map and showed that it coincides with the Landau-Ginzburg superpotential $W$. From this perspective, $W$ is viewed as the generating function of open Gromov-Witten invariants. Each term in $W$ corresponds to a holomorphic disc of Maslov index two bounded by the torus fiber. This gives an SYZ interpretation of the Landau-Ginzburg mirrors \cite{Auroux07, CL}. 
	
	It is a natural and interesting question that whether the $W_{\lambda}$ can also be interpreted as a generating function of equivariant disc counts, to which we give an affirmative answer to in this paper.
	
	\begin{theorem}[Corolloary \ref{cor:tor}] \label{thm:main}
		Let $X$ be a semi-projective and semi-Fano toric manifold with $\dim_{\C} X=d$, and let $T\subset (\bS^1)^d$ be the subtorus determined by the integral basis $u_1,\ldots,u_{\ell} \in \bm{N}\cong \Z^d$. The $T$-equivariant disc potential (\ref{def:G_disc_potential}) $W^L_{T}$ of a regular fiber $L$ of the toric moment map is given by
		\[
		W^L_{T}= \sum_{i=1}^m  \exp(g_i(\check{q}(q))) \bT^{\omega(\beta_i)} \exp (v_i \cdot (x_1,\ldots,x_d))+\sum_{j=1}^{\ell} (u_j \cdot (x_1,\ldots,x_d)) \lambda_j
		\]
		where $\beta_i$ are the basic disc classes bounded by the toric fiber, $g_i(\check{q}(q))$ is given by the inverse mirror map in equation \eqref{eqn:funcn_g}, $q^{\alpha} = \bT^{\omega(\alpha)}$ are the K\"ahler parameters, and $\bT$ is the formal Novikov variable.
	\end{theorem}

	In the Fano case, we have $g_i=0$. By taking ${\ell}=d$ and $u_1,\ldots,u_d$ to be the standard basis of $\mathfrak{t}^d:=\mathrm{Lie}(T)$, and setting $x_i = \log z_i$ for $i=1,\ldots,d$, the above expression for $W^L_{T}$ equals to $W_{\lambda}$.  
	
	
	 We also expect the $T$-equivariant disc potential to be useful in the study of mirror symmetry in $3$ dimensions. Teleman \cite{Tel} proposed that pure topological gauge theory in $3$ dimensions for a compact Lie group $G$ is equivalent the Rozansky-Witten theory for the holomorphic symplectic variety $BFM(G_{\C}^{\vee})$ \cite{BFM}. In this conjectural equivalence, one associates to each compact Hamiltonian $G$-manifolds $X$, a holomorphic family $X^{\vee}$ over a holomorphic Lagrangian $\mathbb{L}_G(X)$ in $BFM(G_{\C}^{\vee})$. Here $X^{\vee}$ is a complex manifold mirror to $X$ in the sense of homological mirror symmetry.  The recipe in \cite{Tel} for constructing $\mathbb{L}_G(X)$ is to realize the equivariant quantum cohomology $QH_G^{\bullet}(X)$ as a $H^G_*(\Omega G)$-module via an equivariant moment correspondence \footnote{In a recent paper \cite{GMP22}, Gonz\'{a}lez, Mak and Pomerleano realized the $H^G_*(\Omega G)$-module structure of $QH_G^{\bullet}(X)$ using the theory of shift operators, in the monotone setting, bypassing the use of Lagrangian correspondence.}, and the latter can be identified with the coordinate ring of $BFM(G_{\C}^{\vee})$. The holomorphic Lagrangian $\mathbb{L}_G(X)$ is given by the support of the module $QH_G^{\bullet}(X)$ (see \cite[Remark~2.1]{teleman18} for the Lagrangian property of $QH_G^{\bullet}(X)$).
	 
	 The expected relation between $T$-equivariant disc potential and the holomorphic Lagrangian $\mathbb{L}_G(X)$ is the following. Let $L$ be a $G$-invariant Lagrangian, and let $T\subset G$ be the maximal subtorus. Assuming that $L$ is weakly unobstructed and has nonnegative minimal Maslov index, it follows from Proposition \ref{prop:gen-m0} that
	 the $T$-equivariant disc potential for $L$ takes the form 
	\[
	W^L_{T}(b)= W(b) + \sum_{i=1}^{\ell} h_i(b) \lambda_i, 
	\]
	for degree one weak bounding cochains $b$. The equivariant term defines a holomorphic map $\rho \colon \mathcal{MC}(L)\to  T_{\C}^{\vee}$, 
	\[
	\rho(b):=(e^{h_1(b)},\ldots,e^{h_{\ell}(b)})
	\] 
	from the weak Maurer-Cartan space $\mathcal{MC}(L)$ of $L$ to $T_{\C}^{\vee}$ (assuming the base field is $\C$ and the functions $h_i$ are convergent over $\C$). Following \cite{teleman18}, we can “push down” the non-equivariant term $W$ to a multi-valued function $\psi$ on $T_{\C}^{\vee}$ whose multi-values are the critical values of $W$ along the fibers of $\rho$. The graph of the differential $d\psi$ defines a holomorphic Lagrangian $\mathbb{L}_{T}(X)$ in $BFM(T_{\C}^{\vee})= T^*T^{\vee}_{\C}$. The BFM space $BFM(G_{\C}^{\vee})$ an affine resolution of singularities of the quotient $T^*T^{\vee}_{\C}/\bm{W}$, where $\bm{W}$ denotes the Weyl group. It is natural to expect that $\mathbb{L}_{T}(X)$ is $\bm{W}$-invariant and can be lifted to the holomorphic Lagrangian $\mathbb{L}_{G}(X)$ in $BFM(G_{\C}^{\vee})$\footnote{We learned the idea of constructing $\mathbb{L}_{G}(X)$ using only the action of the maximal torus from Naichung Conan Leung.}.
	
	
	
	\begin{remark}
		The terms $h_i(b) \lambda_i$ can be understood as obstruction terms for a homotopy equivalence between the equivariant Floer theory of $L \subset X$ and Floer theory of the quotient $L/T \subset X\twobar_\lambda T^{\ell}$.  Daemi-Fukaya \cite{DF17} asserted that this homotopy equivalence can be realized by a Lagrangian correspondence.  In our setting, we can compute the mirror equations $h_i(b)=0$ explicitly.
		
		The mirror map $c(\lambda)$ is crucial to precisely identify which fiber in the mirror family corresponds to $X\twobar_\lambda T^{\ell}$.  By the beautiful work of Woodward-Xu \cite{WX}, the mirror map can be understood as the change from the gauged Floer theory of $X$ to the Fukaya category $\Fuk(X\twobar_\lambda T^{\ell})$ downstairs.  Gauged Floer theory is formulated in terms of vortex equations.  It would be very interesting to investigate the relation with our formulation.
	\end{remark}
	
	Equivariant Lagrangian Floer theory has seen substantial recent developments.  
	In the exact setting and $G=\Z/2\Z$, Seidel-Smith \cite{SS10} provided an approach to understanding $G$-equivariant Floer theory by combining Lagrangian Floer theory and family Morse theory \cite{hutchings08} on $EG\to BG$.  Viterbo \cite{viterbo} illustrated some related ideas for $G=\bS^1$.
	Hendricks-Lipshitz-Sarkar \cite{HLS16a,HLS16b} developed a homotopy coherent method to build up a $G$-equivariant Floer theory. Daemi-Fukaya \cite{DF17} used G-equivariant Kuranishi structure \cite{G-Kuranishi} to tackle the G-equivariant transversality problem, and make a formulation using differential forms.  Bao-Honda \cite{BH18} defined equivariant Lagrangian Floer cohomology for finite group action via semi-global Kuranishi structures. Mirror Lagrangian objects of equivariant vector bundles was studied by Lekili-Pascaleff \cite{Lekili-Pascaleff}. There are many interesting works related to this subject such as \cite{Seidel-lect}. 
	
	In Section \ref{sec:Morse}, we construct a Morse model for $G$-equivariant Lagrangian Floer theory by counting pearly trees  \cite{BC-pearl,FOOO-can} in the Borel construction $L_G$, and study the associated equivariant disc potential.  The key ingredient is a collection of Morse models on the finite dimensional approximations of $L_G$ satisfying certain compatibility conditions (Proposition \ref{prop:extension}).
	This is closest to the approach of Seidel-Smith, although the Lagrangians under our consideration are not exact. The Lagrangians of interest in this paper bound non-constant pseudo-holomorphic discs of non-negative Maslov index. These are known as quantum corrections in mirror symmetry, which can have non-trivial contributions to the equivariant parameters in the disc potential. Since there are only finitely many generators for each finite dimensional approximation of $L_G$, it better serves for computations and for the mirror construction.

	In order for the equivariant paramaeters appearing in $W_{\lambda}$ to manifest. In Section \ref{sec:partial_units}  we adapt the homotopy unit construction developed by Fukaya-Oh-Ohta-Ono \cite[Chapter 7]{FOOO} to construct a curved $A_{\infty}$-algebra 
	\[
	 \left(C_{\textrm{Morse}}^{\bullet}(L;\Lambda_0)\otimes_{\Lambda_0} H^{\bullet}_G(\pt;\Lambda_0), \{\fm^G_k\}_{k\ge 0} \right),
	\] 
	such that the equivariant paramters are partial units (Definition \ref{def:partial_unit}) and consequently the $A_{\infty}$-maps $\fm^G_k$ are $H^{\bullet}_G(\pt;\Lambda_0)$ multi-linear  (Theorem \ref{thm:pull-out-lambda}), 
	under the assumption that $L$ has minimal Maslov index zero and $G$ is a product of unitary groups. 
	Here $\Lambda_0$ is the Novikov ring
	\[
	\Lambda_0=\left\{\sum_{i=0}^{\infty} a_iT^{A_i}|a_i\in \mathrm{k}, A_i\in\R_{\ge 0}; A_i \text{ increases to } +\infty,\right\}, 
	\]	
	where $\mathrm{k}$ is a field of charateristic $0$. This allows us to view $H^{\bullet}_G(\pt;\Lambda_0)$ as a graded coefficient ring. We note that the terms of the form $\fm^G_k(x_1\otimes 1,\ldots x_k\otimes 1)$ are series in the equivariant parameters and receive non-trivial contributions from pearly trees in $L_G$.

	In Section \ref{sec:imm}, we study the $\bS^1$-equivariant Floer theory of the immersed two-sphere with a single nodal point.  This is also known as the pinched two-torus and is the most common singular fiber appearing in an Lagrangian torus fibration. Even although it does not bound any non-constant holomorphic discs of Maslov index zero, it still has a non-trivial equivariant disc potential from the contribution of constant polygons at the nodal point. As the corresponding moduli spaces have non-trivial obstructions, the gluing technique via the isomorphism in the Fukaya category between smooth and pinched tori in \cite{HKL} is crucial in the computation of the explicit expression of the equivariant disc potential.  

	\begin{theorem}[Theorem \ref{thm:equiv-imm}] \label{thm:intro-equiv-imm}
		The $\bS^1$-equivariant disc potential of the immersed sphere $L_0$ is
		\[
		W^{L_0}_{\bS^1} = \log (1-uv) = -\sum_{j=1}^\infty \frac{(uv)^j}{j}
		\]
		where $(u,v) \in \Lambda_0^2$, $\val(uv)>0$ are the formal deformation parameters corresponding to the degree one immersed generators of $L_0$.
	\end{theorem}

	Floer theory for immersed Lagrangians was developed by Akaho-Joyce \cite{AJ}, which is in line with the theory of Fukaya-Oh-Ohta-Ono \cite{FOOO} for smooth Lagrangians.  In \cite{DET}, a different method via Legendrian topology is used to develop the Floer theory of immersed Lagrangian surfaces.
	
	In a subsequent work \cite{HKLZ19}, we compute the equivariant disc potential for Lagrangian immersions in toric Calabi-Yau manifolds.  They have very interesting expressions in relation with the Gross-Siebert slab functions \cite{GS07} and mirror maps of toric Calabi-Yau manifolds \cite{CLL,CCLT2}.  It is closely related to the topological vertex \cite{Katz-Liu,GZ,Fang-Liu,fang-liu-tseng,FLZ}.

	
	\addtocontents{toc}{\protect\setcounter{tocdepth}{1}}
	\subsection*{Acknowledgments}
	The authors would like to thank deeply anonymous referees for detailed and helpful comments.  
	The first and second named authors express their deep gratitude to Cheol-Hyun Cho for explaining his point of view of equivariant Floer theory using Cartan model.  The first named author was supported by the National Research Foundation of Korea (NRF) grant funded by the Korea government NRF-2021R1F1A1057739 and NRF-2020R1A5A1016126. The second named author also thanks Ben Webster and Kevin Costello for the discussion on Teleman's conjecture at Perimeter Institute.  His work is supported by Simons Collaboration Grant \#580648.
	\addtocontents{toc}{\protect\setcounter{tocdepth}{2}}
	
\section{A Morse model for Lagrangian Floer theory} \label{sec:MorseLag} 
	
	\subsection{The singular chain model}
	\label{sec:simplicial} 
 
Let $(X,\omega_X)$ be a symplectic manifold of real dimension $2d$. We assume that $X$ is convex at infinity or geometrically bounded if it is non-compact. Choose a compatible almost complex structure $J_X$. For a closed, connected, and relatively spin Lagrangian submanifold $L$, Fukaya-Oh-Ohta-Ono constructed a countably generated subcomplex $C^{\bullet}(L;\Lambda_0)$ of the singular chain complex $S^{\bullet}(L;\Lambda_0)$ (regarded as a cochain complex) and an $A_{\infty}$-algebra $(C^{\bullet}(L;\Lambda_0),\tilde{\fm})$.
We briefly recall their construction below, which will be used in constructing the equivariant Morse model in Section~\ref{sec:Morse}. We refer the reader to \cite{FOOO} for more details. 

For a relative homology class $\beta \in H_2(X,L;\Z)$, we denote by $\mathcal{M}(\beta;L)$ the moduli space of $J_X$-holomorphic stable maps from a bordered Riemann surface of genus $0$ representing $\beta$ and by $\mathcal{M}_{k+1}(\beta;L)$ the moduli space together with $k+1$ boundary marked points $z_0,\ldots,z_k$ ordered counter-clockwise. The moduli space $\mathcal{M}_{k+1}(\beta;L)$ has virtual dimension $\mu(\beta)+k-2+d$, where $\mu(\beta)$ is the Maslov index of $\beta$.  It carries the evaluation map 
\[
\ev \colon \mathcal{M}_{k+1}(\beta;L)\to L^k, \quad u \mapsto (u(z_1), \dots, u(z_k)).
\] 
Let $H_2^{\mathrm{eff}}(X,L)$ denote the monoid of effective classes in $H_2(X,L;\Z)$, represented by $J_X$-holomorphic maps, i.e.
\[
H_2^{\mathrm{eff}}(X,L)= \left\{\beta\in H_2(X,L;\Z) \mid \mathcal{M}(\beta;L)\ne\emptyset \right\}.
\]
For a $k$-tuple $\vec{P}=(P_1,\ldots,P_k)$ of smooth singular chains of $L$, we denote by $\mathcal{M}_{k+1}(\beta;L;\vec{P})$ the fiber product (in the sense of Kuranishi structures) of $\mathcal{M}_{k+1}(\beta;L)$ with $P_1, \dots, P_k$, i.e.,
\[
\mathcal{M}_{k+1}(\beta;L;\vec{P})=\mathcal{M}_{k+1}(\beta;L)\times_{L^k} \prod P_i.
\]

We recall the notion of a \emph{generation} $g \in \mathbb{Z}$, which will be used to determine perturbations of moduli spaces and subsets of singular chains inductively in \cite[Section 7]{FOOO}.
As preliminaries, for $\beta\in H_2^{\mathrm{eff}}(X,L)$, we set
\[
\norm{\beta}=\begin{cases}
\sup \left\{n \mid \exists \beta_1,\ldots,\beta_n\in H_2^{\mathrm{eff}}(X,L)\setminus \{\beta_0\} \text{ with } \sum_{i=1}^n \beta_i=\beta \right\}+\floor{\omega(\beta)}-1 & \text{if }\beta\ne \beta_0 ; \\
-1 & \text{if } \beta=\beta_0 
\end{cases}
\]
where $\beta_0$ is the constant disc class.
Also, we employ a function $\mathfrak{d} \colon \{1,\ldots,k\}\to \Z_{\ge 0}$ to keep track of generations of inputs.
For a pair $(\mathfrak{d},\beta)$ with $\beta\in H_2^{\mathrm{eff}}(X,L)$, we define
\[
\norm{(\mathfrak{d},\beta)}=\begin{cases}
\max_{i\in \{1,\ldots,k\}} \{\mathfrak{d}(i)\}+\norm{\beta}+k & \text{if } k\ne 0;\\
\norm{\beta} & \text{if } k=0.
\end{cases}
\]	

We start with a countable set $\mathcal{X}_{-1}(L)$ of smooth singular chains on $L$ such that
\begin{equation*} \label{eq:-1_gen	}
C_{(-1)}(L;\Lambda_0)= \Lambda_0 \cdot \mathcal{X}_{-1}(L)
\end{equation*}
is a subcomplex of $S^{\bullet}(L;\Lambda_0)$ isomorphic on cohomology.
For $g\ge 0$, we inductively choose a countable set $\mathcal{X}_g(L)$ of smooth singular chains on $L$ and a system of multisections $\fs_{\fd,k,\beta,\vec{P}}$ for $\mathcal{M}^{\fd}_{k+1}(\beta;L;\vec{P})$ satisfying $\norm{(\mathfrak{d},\beta)}=g$. Here the superscript $\fd$ is written to emphasize the generations of the inputs $P_i\in \mathcal{X}_{\fd(i)}(L)$.
At each inductive step, new multisections $\fs_{\fd,k,\beta,\vec{P}}$ are chosen to be transversal to the zero section and we extend the multisections previously chosen for the boundary strata $\partial\mathcal{M}_{k+1}(\beta;L;\vec{P})$. (To achieve transversality, the Kuranishi structures for the moduli spaces $\cM_{k+1}(\beta;L)$ are chosen to be weakly submersive.) Moreover, the zero locus
\begin{equation}\label{equ_frakks}
\mathcal{M}^{\fd}_{k+1}(\beta;L;\vec{P})^{\fs_{\fd,k,\beta,\vec{P}}}:=\fs_{\fd,k,\beta,\vec{P}}^{-1}(0)
\end{equation}
is triangulated, extending the triangulation on its boundary strata. The new simplices in the triangulation are then regarded as elements of $\mathcal{X}_{g}(L)$.  Additional singular simplices are added to $\mathcal{X}_{g}(L)$ so that 
\begin{equation} \label{eq:C_g}
C_{(g)}(L;\Lambda_0)=\bigoplus_{g'\le g} \Lambda_0 \cdot \mathcal{X}_{g'}(L)
\end{equation}
remains a subcomplex of $S^{\bullet}(L;\Lambda_0)$ isomorphic on cohomology.

Let $C^{\bullet}(L;\Lambda_0)=\lim_{\to} C_{(g)}(L;\Lambda_0)$. The $A_{\infty}$-structure map $\tilde{\fm}_k \colon C^{\bullet}(L;\Lambda_0)^{\otimes k}\to C^{\bullet}(L;\Lambda_0)$ is defined by 
\begin{equation*}
\tilde{\fm}_k(P_1,\ldots,P_k)=\sum_{\beta \in H_2^{\mathrm{eff}}(X,L)} \bT^{\omega_X(\beta)} \tilde{\fm}_{k,\beta}(P_1,\ldots,P_k)
\end{equation*}
where
\begin{equation*}
\tilde{\fm}_{k,\beta}(P_1,\ldots,P_k)=
\begin{cases}
0 &\mbox{for $(k,\beta)\ne (0,\beta_0)$}\\
(-1)^d \partial P &\mbox{for $(k,\beta)\ne (1,\beta_0)$}\\
(\ev_0)_*\left(\mathcal{M}_{k+1}(\beta;L;\vec{P})^{\fs}\right) &\mbox{otherwise.}
\end{cases}	
\end{equation*}
Here, $\partial$ is the coboundary map on $C^{\bullet}(L;\Lambda_0)$. The map
$\tilde{\fm}_{k,\beta}$ is of degree $2-k-\mu(\beta)$.


\begin{remark}
We note that it is only possible to choose finitely many multisections at once while still having them being sufficiently close to the original Kuranishi maps. To deal with this technical difficulty, Fukaya-Oh-Ohta-Ono introduced the notion of $A_{n,K}$-structure associated to $L$ using moduli space of pseudo-holomorphic discs with bounded energy and number of marked points in \cite[Chapter 7]{FOOO}. The $A_{\infty}$-structure $\tilde{\fm}$ on $C^{\bullet}(L;\Lambda_{0})$ is obtained from the $A_{n,K}$-structures via homological techniques. 
\end{remark}

\begin{remark}
In order to define a $\Z$-graded Floer cohomology, the conventional grading of the map $\mathfrak{m}_k$ is $2-k$. This means aside from the case of graded Lagrangian submanifolds, one has to define the Novikov ring $\Lambda_0$ using an extra grading parameter in order to compensate for the Maslov indices. In this paper, we work on the chain level and do not follow this convention. The grading of $\tilde{\mathfrak{m}}_{k,\beta}$ is crucial in understanding the vanishing of certain terms when doing computations with the Morse model.
\end{remark}

The singular chain model $(C^{\bullet}(L;\Lambda_0),\tilde{\fm})$ constructed in \cite{FOOO} does not have a strict unit in general. It was shown that the fundamental cycle $\bm{e}$ of $L$ is a \emph{homotopy unit}.  We briefly describe key properties and ideas below, and refer the readers to \cite[Chapter 3.3]{FOOO} for the precise definition of a homotopy unit and to \cite[Chapter 7.3]{FOOO} for details of the homotopy unit construction.

The constructed $A_\infty$-algebra $(C^{\bullet}(L;\Lambda_0),\tilde{\fm})$ will be enlarged to a \emph{unital} $A_{\infty}$-algebra $(C^{\bullet}(L;\Lambda_0)^+,\tilde{\fm}^+)$ that is homotopy equivalent to $C^{\bullet}(L;\Lambda_0)$.
At first, we enlarge the chain complex $C^{\bullet}(L;\Lambda_0)$ by adding a generator $\bm{e}^+$ of degree $0$ serving as the strict unit and a generator $\bm{f}$ of degree $(-1)$ serving as a homotopy between $\bm{e}$ and $\bm{e}^+$, that is,
\begin{equation}\label{equ_cLlambda0+}
C^{\bullet}(L;\Lambda_0)^+=C^{\bullet}(L;\Lambda_0)\oplus \Lambda_0 \cdot\bm{e}^+ \oplus \Lambda_0 \cdot\bm{f}.
\end{equation}
Next, the $A_{\infty}$-structure maps $\tilde{\fm}^+ = \{ \tilde{\fm}^+_k \}_{k\ge 0}$ are defined with the following properties$\colon$
\begin{enumerate}
\item The restriction of $\tilde{\fm}^+$ to $C^{\bullet}(L;\Lambda_0)$ agrees with $\tilde{\fm}$,
\item $\bm{e}^+$ is the strict unit, i.e.
\begin{enumerate}[label=(\roman*)]
\item $\tilde{\fm}^+_2(\bm{e}^+,x)=(-1)^{|x|}\tilde{\fm}^+_2(x,\bm{e}^+)=x$ for $x\in C^{\bullet}(L;\Lambda_0)^+$,
\item $\tilde{\fm}^+_k(\ldots,\bm{e}^+,\ldots)=0$ for $k\ne 2$.
\end{enumerate}
\item $\bm{f}$ is a homotopy between $\bm{e}$ and $\bm{e}^+$.
\end{enumerate}


Later on, we will define a disc potential (Definition~\ref{def:disc_potential}), which is to be of the most interest to us in the manuscript, after passing to the enlarged $A_\infty$-algebra with a strict unit. 
To compute such a disc potential from the original singular chain model (before the enlargement), we need to examine the homotopy $\bm{f}$ between $\bm{e}$ and $\bm{e}^+$.  

A construction and properties of the homotopy $\bm{f}$ are in order. 
Let $\vec{a}= (a_1,\ldots,a_{|\vec{a}|})$ be an ordered subset of $\{1,\ldots,k\}$ satisfying $a_1< \ldots < a_{|\vec{a}|}$. 
For $\vec{P}=(P_1,\ldots,P_{k-|\vec{a}|})$ with $P_i\in\mathcal{X}_{\fd(i)}(L)$, let $\vec{P}^{+}$ be the $k$-tuple obtained by inserting $\bm{e}$ into the $\vec{a}$-th places of $\vec{P}$. We  assume $L\in \mathcal{X}_{0}(L)$. We set $\fd^{+}(i)=g_i^{+}$ if $\vec{P}^{+}=(P_1^{+},\ldots,P_k^{+})$ where $P^{+}_i\in\mathcal{X}_{g^{+}_i}(L)$.
Consider the forgetful map
\begin{equation}
\label{forget}
\forget_{\vec{a}} \colon \mathcal{M}_{k+1}(\beta;L;\vec{P}^{+})\to \mathcal{M}_{k+1-|\vec{a}|}(\beta;L;\vec{P}),
\end{equation}
which forgets the $\vec{a}$-th marked points and then stabilizes.

We choose a perturbation on 
\begin{equation} \label{eq:homotopy_moduli}
[0,1]^{|\vec{a}|}\times \cM_{k+1}(\beta;L;\vec{P}^{+})
\end{equation}
as follows: 
For a splitting $\vec{a}^1\coprod \vec{a}^2=\vec{a}$ of $\vec{a}$, we denote by $\vec{P}'$ the $(k-|\vec{a}^1|)$-tuple given by removing $\bm{e}$ from the $\vec{a}^1$-th places of $\vec{P}^{+}$.  Let $(t_1,\ldots,t_{|\vec{a}|})$ be coordinates on $[0,1]^{|\vec{a}|}$. If $t_i=0$ for all $i\in\vec{a}^1$, we take the perturbation $\fs_{\fd^{+},k,\beta,\vec{P}^{+}}$ transversal to the zero section obtained by just inserting $\bm{e}$ into the $a_1,\ldots,a_{|\vec{a}^1|}$-th place. If $t_i=1$ for all $i\in\vec{a}^1$, we take the perturbation pulled back via the map $\forget_{\vec{a}^1} \colon \mathcal{M}_{k+1}(\beta;L;\vec{P}^{+})\to \mathcal{M}_{k+1-|\vec{a}^1|}(\beta;L;\vec{P}')$ in~\eqref{forget}. 
Finally, we take a perturbation $\fs_{\fd^{+},k,\beta,\vec{P}^{+}}^+$ on $[0,1]^{|\vec{a}|}\times \cM_{k+1}(\beta;L;\vec{P}^{+})$ transversal to the zero section interpolating between them. 

The maps $\tilde{\fm}^+_{k,\beta}$ for $(k,\beta)(\ne (1,\beta_0))$ with inputs $\bm{f}$ inserted into $\vec{a}$-th place of $\vec{P}$ are defined by
\begin{equation}
\label{m+}
(\ev_0)_*\left(\left([0,1]^{|\vec{a}|}\times \cM_{k+1}(\beta;L;\vec{P}^{+})\right)^{\fs_{\fd^{+},k,\beta,\vec{P}^{+}}^+}/\sim\right)	
\end{equation}
where $\sim$ is the equivalence relation collapsing fibers of $\forget_{\vec{a}}$, see \cite[Definition 7.3.28]{FOOO}. Notice that the zero locus of a pullback multisection is a degenerate singular chain which becomes zero in the quotient.

We also set $\tilde{\fm}^+_{1,\beta_0}(\bm{f})=\bm{e}^+-\bm{e}$
and therefore
\begin{equation} \label{eq:m_1_f}
\tilde{\fm}^+_1(\bm{f})=\bm{e}^+-\bm{e}+\bm{h},
\end{equation}
where
\begin{equation*}
\label{eq:fh}
\bm{h}=\sum_{\beta\ne \beta_0} \bm{T}^{\omega_X(\beta)} (\ev_0)_* \left(([0,1]\times \cM_2(\beta;L;\bm{e}))^{\fs^+}/\sim \right).
\end{equation*}
Note that $\bm{h} \equiv 0 \mod \Lambda_+ \cdot C^{\bullet}(L;\Lambda_0)$ and the terms 
\begin{equation}\label{equ_obdeg}
\sum_{\mu(\beta)\ge 2} (\ev_0)_* \left(([0,1]\times \cM_2(\beta;L;\bm{e}))^{\fs^+}/\sim \right)
\end{equation}
are of degrees at most $-2$. 

\subsection{Morse homology $=$ singular homology: an alternative approach}
\label{sec:SingMorse model}
Pearl complex was employed by Biran and Cornea \cite{BC-pearl, BC-survey} to construct Lagrangian Floer cohomology of monotone Lagrangians (a similar complex also appeared in an earlier work of Oh \cite{Oh96R}).  It has many important applications, including the proof of homological mirror symmetry for Fermat Calabi-Yau hypersurfaces by Sheridan \cite{Sheridan-CY}.

A Morse model of Lagrangian Floer theory was constructed in \cite{FOOO-can} based on their singular chain model.  
We follow their construction to construct a Morse model in Section~\ref{sec:Morse model}, with a modification that we add certain degenerate chains as summands in the definition of unstable chains $\Delta_p$ of critical points $p$ and the chains of forward orbits $F(\tau)$ of singular chains $\tau$. 
The main reason for the modification is that there are unwanted degenerate chains appeared in $\partial \Delta_p$ and $\partial F(\tau)$, and we `contract them away' by adding cones over the unwanted terms.
This modification enables us to realize a Morse complex as a singular chain complex, which will be explained in this section. 


Let $f \colon L\to\R$ be a Morse function on a Lagrangian submanifold $L$ and $V$ a negative pseudo-gradient vector field for $f$, i.e. $df(V)|_{p}\le 0$ and the equality holds if and only if  $p\in\mathrm{Crit}(f)$. For each critical point $p\in\mathrm{Crit}(f)$, the vector field $V$ coincides with the negative gradient vector field for the Euclidean metric on a Morse chart of $p$. The flow of $V$ is denoted by $\Phi_t$. We denote by $W^s(f;p)$ and $W^u(f;p)$ the stable manifold of $p$ and unstable manifold of $p$, respectively. Namely, 
\begin{equation*}
W^s(f;p)=\left\{q\in L \mid \lim_{t\to +\infty}\Phi_t(q)=p\right\} \, \mbox{  and  }  \, W^u(f;p)=\left\{q\in L \mid \lim_{t\to -\infty}\Phi_t(q)=p\right\}.
\end{equation*}
The \emph{degree} of $p$ is denoted by $|p|$ and defined by
\[
 |p|:=d-\mathrm{ind}(p),
\] 
where $\mathrm{ind}(p)$ is the Morse index of $p$. Then
$|p|=\dim{W^s(f;p)}=\mathrm{codim } \, W^u(f;p)$.

We assume that $V$ satisfies the Smale condition and we call such a pair $(f,V)$ \emph{Morse-Smale}. For two distinct points $p,q\in\Crit(f)$, let $\mathcal{M}(p,q)$ be the moduli space of flow lines from $p$ to $q$. The Smale condition implies that the virtual dimension of $\mathcal{M}(p,q)$ is 
\[
\dim \mathcal{M}(p,q)=|q|-|p|-1.
\]
The moduli space $\mathcal{M}(p,q)$ and the unstable manifold $W^u(f;p)$ can be respectively compactified into smooth manifolds  $\overline{\mathcal{M}}(p,q)$ and $\overline{W}^u(f;p)$ with corners. The strata of codimension $k$ are denoted by
$\overline{\mathcal{M}}(p,q)_k$ and $\overline{W}^u(f;p)_k$, respectively. The strata consist of broken flow lines of the following form$\colon$
\begin{equation*}
\overline{\mathcal{M}}(p,q)_k=\bigcup_{|p|<|r_1|<\cdots<|r_k|<|q|} \mathcal{M}(p,r_1)\times\cdots\times \mathcal{M}(r_{k},q)
\end{equation*}
and
\begin{equation*}
\overline{W}^u(f;p)_k=\bigcup_{|p|<|r_1|<\cdots<|r_k|} \mathcal{M}(p,r_1)\times\cdots\times \mathcal{M}(r_{k-1},r_k)\times W^u(f;r_k).
\end{equation*}

Let $C^{\bullet}(f;\Lambda_0)$ be the cochain complex 
\[
C^{\bullet}(f;\Lambda_0)=\bigoplus_{p\in\mathrm{Crit}(f)} \Lambda_0 \cdot p,
\] 
with grading given by $|p|$. The Morse differential $\delta \colon C^{\bullet}(f;\Lambda_0)\to C^{\bullet+1}(f;\Lambda_0)$ is defined by counting gradient flow lines:
\[
\delta p= \sum_{|q|=|p|+1} \sharp \mathcal{M}(p,q)\cdot q.
\]

In \cite{FOOO-can}, a (non-unital) $A_{\infty}$-algebra structure was constructed on $C^{\bullet}(f;\Lambda_0)$ assuming $L$ has a triangulation whose simplices are the closures of $W^u(f;p)$. To establish an isomorphism between Morse and singular cohomology, we need to associate to each critical point $p$ of degree $|p|$ a singular chain. We can do so by choosing a triangulation on $\overline{W}^u(f;p)$ and regard it as a singular chain $d_p$. In general, $\overline{W}^u(f;p)$ has boundary components of the form $\overline{\mathcal{M}}(p,r)\times \overline{W}^u(f;r)$ with $|r|\ge |p|+2$. If $\dim \overline{\mathcal{M}}(p,r)\ge 1$ and the image of $\overline{\mathcal{M}}(p,r)\times \overline{W}^u(f;r)$ in $L$ is supported on $\overline{W}^u(f;r)$, the facet of $d_p$ representing $\overline{\mathcal{M}}(p,r)\times \overline{W}^u(f;r)$ is a \emph{degenerate} chain on $L$. Thus, for the assignment $p\mapsto d_p$ to be a cochain map, one should mod out the  degenerate chains. On the other hand, $\mathcal{X}_g(L)$ are constructed including degenerate chains since the $A_{\infty}$ products of degenerate chains may no longer be a degenerate chain. 

\begin{figure}[h]
	\begin{center}
		\includegraphics[scale=0.4]{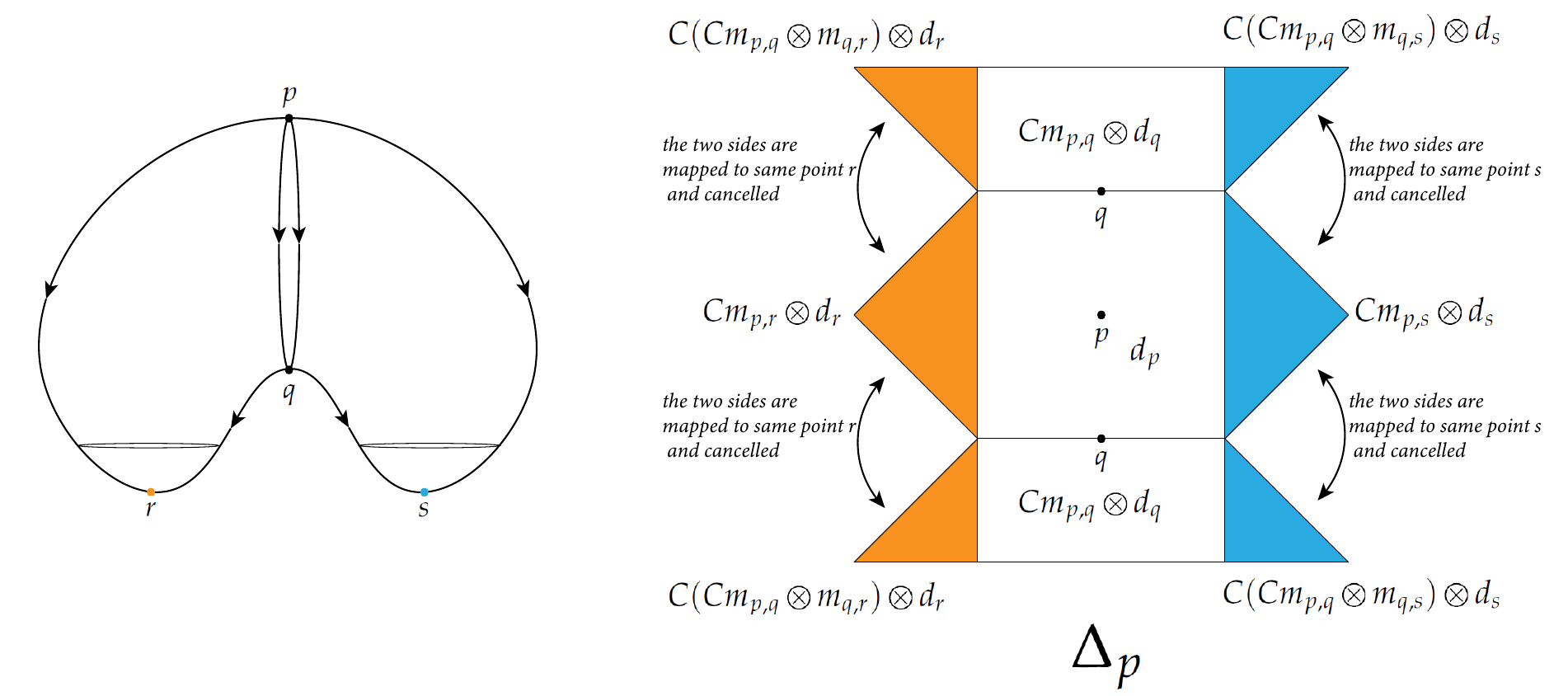}
		\caption{On the left is a Morse function on $S^2$. On the right is unstable chain $\Delta_p$ assigned to its maximal point $p$.}
		\label{fig:deg-chain}
	\end{center}
\end{figure}


We circumvent this difficulty by systematically adding certain degenerate chains to the fundamental chains of $\overline{W}^u(f;p)$ and then define a map from $C^{\bullet}(f;\Z)$ to $S^{\bullet}(L;\Z)$ accordingly in a way that it is a cochain map which induces an isomorphism on cohomology. (See Figure \ref{fig:deg-chain} for an illustration of the added degenerate chains.) 

We triangulate $\overline{W}^u(f;p)$ and $\overline{\mathcal{M}}(p,q)$ to regard them as smooth singular chains, which are denoted by $d_p$ and $m_{p,q}$, respectively. By choosing the triangulations inductively on $|p|$ and $|q|-|p|$ and by extending triangulation from the boundary strata, we can assume 

\begin{equation} \label{eq:sign_1}
\partial d_p=\sum_{|q|>|p|} (-1)^{|q|-|p|-1} m_{p,q}\otimes d_q, 
\end{equation}
and
\begin{equation} \label{eq:sign_2}
\partial m_{p,q}=\sum_{|q|>|r|>|p|} (-1)^{|r|-|p|-1} m_{p,r}\otimes m_{r,q}.
\end{equation}
Here the tensor product $A\otimes B$ of two simplicial complexes $A$ and $B$ denotes the following triangulation on the direct product $A\times B$: We choose linear orders on vertices of $A$ and $B$ (compatible with their orientations). Then there exists a unique triangulation $A\otimes B$ of $A\times B$ such that the vertices of $A\otimes B$ are pairs $(x,y)$, where $x$ is a vertex of $A$ and $y$ is a vertex of $B$, and an $n$-simplex in $A\otimes B$ is defined by a set $\{(x_0,y_0),\ldots,(x_n,y_n)\}$ such that $x_0\le\ldots\le x_n$, $y_0\le\ldots \le y_n$, $\{x_0,\ldots,x_n\}$ defines a simplex $\tau_1$ on $A$ and $\{y_0,\ldots,y_n\}$ defines a simplex $\tau_2$ on $B$ with $\dim \tau_1+ \dim \tau_2\ge n$. Notice that this is the standard product structure in the category of simplicial sets. 

We note that the linear map given by the assignment $p\mapsto d_p$ is \emph{not} a cochain map because of the presence of the boundary strata of the form $m_{p,q}\times d_q$ for $q$ with $|q|>|p|+1$.
To kill those error terms, we shall modify $d_p$. For an $n$-simplex $\tau$, let $C\tau$ be the simplicial cone over $\tau$. For $n\ge 1$, the cone $C\tau$ has vertices
\[
\{\text{vertices of } C\tau\}=\{\text{vertices of } \tau\}\medcup \{*\},
\]  
and simplices
\[
\begin{split}
\{n \text{-simplices of } C\tau \} &=\{n \text{-simplices of } \tau \} \\ &\medcup \{\text{cones of } (n-1) \text{-simplices of } \tau \text{ with the vertex } \{* \}\}.
\end{split}
\]
We orient $C\tau$ in such a way that 
\begin{equation} \label{eq:boundary_cone}
\partial (C\tau) =(-1)^{\dim \tau+1} \tau + C(\partial \tau),  \qquad \text{ for } \dim \tau\ge 1.  
\end{equation}
For a singular chain $\tau\to L$, the chain  $C\tau \to L$ is defined by precomposing $\tau\to L$ with the face map $C\tau\to \tau$ obtained by omitting the vertex $\{*\}$. When $\tau=\{*\}$, $C\tau$ is the $1$-simplex and $\partial (C\tau)$ is the disjoint union of two $0$-simplices equipped with opposite orientations. Moreover, we have $\partial (C\tau)=0$ as a singular chain.

For each component $m_{p,q} \times d_q$ of $\partial d_p$, we define a singular chain $\Delta_{p,q}$ by 
\begin{equation*}
	\Delta_{p,q}= \sum_{\substack{r_1=q \\ |r_1|<\cdots<|r_n|}}   C(\cdots C(C(Cm_{p,r_1}\otimes m_{r_1,r_{2}})\otimes  m_{r_2,r_3} )\otimes \cdots \otimes m_{r_{n-1},r_n}) \otimes d_{r_n}.
\end{equation*}
The chain $\Delta_p$ is then defined by 
\begin{equation}\label{Delta_p}
\Delta_p=d_p+\sum_{|q|>|p|} \Delta_{p,q}.
\end{equation}
Using~\eqref{Delta_p}, we define the map 
\begin{equation} \label{eq:delta_p}
\iota \colon C^{\bullet}(f;\Z)\to S^{\bullet}(L;\Z)
\end{equation}
by extending the assignment $\iota(p)=\Delta_p$ linearly.

\begin{lemma} \label{lemma:chain_map}
The linear map $\iota \colon C^{\bullet}(f;\Z)\to S^{\bullet}(L;\Z)$ is a cochain map, that is, $\iota \circ \delta   = \partial \circ \iota$. 
\end{lemma}
\begin{proof}
By the linearity, it suffices to show that 
\begin{equation}
\partial \Delta_p =\sum_{|r|=|p|+1} m_{p,r}\otimes \Delta_r.
\end{equation}
By iterating \eqref{eq:boundary_cone}, we compute 
\begin{align*}
\partial \Delta_p&=\partial d_p + \sum_{q}  \partial \Delta_{p,q}   \\
&= \sum_{q} (-1)^{|q|-|p|-1} m_{p,q}\otimes d_q \\
&\quad +  \sum_{r_1,\ldots,r_n}  (-1)^{|r_n|-|p|} C(\cdots Cm_{p,r_1}\otimes \cdots \otimes m_{r_{n-2},r_{n-1}})\otimes  m_{r_{n-1},r_n}\otimes d_{r_n} \\
& \quad + \sum_{r_1,\ldots,r_n}  C(\cdots Cm_{p,r_1}\otimes \cdots \otimes m_{r_{n-1},r_n}) \otimes \left(\sum_{r_{n+1}}  (-1)^{|r_{n+1}|-|p|-1} m_{r_n,r_{n+1}}\otimes d_{r_{n+1}}   \right)  \\
&\quad + \sum_{r_1,\ldots,r_n} \sum_{j} (-1)^{ |r_{j}|-|p|}   C(\cdots C(\cdots Cm_{p,r_1}\otimes \cdots \otimes m_{r_{j-1},r_j}\otimes m_{r_j,r_{j+1}})\otimes \cdots  \otimes m_{r_{n-1},r_n})\otimes d_{r_n}\\
&\quad  + \sum_{r_1,\ldots,r_n}\sum_{j} C(\cdots C(\cdots Cm_{p,r_1}\otimes \cdots \otimes \left(\sum_{s_j} (-1)^{|s_j|-|p|-1} m_{r_{j},s_j}\otimes m_{s_j,r_{j+1}}\right) ) \otimes \cdots  \otimes m_{r_{n-1},r_n}) \otimes d_{r_n} \\
&\quad + \sum_{r_1,\ldots,r_n}   C(\cdots  C(\partial(Cm_{p,r_1})\otimes m_{r_1,r_2}) \cdots \otimes m_{r_{n-1},r_n}) \otimes  d_{r_{n}} \\ 
&=\sum_{|q|=|p|+1}     m_{p,q}\otimes d_q  + \sum_{r_1,\ldots,r_n} C(\cdots C\left(\sum_{|q|=|p|+1} m_{p,q}\otimes m_{q,r_1}\right)\otimes \cdots \otimes m_{r_{n-1},r_n})  \otimes d_{r_{n}} \\ 
&= \sum_{|q|=|p|+1}   m_{p,q} \otimes \left( d_q +  \sum_{r_1,\ldots,r_n}    C(\cdots Cm_{q,r_1}\otimes \cdots \otimes m_{r_{n-1},r_n}) \otimes d_{r_{n}} \right) \\ 
&=\sum_{|q|=|p|+1}  m_{p,q} \otimes (d_q+\sum_{r} \Delta_{q,r} )= \sum_{|q|=|p|+1}  m_{p,q} \otimes \Delta_q. 
\end{align*}
We note that the first and second equalities above hold by definition. The third equality follows from the cancellation of terms appearing twice with opposite signs and  the fact that $\partial (Cm_{p,q})=0$ as a singular chain when $\dim m_{p,q}=0$, see \eqref{eq:sign_1}, \eqref{eq:sign_2} and \eqref{eq:boundary_cone}  for the sign rules. The fourth equality holds because $C(\{*\}\otimes \tau)=\{*\}\otimes C\tau$.
\end{proof}

In order to show that \eqref{eq:delta_p} is a quasi-isomorphism, we introduce some terminologies.
We call a simplex $\tau \to L$ \emph{generic} if each face of $\tau$ is transverse to the stable submanifolds of all the critical points of $f$. We denote by $S^{\bullet}_{gen}(L;\Z)$ the subcomplex of $S^{\bullet}(L;\Z)$ generated by generic simplices. It is well-known that the inclusion $S^{\bullet}_{gen}(L;\Z)\into S^{\bullet}(L;\Z)$ is a quasi-isomorphism. 

For a generic simplex $\tau$, we denote by $\mathcal{M}(\tau,q)$ the moduli space of flow lines from $\tau$ to a critical point $q$. Its dimension is equal to $\dim \mathcal{M}(\tau,q)=\dim \tau-d+|q|$.  
Let $\tau(j)$ denote the codimension $j$ strata of $\tau$. The moduli space $\mathcal{M}(\tau,q)$ has a natural compactification to a smooth manifold with corners $\overline{\mathcal{M}}(\tau,q)$. Its strata of codimension $k$ are
\begin{equation} \label{eq:compactification}
\overline{\mathcal{M}}(\tau,q)_k=\bigcup_{j=0}^k \bigcup_{|p_1|<\cdots <|p_j|<|q|} \cM(\tau(k-j),p_1)\times \cM(p_1,p_2)\times\ldots\times \cM(p_{j-1},p_j)\times \cM(p_j,q).
\end{equation}


From the description of the compactification \eqref{eq:compactification}, we obtain the following lemma. 

\begin{lemma} \label{lem:projection_chain_map}
The linear map 
\begin{equation} \label{eq:projection}
\Pi \colon S^{\bullet}_{gen}(L;\Z) \to C^{\bullet}(f;\Z)
\end{equation}
given by
\begin{equation*}
\Pi(\tau)=\sum_{\dim \tau=d-|q|} \sharp \mathcal{M}(\tau,q)\cdot q.
\end{equation*}
is a cochain map, that is, $\delta\circ \Pi = \Pi\circ \partial$.
\end{lemma}

Now, we define the \emph{forward orbit} $\mathscr{F}(\tau)$ of a generic simplex $\tau$ as the set
$
\mathscr{F}(\tau):=[0,\infty)\times \tau
$
with the map  $e \colon \mathscr{F}(\tau)\to L$ defined by $e(t,x)=\Phi_t(\tau(x))$, 
where $\Phi_t$ is the flow of the negative pseudo-gradient vector field $V$. 
The orbit $\mathscr{F}(\tau)$ has a natural compactification to a smooth manifold with corners $\overline{\mathscr{F}}(\tau)$, whose codimension $k$ strata are
\[
\overline{\mathscr{F}}(\tau)_k=\mathscr{F}(\tau(k))\bigcup_{j=1}^k\bigcup_{|r_1|<\cdots<|r_j|<|q|} \cM(\tau(k-j),r_1)\times \cM(r_1,r_2)\times\cdots \times \cM(r_{j-1},r_j)\otimes W^u(r_j).
\]

We triangulate $\overline{\mathcal{M}}(\tau,q)$ to regard it as a singular chain $m_{\tau,q}$. We can choose triangulations inductively on the dimension of $\tau$ such that
\begin{equation*}
	\partial m_{\tau,q}= m_{\partial \tau,q}+\sum_{|p|>|q|} (-1)^{\dim\tau+d-|q|} m_{\tau,p}\otimes m_{p,q}.
\end{equation*}
We then triangulate $\overline{\mathscr{F}}(\tau)$ to regard it as a singular chain $F(\tau)$. We can choose triangulations inductively on the dimension of $\tau$ such that 
\begin{equation*}
	\partial F(\tau)=-\tau-F(\partial \tau)+\sum_{\dim \tau\ge d-|q|} m_{\tau,q}\otimes d_q.
\end{equation*}
This gives rise to a linear map
\begin{equation} \label{eq:forward_chain}
F \colon S^{\bullet}_{gen}(L;\Z) \to S^{\bullet-1}(L;\Z)
\end{equation}
defined by extending $\tau \mapsto F(\tau)$ linearly. To show that the map $\iota$ in Lemma~\ref{lemma:chain_map} is a quasi-isomorphism, we shall use this map $F$ to produce a cochain homotopy.

\begin{theorem} \label{thm:quasi_iso}
The cochain map $\iota \colon C^{\bullet}(f;\Z)\to S^{\bullet}(L;\Z)$ defined by $\iota(p)=\Delta_p$ is a quasi-isomorphism.
\end{theorem}

\begin{proof}
We note that the chains $\Delta_p$ might not be contained in $S^{\bullet}_{gen}(L;\Z)$ in general. However, by the Morse-Smale property, every stratum of the compactified unstable submanifold $\overline{W}^u(f;p)$ is transverse to the stable submanifolds of all critical points of $f$. Thus, by choosing suitable triangulations on $\overline{W}^u(f;p)$, we can assume $\iota$ factors through $S^{\bullet}_{gen}(L;\Z)$. Since the inclusion of $S^{\bullet}_{gen}(L;\Z)$ into $S^{\bullet}(L;\Z)$ is a quasi-isomorphism, it suffices to show that $\iota \colon C^{\bullet}(f;\Z)\to S^{\bullet}_{gen}(L;\Z)$ is a quasi-isomorphism.

It is obvious that $\Pi\circ \iota = \mathrm{id}$ on $C^{\bullet}(f;\Z)$. 
It remains to show $\iota\circ \Pi$ is homotopic to the identity map on $S^{\bullet}_{gen}(L;\Z)$. Namely, we need to exhibit a map
\begin{equation*} 
G \colon S^{\bullet}_{gen}(L;\Z) \to S_{gen}^{\bullet-1}(L;\Z)
\end{equation*}
satisfying 
\begin{equation}  \label{eq:chain_homotopy}
\iota\circ \Pi-\mathrm{id}= \partial \circ G+G\circ \partial.
\end{equation}
The map G is defined as follows. For a generic simplex $\tau$, we have
\[
(\iota \circ \Pi)(\tau)=\sum_{\dim \tau= d-|q|}  m_{\tau,q}\otimes \Delta_q.
\]
and
\[
\sum_{\dim \tau= d-|q|}  m_{\tau,q}\otimes d_q-\tau + \sum_{\dim \tau > d-|q| } m_{\tau,q}\otimes d_q = \partial F(\tau)+ F(\partial \tau).
\]
Thus, $F$ is \emph{not} a cochain homotopy because of the following reasons$\colon$
\begin{enumerate}
\item (LHS) consists of the boundary components of $m_{\tau,q}\otimes d_q$  (instead of $m_{\tau,q}\otimes \Delta_q$) for $\dim \tau= d-|q|$.
\item (LHS) contains the error terms of the form $m_{\tau,q}\otimes d_q$ for $\dim \tau > d-|q|$.
\item The image of $F$ might not be contained in $S_{gen}^{\bullet}(L;\Z)$.
\end{enumerate}
As we modify $d_p$, we add the following degenerate chains to $F(\tau)$. For each facet of the form $m_{\tau,q}\otimes d_q$, we define a chain $\Delta_{\tau,q}$ by
\begin{equation*}
\Delta_{\tau,q} = \sum_{r_1,\ldots,r_n}   C(\cdots C(Cm_{\tau,q}\otimes m_{q,r_1})\otimes \cdots \otimes m_{r_{n-1},r_n})\otimes d_{r_n}. 
\end{equation*}
Then we define the linear map $G$ by mapping $\tau$ into a chain $G(\tau)$ where
\begin{equation} \label{eq:homotopy_operator}
G(\tau) := F(\tau)+ \sum_{q}  \Delta_{\tau,q}.
\end{equation}
Similar to the proof of Lemma \ref{lemma:chain_map}, we confirm the equation~\eqref{eq:chain_homotopy}$\colon$
\begin{align*} \label{eq:chain_homotopy_2}
\partial G(\tau) & = -\tau - F(\partial \tau) + \sum_{\dim \tau= d-|q|}  m_{\tau,q}\otimes \Delta_q\\
& \quad -\sum_{r_1,\ldots,r_n}  \sum_{q}  C(\cdots C(Cm_{\partial \tau,q}\otimes m_{q,r_1})\otimes \cdots \otimes m_{r_{n-1},r_n})\otimes d_{r_n}\\
&= -\tau - F(\partial \tau) - \sum_q \Delta_{\partial \tau, q}  + \sum_{\dim \tau= d-|q|}  m_{\tau,q}\otimes \Delta_q\\
&=-\tau- G(\partial \tau) + \sum_{\dim \tau= d-|q|}  m_{\tau,q}\otimes \Delta_q. 
\end{align*}
Finally, we note that for a generic simplex $\tau$, every stratum of the compactified forward orbit $\overline{\mathscr{F}}(\tau)$ is transverse to stable submanifolds of all critical points of $f$. Thus, by choosing a suitable triangulations, the image of  $G$ is contained in $S^{\bullet}_{gen}(L;\Z)$. Therefore, this gives us a desired cochain homotopy.
\end{proof}

We derive a property of the homotopy $G$ for the later purpose.
\begin{lemma} \label{lemm:Pi_G}
The composition $\Pi\circ G=0$.
\end{lemma}
\begin{proof}
For a generic $n$-simplex $\tau$, we have 
\[
(\Pi\circ G)(\tau)= \sum_{|q|=d-n-1}  \sharp \mathcal{M}(G(\tau),q)\cdot q.
\]
Note that $\dim \mathcal{M}(\tau,q)= \emptyset$ for every critical point $q$ with $|q|<d-n$. 
It implies that $\mathcal{M}(G(\tau),q)= \emptyset$ since there exists no flow line from $\tau$ to $q$ with $|q|=d-n-1$.
\end{proof}

\begin{remark}
The standard approaches to prove the equivalence between Morse and singular homology are to use self-indexing Morse function and the Morse-Smale complex, or to mod out the degenerate chains (see e.g. \cite{hutchings02,hutchings99}). The closest approach we were able to find in the literature is \cite{schwarz99}, in which one associates a pseudo-cycle to each Morse cycle. 
\end{remark}


\subsection{Morse model with a strict unit}
\label{sec:Morse model}

In this section, by adapting the homotopy unit construction in \cite[Chapter 7]{FOOO} (see also Charest-Woodward \cite{CW15}) and using the quasi-isomorphism $\iota$ in Theorem~\ref{thm:quasi_iso}, we explain how to construct a unital $A_\infty$-algebra on a pearl complex. 

We begin by recalling the construction of a non-unital $A_{\infty}$-algebra structure on $C^{\bullet}(f; \Lambda_0)$ outlined in \cite{FOOO-can}.
Let
\[
\mathcal{X}_{-1}(L)=\{\Delta_p \mid p\in\Crit(f)\}.
\]
We identify $C^{\bullet}_{(-1)}(L;\Lambda_0)$ with $C^{\bullet}(f;\Lambda_0)$ via the inclusion map $\iota$ in~\eqref{eq:delta_p}.

Let $g \ge 0$ be a generation in Section~\ref{sec:simplicial}.
For all $g^\prime$ with $g^\prime <g$, suppose a countable set $\mathcal{X}_{g'}(L)$ of smooth singular chain on $L$ had been constructed. The perturbations  $\fs_{\fd,\ell,\beta,\vec{P}}$ and $\fs^+_{\fd^{+},k,\beta,\vec{P}^{+}}$ in~\eqref{equ_frakks} and~\eqref{m+} can be chosen to have the following properties: 
Every simplex $\tau$ in the triangulation of
	\[
	(\ev_0)_*\left(\mathcal{M}_{\ell+1}(\beta;L;P_1,\ldots,P_{\ell})^{\fs_{\fd,\ell,\beta,\vec{P}}}\right) \mbox{ with $\norm{(\fd,\beta)}=g$} 
	\]
is generic, i.e., it is transverse to the stable submanifold for all $p\in\Crit(f)$. 



We denote by $\mathcal{X}_{g}(L)^{\circ}$ the set of all such simplices and a set of generic chains such that the complex 
\[
C^{\bullet}_{(g)}(L;\Lambda_0)^{\circ}= \Lambda_0\cdot  \mathcal{X}_{g}(L)^{\circ}\oplus \left( \bigoplus_{g'<g} \Lambda_0\cdot \mathcal{X}_{g'}(L)    \right)
\] 
is a quasi-isomorphic subcomplex of $S^{\bullet}(L;\Lambda_0)$. Notice that $C^{\bullet}_{(g)}(L;\Lambda_0)^{\circ}$ is a subcomplex of $S^{\bullet}_{gen}(L;\Z)$. and hence we have a restriction of the projection map 
\[
\Pi \colon C^{\bullet}_{(g)}(L;\Lambda_0)^{\circ} \to  C^{\bullet}(f;\Lambda_0)\cong C^{\bullet}_{(-1)}(L;\Lambda_0).
\] 

In order for the restriction of the homotopy operator $G$ in~\eqref{eq:homotopy_operator} to be well-defined, we enlarge $\mathcal{X}_{g}(L)^{\circ}$ to a set $\mathcal{X}_{g}(L)$ by adding singular chains of the form $G^n(\tau)$ and $\partial G^n(\tau)$ for all $\tau \in \mathcal{X}_{g}(L)$ and $n\ge 1$. 
 
 \begin{lemma} Let $C^{\bullet}_{(g)}(L;\Lambda_0)$ be the cochain complex defined by \eqref{eq:C_g}. Then
 $C^{\bullet}_{(g)}(L;\Lambda_0)$ satisfies the following properties:
 	\begin{enumerate}
 		\item $C^{\bullet}_{(g)}(L;\Lambda_0)$ is closed under $F_{\Delta}$.
 		\item $C^{\bullet}_{(g)}(L;\Lambda_0)$ is a quasi-isomorphic subcomplex of $S^{\bullet}(L;\Z)$.
 	\end{enumerate}
 \end{lemma}
\begin{proof}
For (1), we only have to check that $G(\partial G^n(\tau))\in C^{\bullet}_{(g)}(L;\Lambda_0)$. By \eqref{eq:chain_homotopy}, we have
\[
G(\partial F^n_{\Delta}(\tau))= -\partial G^{n+1}(\tau)-\partial G^n(\tau)+\Pi(\tau),
\]
and the RHS is in $C^{\bullet}_{(g)}(L;\Lambda_0)$. 

For (2), we need to show that added cocycles are cohomologous to the existing cocyles.  The cocyles of the form $G^n(\tau)$ are zero in cohomology. 
Suppose $G(\tau)$ is a cocycle, i.e., $\partial G^n(\tau)=0$. Then, 
\[
\partial G^{n+1}(\tau)=\Pi(G^n(\tau))-G^n(\tau), 
\]
by \eqref{eq:chain_homotopy}. This means $G^n(\tau)$ is cohomologous to the cocycle $\Pi(G^n(\tau))\in C^{\bullet}_{(g)}(L;\Lambda_0)^{\circ}$. Since $C^{\bullet}_{(g)}(L;\Lambda_0)^{\circ}$ is a quasi-isomorphic subcomplex of $S^{\bullet}(L;\Z)$, $C^{\bullet}_{(g)}(L;\Lambda_0)$ remains a quasi-isomorphic subcomplex.
\end{proof}

The procedure above gives us three linear maps
\begin{align*}
\iota \colon & C^{\bullet}(f;\Lambda_0) \to C^{\bullet}(L;\Lambda_0),\\
\Pi \colon & C^{\bullet}(L;\Lambda_0) \to C^{\bullet}(f;\Lambda_0),\\
G \colon &C^{\bullet}(L;\Lambda_0) \to C^{\bullet-1}(L;\Lambda_0).
\end{align*}
The triple $(\iota, \Pi,G)$ forms a \textit{contraction} for the homological perturbation:

\begin{defn} \label{def:contraction}
Let $(C,\frak{m}_{1,0})$ and $(H,\delta)$ be two cochain complexes. A triple $(\iota,\Pi,G)$ consisting of maps $\iota \colon H\to C$, $\Pi \colon C\to H$ of degree zero and a map $G \colon C\to C$ of degree $-1$ is called a \emph{contraction} if the following equations hold
\begin{equation}
	\frak{m}_{1,0}\circ \iota = \iota \circ \delta, \quad \Pi\circ \frak{m}_{1,0}= \delta \circ \Pi,
\end{equation}
\begin{equation}
\iota \circ \Pi - id_{C} = \frak{m}_{1,0}\circ G + G\circ \frak{m}_{1,0}. 
\end{equation}
\end{defn}
We apply homological perturbation to reduce the $A_{\infty}$-algebra structure on  $C^{\bullet}(L;\Lambda_0)$ to $C^{\bullet}(f;\Lambda_0)$ using the contraction $(\iota,\Pi,G)$, following \cite[Theorem ~5.1]{FOOO-can} (see also \cite[Theorem ~4.4]{Yuan21}). 

The $A_{\infty}$-algebra structure on $C^{\bullet}(f;\Lambda_0)$ is \emph{not} unital in general. In order to define a unital $A_{\infty}$-algebra structure on a pearl complex, we modify the construction as follows.

For simplicity, we shall always assume $f$ has a unique maximum point $\bunit$, so that $\Delta_{\bunit}=\bm{e}$ is the fundamental cycle. 
Let us set
\begin{equation}\label{equ_CFLlam0}
	CF^{\bullet}(L;\Lambda_0)=C^{\bullet}(f;\Lambda_0)\bigoplus \Lambda_0 \cdot \wunit \bigoplus \Lambda_0 \cdot \gunit
\end{equation}
with $|\wunit|=0$ and $|\gunit|=-1$. 
 (The superscripts $\tiny{\triangledown}$, $\blacktriangledown$ and $\textcolor{gray}{\blacktriangledown}$ are borrowed from \cite{CW15}, where they adapted the homotopy unit construction to the stabilizing divisor perturbation scheme for their unital Morse model.) We extend the Morse differential $\delta$ to  $CF^{\bullet}(L;\Lambda_0)$ by setting 
\begin{equation*}
	\label{delta+}
	\delta(\wunit)=0, \quad \delta(\gunit)=\wunit-\bunit. 
\end{equation*}
We extend $\iota$ and $\Pi$ to the chain maps
\begin{align*}
\iota \colon &CF^{\bullet}(L;\Lambda_0)\to C^{\bullet}(L;\Lambda_0)^+,\\
\Pi \colon & C^{\bullet}(L;\Lambda_0)^+ \to CF^{\bullet}(L;\Lambda_0),
\end{align*}
by setting
\begin{align*}
&\iota( \wunit)= \bm{e}, \quad  \iota(\gunit)=\bm{f},\\
&\Pi(\bm{e})=\wunit, \quad \Pi(\bm{f})=\gunit. 
\end{align*}
Here $C^{\bullet}(L;\Lambda_0)^+$ is the unital singular chain model \eqref{equ_cLlambda0+}  constructed via the homotopy unit construction. We now define a modified homotopy operator 
\begin{equation*}
G^+ \colon C^{\bullet}(L;\Lambda_0)^+\to C^{\bullet-1}(L;\Lambda_0)^+.
\end{equation*}
First, we extend $G$ by setting
\begin{equation*}
G(\bm{e}^+)=-\bm{f}, \quad G(\bm{f})=0.
\end{equation*}
Next, we choose a splitting 
\[
C^{\bullet}(L;\Lambda_0)^+=V\oplus W,
\]
where $V$ is the submodule generated by $\Image(G)$ and $\mathcal{X}_{-1}(L)$, and $W$ is a complement. We then define a linear map $G^+$ by setting
\begin{equation} \label{eq:G}
\begin{cases}
	G^+(\sigma)=G(\sigma) &\mbox{for $\sigma\in W$},  \\
	G^+(\sigma)=0 &\mbox{for $\sigma\in V$}. 
\end{cases}
\end{equation}

\begin{lemma} \label{lem:G}
The linear map $G^+$ is a cochain homotopy between $\iota\circ \Pi$ and the identity map, i.e., 
\begin{equation} \label{eq:G_chain_homotopy}
\iota\circ \Pi - id= \partial \circ G^+ + G^+ \circ \partial.
\end{equation} 
\end{lemma}

\begin{proof}
We note that $G^+$ satisfies \eqref{eq:G_chain_homotopy} on $W$ since $G$ is a cochain homotopy. For $\sigma \in V$, it suffices to check the cases when $\sigma=\Delta_p$ or $\sigma=G(\sigma')$ for some $\sigma'\in W$. In the former case, plugging $\Delta_p$ into the LHS and RHS of  \eqref{eq:G_chain_homotopy} respectively, we get
\[
(\iota \circ \Pi)(\Delta_p)-\Delta_p= \Delta_p-\Delta_p=0, 
\]
and
\[
\partial G^+(\Delta_p) - G^+(\partial \Delta_p)=0- \sum_{|q|=|p|+1} G^+(\Delta_q)=0.
\]
In the latter case, plugging $G(\sigma')$ into the LHS and RHS of  \eqref{eq:G_chain_homotopy} respectively, we get
\[
(\iota\circ \Pi) (G(\sigma')) - G(\sigma')=-G(\sigma')
\]
by Lemma \ref{lemm:Pi_G}, and 
\begin{align*}
\partial G^+(G(\sigma'))+ G^+(\partial G(\sigma') ) &= G^+(-\sigma' - G(\partial \sigma')+ \iota\circ \Pi (\sigma'))\\
&= -G^+(\sigma')-G^+(G(\partial \sigma'))+G^+( \iota\circ \Pi (\sigma'))\\
&=-G(\sigma'). 
\end{align*}
This shows that $G^+$ is a homotopy operator.
\end{proof}

The contraction $(\iota,\Pi,G^+)$ has the following additional properties which ensure that the $A_\infty$-algebra arising from the homological perturbation is unital.

\begin{defn}[\cite{Yuan21}] \label{def:strong_contraction}
A contraction $(\iota,\Pi,G^+)$ is called a \textit{strong contraction} if it has the following additional properties:
\begin{equation} \label{eq:property_1}
\Pi\circ \iota - \text{id}_{H}=0,
\end{equation}
\begin{equation}  \label{eq:property_2}
G^+\circ G^+=0,
\end{equation}
\begin{equation}  \label{eq:property_3}
G^+\circ \iota =0,
\end{equation}
\begin{equation}  \label{eq:property_4}
\Pi\circ G^+=0.
\end{equation}
\end{defn}

\begin{lemma}
The triple $(\iota,\Pi,G^+)$ is a strong contraction. 
\end{lemma}
\begin{proof}
The triple  $(\iota,\Pi,G^+)$ is a contraction by Lemma \ref{lem:G} and the fact that the triple $(\iota,\Pi,G)$ is a contraction. As for the additional conditions, \eqref{eq:property_1} follows from the definition of $\iota$ and $\Pi$, \eqref{eq:property_2} and  \eqref{eq:property_3} follow from the definition of $G^+$, and  \eqref{eq:property_4} follows from Lemma \ref{lemm:Pi_G}. 
\end{proof}

We apply homological perturbation to reduce the $A_{\infty}$-algebra structure from $C^{\bullet}(L;\Lambda_0)^+$ to $CF^{\bullet}(L;\Lambda_0)$ using the strong contraction $(\iota,\Pi,G^+)$. We denote this $A_{\infty}$-algebra by $(CF^{\bullet}(L;\Lambda_0),\fm)$. By the proof of Proposition 4.7 in \cite{Yuan21}, the $A_\infty$-algebra $(CF^{\bullet}(L;\Lambda_0),\fm)$ from the strong constraction becomes unital. In sum, we have derived the following theorem. 

\begin{theorem}\label{thm:homotopy_equivalence}
The pair $(CF^{\bullet}(L;\Lambda_0),\fm)$ is a unital $A_{\infty}$-algebra with the strict unit $\wunit$, which is homotopy equivalent to 
$(C^{\bullet}(L;\Lambda_0)^+,\tilde{\fm}^+)$. Moreover, the homotopy equivalence is unital.
\end{theorem}

We now give an explicit description for the $A_\infty$-structure maps of $(CF^{\bullet}(L;\Lambda_0),\fm)$. 
The structure maps $\fm_k$ can be exrpessed in terms of maps $\fm_{\Gamma}$ and $\ff_{\Gamma}$ associated to \emph{decorated planar rooted trees}.
Let us fix a labeling $\{\beta_0,\beta_1,\ldots\}$ of elements of $H_2^{\mathrm{eff}}(X,L)$ where $\beta_0$ is the constant disc class.

\begin{defn}[\cite{FOOO-can}]
	A \emph{decorated planar rooted tree} is a quintuple $\Gamma=(T,i,v_0,V_{tad},\eta)$ consisting of 
	\begin{itemize}
		\item $T$ is a tree,
		\item $i \colon T\to D^2$ is an embedding into the unit disc,
		\item $v_0$ is the root vertex and $i(v_0)\in\partial D^2$,
		\item $V_{tad}$ is the set of interior vertices with valency $1$,
		\item $\eta \colon V(\Gamma)_{int}\to \Z_{\ge 0}$,
	\end{itemize}
	where $V(\Gamma)$ is the set of vertices, $V(\Gamma)_{ext}=i^{-1}(\partial D^2)$ is the set of exterior vertices, and $V(\Gamma)_{int}=V(\Gamma)\setminus V(\Gamma)_{ext}$ is the set of interior vertices.
\end{defn}

 For $k \ge 0$, we denote by $\bm{\Gamma}_{k+1}$ the set of trees $\Gamma=(T,\iota,v_0,\eta)$ with $|V(\Gamma)_{ext}|=k+1$ and $\eta(v)>0$ if the valency $\ell(v)$ of $v$ is $1$ or $2$. In other words, a holomorphic disc in class $\beta_{\eta(v)}$ with $\ell(v)$ boundary marked points is stable. 
We will refer to the elements of $\bm{\Gamma}_{k+1}$ as stable trees. 

We denote by $\Gamma^0\in \bm{\Gamma}_2$ the unique tree with no interior vertices.
For this tree $\Gamma_0$, set
\[
\begin{cases}
\fm_{\Gamma_0} \colon CF^{\bullet}(L;\Lambda_0)\to CF^{\bullet}(L;\Lambda_0) &\mbox{ given by $\fm_{\Gamma_0}:= \tilde{\fm}_{1,\beta_0}$}  \\
\ff_{\Gamma_0} \colon CF^{\bullet}(L;\Lambda_0)\to C^{\bullet}(L;\Lambda_0)^+ &\mbox{ given by $\ff_{\Gamma_0}:= \iota$} 
\end{cases}
\]
where $\iota$ is the inclusion map in~\eqref{eq:delta_p}. 
For each $k \ge 0$, $\bm{\Gamma}_{k+1}$ contains a unique element that has a single interior vertex $v$, which is denoted by $\Gamma_{k+1}$.
We then define
\[
\begin{cases}
\fm_{\Gamma_{k+1}}=\Pi\circ \tilde{\fm}^+_{k,\beta_{\eta(v)}}, \\
\ff_{\Gamma_{k+1}}=G\circ \tilde{\fm}^+_{k,\beta_{\eta(v)}}.
\end{cases}
\]
For general $\Gamma$, we cut it at the vertex $v$ closest to $v_0$ so that $\Gamma$ is decomposed into $\Gamma^{(1)},\ldots,\Gamma^{(\ell)}$ and an interval adjacent to $v_0$ in the counter-clockwise order. We inductively define
\[
\begin{cases}
\fm_{\Gamma}=\Pi\circ \tilde{\fm}^+_{\ell,\beta_{\eta(v)}}\circ (\ff_{\Gamma^{(1)}} \otimes\ldots\otimes\ff_{\Gamma^{(\ell)}})\\
\ff_{\Gamma}=G\circ \tilde{\fm}^+_{\ell,\beta_{\eta(v)}}\circ (\ff_{\Gamma^{(1)}} \otimes\ldots\otimes\ff_{\Gamma^{(\ell)}}).
\end{cases}
\]

The maps $\ff_k \colon CF^{\bullet}(L;\Lambda_0)^{\otimes k}\to C^{\bullet}(L;\Lambda_0)^+$ defined by 
\begin{equation*}
	\label{eq:f_k}
	\ff_k=\sum_{\Gamma\in\bm{\Gamma}_{k+1}} \bT^{\omega_X(\Gamma)} \ff_{\Gamma},
\end{equation*}
is a homotopy equivalence by Theorem \ref{thm:homotopy_equivalence}. The $A_{\infty}$-maps $\fm_k \colon CF^{\bullet}(L;\Lambda_0)^{\otimes k}\to CF^{\bullet}(L;\Lambda_0)$ are given by 
\begin{equation}
\label{eq:m_k}
\fm_k=\sum_{\Gamma\in\bm{\Gamma}_{k+1}} \bT^{\omega_X(\Gamma)} \fm_{\Gamma},
\end{equation}
where $\omega_X(\Gamma)=\sum_{v\in V_{int}} \omega_X(\beta_{\eta(v)})$.


\begin{figure}[h]
	\begin{center}
		\includegraphics[scale=0.57]{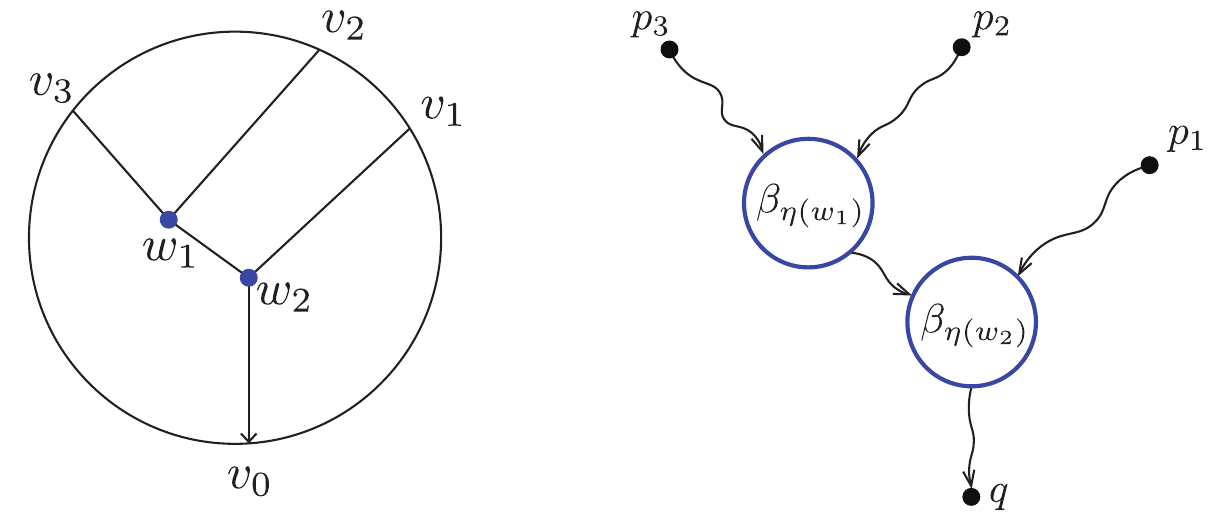}
		\caption{Pearl trees}
		\label{fig:pearl}
	\end{center}
\end{figure}

The maps $\fm_{\Gamma}$ restricted to $C^{\bullet}(f;\Lambda_0)$ are given by counting pearly trees as depicted in Figure~\ref{fig:pearl} (This observation was made in \cite{FOOO-can}. See also \cite{CW15} for a description of $\fm_{\Gamma}$ with general inputs in terms of pearly trees).  
For a decorated tree $\Gamma\in\bm{\Gamma}_{k+1}$, the exterior vertices $v_0,\ldots,v_k$ are labeled respecting the counter-clockwise orientation. 
Each edge is oriented in the direction from the $k$ input vertices $v_1,\ldots,v_k$ to the root vertex $v_0$. 
We denote by $e_i$ the edge attached to $v_i$ and by $v^{\pm}(e)$ the vertices such that $e$ is the edge from $v^-(e)$ to $v^+(e)$. For critical points $p_1,\ldots,p_k,q\in \mathrm{Crit}(f)$, consider the moduli space $\mathcal{M}_{\Gamma}(f;p_1,\ldots,p_k,q)$ consisting of the following configurations:
\begin{itemize}
	\item for each interior vertex $v$, a bordered stable map $u_v$ representing the class $\beta_{\eta(v)}$ with $\ell(v)$ boundary marked points. We denote by $p(e,v)$ the marked point corresponding to the edge $e$ attached to $v$;
	\item  for $i=1,\ldots,k$, the input edge $e_i$ corresponds to a flow line $\gamma_i$ from $p_i$ to $u_{v^+(e_i)}(p(e_i,v^+(e_i))$;
	\item the output edge $e_0$ corresponds to a flow line $\gamma_0$ from $u_{v^-(e_0)}(p(e_0,v^-(e_0)))$ to $q$;
	\item an interior edge $e$ corresponds to a flow line $\gamma_e$ from $u_{v^{-}(e)}(p(e,v^-(e)))$ to $u_{v^{+}(e)}(p(e,v^+(e)))$.
\end{itemize}

The virtual dimension of $\mathcal{M}_{\Gamma}(f;p_1,\ldots,p_k,q)$ is equal to 
\begin{equation} \label{eq:dim_formula}
\dim \mathcal{M}_{\Gamma}(f;p_1,\ldots,p_k,q)=k-2+\mu(\Gamma)-\sum_{i=1}^k |p_i|+|q|,
\end{equation}
where $\mu(\Gamma)=\sum_{v\in V_{int}} \mu(\beta_{\eta(v)})$.
The map $\mathfrak{m}_{\Gamma} \colon C^{\bullet}(f;\Lambda_0)^{\otimes k}\to C^{\bullet}(f;\Lambda_0)$ is defined by
\[	
\mathfrak{m}_{\Gamma}(p_1,\ldots,p_k)=\sum_{q\in\mathrm{Crit}(f)}\sharp \mathcal{M}_{\Gamma}(f;p_1,\ldots,p_k,q)\cdot q,
\]
where $\sharp\mathcal{M}_{\Gamma}(f;p_1,\ldots,p_k,q)$ is the signed count of the moduli space of virtual dimension zero.





\subsection{Disc potentials in Lagrangian Floer theory} 

In this section, we recall the disc potential of $(CF^{\bullet}(L;\Lambda_0),\fm)$ introduced by Fukaya-Oh-Ohta-Ono \cite{FOOO}.

Let $\Lambda_+\subset \Lambda_0$ denote the maximal ideal
\[	
\Lambda_+=\left\{\sum_{i=0}^{\infty} a_iT^{A_i}|a_i\in \mathrm{k}, A_i\in\R_{> 0}; A_i \text{ increases to } +\infty,\right\}.
\]
Let $(A,\fm)$ be an $A_{\infty}$-algebra over $\Lambda_0$ with the strict unit $e_A$,
and set 
\[
A_+=\{x\in A \mid x\equiv 0 \mod \Lambda_+\cdot A \}.
\]

The \emph{weak Maurer-Cartan equation} for an element $b\in A_+$ is given by
\begin{equation*}
\label{eq:MC}
\fm^b_0(1)=\fm_0(1)+\fm_1(b)+\fm_{2}(b,b)+\ldots \in \Lambda_0 \cdot e_A.
\end{equation*}
The condition that $b\in A_+$ ensures the convergence of $\fm^b_0(1)$. A solution $b$ of \eqref{eq:MC} is called a \emph{weak bounding cochain}. We denote by 
\begin{equation*}
\mathcal{MC}(A)= \left\{b\in A^{odd}_+ \mid \fm_0^b(1)\in \Lambda_0\cdot e_A \right\},
\end{equation*}
the space of weak Maurer-Cartan elements. 
We say an $A_{\infty}$-algebra $A$ is \emph{weakly unobstructed} if $MC(A)$ is nonempty, in which case we have $(\fm^b_1)^2=0$ for any $b\in \mathcal{MC}(A)$, thus defining a cohomology theory $H^{\bullet}(A,\fm^b_1)$. 

Let us put 
\begin{equation*}
e^b:=1+b+b\otimes b+\ldots.
\end{equation*}
For an element $b\in \mathcal{MC}(A)$, we can define a deformation $\fm^b$ of the $A_{\infty}$-structure $\fm$ by
\begin{equation} \label{eq:deformation}
\fm^b_k(x_1,\ldots,x_k)=\fm(e^b,x_1,e^b,x_2,e^b,\ldots,e^b,x_k,e^b), \quad x_1,\ldots,x_k\in A.
\end{equation}

Now, for $(A,\fm)=(CF^{\bullet}(L;\Lambda_0),\fm)$, we denote by $\mathcal{MC}(L)$ the space of odd degree 
weak bounding cochains. In general, one should also consider gauge equivalences between weak Maurer-Cartan solutions. However, we omit them here since they will trivial in our examples. 

We say that the deformation of $L$ by $b$ (or simply $(L,b)$) is \emph{unobstructed} (resp. \emph{weakly unobstructed}), if $\mathfrak{m}^{b}(1)=0$ (resp. $b\in \mathcal{MC}(L)$). In particular, if $b=0$, we will simply say $L$ is unobstructed.

The following lemma concerns the weakly unobstructedness of $(CF^{\bullet}(L;\Lambda_0),\fm)$. This technique of finding weak bounding cochains in the presence of a homotopy unit was introduced in \cite[Chapter 7]{FOOO} and \cite[Lemma 2.44]{CW15}.

\begin{lemma}
\label{lem:unobstr}
Let $b\in CF^{1}(L;\Lambda_+)$. Suppose $\fm_0^b=W(b)\bunit$ and the minimal Maslov index of $L$ is nonnegative. Then there exists $b^+\in CF^{odd}(L;\Lambda_0)$ such that $\fm_0^{b^+}(1)=W^{\wT}(b)\wunit$ for some $W^{\wT}(b)\in \Lambda_0$, i.e., $(CF^{\bullet}(L;\Lambda_0),\fm)$ is weakly unobstructed. In particular, if the minimal Maslov index of $L$ is at least two, then, $W^{\wT}(b)=W(b)$.
\end{lemma}

\begin{proof}
By \eqref{eq:m_1_f}, we have
\begin{equation}
\label{eq:m_1_f_2}
\fm_1(\gunit)= \fm_{\Gamma_0}(\gunit) + \sum_{\Gamma \in\bm{\Gamma}_{2} - \{\Gamma_0\}} \bT^{\omega_X(\Gamma)} \fm_{\Gamma}(\gunit) =  \wunit-\bunit + \sum_{\Gamma \in\bm{\Gamma}_{2} - \{\Gamma_0 \}} \bT^{\omega_X(\Gamma)} \fm_{\Gamma}(\gunit), 
\end{equation}
where $\Gamma_0$ is the unique tree with no interior vertices. Let $\Gamma\in\bm{\Gamma}_{2}$ be a decorated tree. Since the minimal Maslov index of $L$ is assumed to be nonnegative, by the dimension formula \eqref{eq:dim_formula}, $\fm_{\Gamma}(\gunit)$ has degree at most zero. If the Maslov index $\mu(\Gamma)>0$, $\fm_{\Gamma}(\gunit)$ is a singular chain of negative degree and its projection to $CF^{\bullet}(L;\Lambda_0)$ vanishes. This means $\fm_{\Gamma}(\gunit)$ is of degree zero and hence a multiple of $\bunit$. (We note that while $\wunit$ is also in degree zero, it is not a critical point so that it does not appear in $\fm_{\Gamma}(\gunit)$. It only appears as an output for $\fm_2(\wunit,\wunit)$.) Moreover, since $\omega_X(\Gamma)>$0 for $\Gamma \ne \Gamma_0$, we can rewrite \eqref{eq:m_1_f_2} as
\begin{equation}
\label{eq:m_1_f_3}
\fm_1(\gunit)=\wunit-(1-h)\bunit,
\end{equation}
for some $h\in \Lambda_+$. 	

Let us tentatively take $b'=b+\gunit$ and write
\[
\fm_0^{b'}(1)=\fm_0(1)+\fm_1(b+\gunit)+\fm_2(b+\gunit,b+\gunit)+\ldots.
\]
By the assumption that minimal Maslov index of $L$ is nonnegative and the dimension formula \eqref{eq:dim_formula}, the $\fm_k$ operations with more than one $\gunit$ as  inputs vanish since the outputs are of degrees at most $(-2)$. Thus, we can write
\[
\fm_0^{b'}(1)=W(b)\bunit+ \fm_{1}^b (\gunit),
\]
where 
\begin{align*}
\fm_{1}^b (\gunit) &= \fm_1(\gunit) + \sum_{\ell_1 + \ell_2 \ge 1} \fm_{1+\ell_1+\ell_2}(b^{\otimes \ell_1}, \gunit, b^{\otimes \ell_2})\\
&= \wunit-(1-h)\bunit +  \sum_{\ell_1 + \ell_2 \ge 1} \sum_{\Gamma\in\bm{\Gamma}_{\ell_1 + \ell_2+2}} \bT^{\omega_X(\Gamma)} \fm_{\Gamma}(b^{\otimes \ell_1}, \gunit, b^{\otimes \ell_2}).
\end{align*}
Since the terms in the summation are of degree $0$, the sum is a multiple of $\bunit$. We can write 
\[
\fm_{1}^b (\gunit)= \wunit-(1-h(b))\bunit,
\]
for some $h(b) \in \Lambda_0$.  Suppose we have $\Gamma\in\bm{\Gamma}_{\ell_1 + \ell_2+2}$ with $\omega_X(\Gamma)=0$, then the interior vertices of $\Gamma$ are decorated by the constant disc class, and $\fm_{\Gamma}(b^{\otimes \ell_1}, \gunit, b^{\otimes \ell_2})$ counts Morse flow trees ending at the maximal point $\bunit$. Since our flow is decreasing, it is easy to see that such flow trees do not exist and hence $\fm_{\Gamma}(b^{\otimes \ell_1}, \gunit, b^{\otimes \ell_2})=0$ when $\omega_X(\Gamma)=0$. This means $h(b)\in \Lambda_+$.

Now, by setting $b^+=b+\frac{W(b)}{1-h(b)}\gunit$, we get
\begin{equation*}
\fm_0^{b^+}(1)=W(b)\bunit+\frac{W(b)}{1-h(b)}\left(\wunit-(1-h(b))\bunit\right)=\frac{W(b)}{1-h(b)}\wunit=:W^{\wT}(b)\wunit.
\end{equation*}
Furthermore, 
If the minimal Maslov index of $L$ is at least two, then $h(b)=0$ by the dimension formula \eqref{eq:dim_formula}, and hence $W^{\wT}(b)=W(b)$.
\end{proof}	

Let $\mathcal{MC}^{\blT}(L)$ the space of elements $b\in CF^{1}(L;\Lambda_+)$ which satisfy
\[
\fm^b_0(1)=W(b)\bunit  \text{ for some } W(b)\in \Lambda_0.
\]
By the above lemma, for each such $b$, there exists $b^+\in CF^{odd}(L;\Lambda_0)$ such that $\fm_0^{b^+}(1)=W^{\wT}(b)\wunit$. Thus $W^{\wT}(b)$ can be viewed as a function on $\mathcal{MC}^{\blT}(L)$.

\begin{defn} \label{def:disc_potential}
We call $W^L:=W^{\wT} \colon \mathcal{MC}^{\blT}(L)\to \Lambda_0$ the \emph{disc potential} of $L$. 
\end{defn}


\section{A Morse model for equivariant Lagrangian Floer theory} \label{sec:Morse} 

There are several approaches to $G$-equivariant Lagrangian Floer theory for a pair of G-invariant Lagrangians in the existing literature.  For $G=\Z_2$, Seidel-Smith \cite{SS10} used Floer homology coupled with Morse theory on $EG$ to define $G$-equivariant Lagrangian Floer homology of exact Lagrangians. 
In \cite{HLS16a,HLS16b},  Hendricks-Lipshitz-Sarkar used a homotopy theoretic method to define $G$-equivariant Lagrangian Floer homology for a compact Lie group  $G$. Daemi-Fukaya \cite{DF17} defined an equivariant de Rham model using $G$-equivariant Kuranishi structures developed in \cite{FOOO, G-Kuranishi}. Bao-Honda \cite{BH18} defined equivariant Lagrangian Floer cohomology in the case of finite group action via semi-global Kuranishi structures.

In this section, we develop a Morse model for the $G$-equivariant Lagrangian Floer theory focusing on a single $G$-invariant Lagrangian. 
The underlying cochain complexes can be constructed such that each complex is finite dimensional over the cohomology ring of $BG$, and the $A_\infty$-maps are given by counting pearly trees in the Borel construction $L_G$.  This suits better for our purpose of computing disc potentials and constructing SYZ mirrors. The $A_{\infty}$-algebra we construct will be unital. This uses the homotopy unit construction in \cite{FOOO}, which has also been adapted to the stabilizing divisor perturbation scheme by Charest-Woodward \cite{CW15} for their (non-equivaraint) unital Morse model.  

We define the $G$-equivariant Lagrangian Floer theory as the (ordinary) Lagrangian Floer theory of $L_G$ as a Lagrangian submanifold of a certain symplectic manifold $X_G$. This avoids the issue of equivariant transversality.  The Morse model we use makes the theory much more explicit and computable.  Since the Lagrangians we study here bounds non-constant pseudo-holomorphic discs, we use the machinery of \cite{FOOO} to handle the obstructions.  On the other hand, for computing the equivariant disc potential of toric Fano manifolds, virtual technique is not necessary. 

\subsection{The equivariant Morse model}
\label{sec:G Morse model}
Let $G$ be a compact Lie group. We begin by choosing smooth finite dimensional approximations for the universal bundle $EG \to BG$ over the classifying space. Namely, we have the following diagram
\begin{equation}
\label{eq:EG(N)}
\begin{tikzcd}
EG(0) = G \arrow[hookrightarrow]{r} \arrow{d}   & EG(1) \arrow[hookrightarrow]{r} \arrow{d} & EG(2) \arrow[hookrightarrow]{r} \arrow{d} & \cdots \\
BG(0) = \pt \arrow[hookrightarrow]{r}    & BG(1) \arrow[hookrightarrow]{r}  & BG(2) \arrow[hookrightarrow]{r}  & \cdots 
\end{tikzcd}
\end{equation}
where $EG(N)$ and $BG(N)$ are compact smooth manifolds, and the horizontal arrows are smooth embeddings. 
We note that $BG(N)$ are connected for $N\ge 1$ since $EG(N)$ is $(N-1)$-connected.

We also consider the cotangent bundles $T^*EG(N)$ and $T^*BG(N)$ equipped with the canonical symplectic forms.
Every symplectic $G$-action on a symplectic manifold can be naturally lifted to a Hamiltonian $G$-action  on its cotangent bundle. 
We choose a moment map $\mu_N \colon T^*EG(N)\to\mathfrak{g}^*$ for the Hamiltonian $G$-action on $T^*EG(N)$ lifted from the $G$-action on $EG(N)$ where $\mathfrak{g}^*$ is the dual Lie algebra of $G$. 
Since $G$ acts on $T^*EG(N)$ freely, we have a canonical isomorphism
$$
T^*EG(N)\twobar G:=\mu_N^{-1}(0)/G\cong T^*BG(N),
$$
as symplectic manifolds. 
We denote by $J_{T^*BG(N)}$ the almost complex structure induced by a $G$-invariant compatible almost complex structure on $T^*EG(N)$. Notice that $J_{T^*BG(N)}$ is compatible with the canonical symplectic form on $T^*BG(N)$.

Let $(X,\omega_X)$ be a symplectic $G$-manifold, i.e. $X$ is a symplectic manifold endowed with a $G$-action which preserves $\omega_X$. As in Section~\ref{sec:simplicial}, $X$ is assumed to be convex at infinity or geometrically bounded if it is non-compact. 
Let $L$ be a $G$-invariant, closed, connected, and relatively spin Lagrangian submanifold $X$. We will assume the relative spin structure is preserved by the $G$-action. 
Let us fix a $G$-invariant almost complex structure $J_X$ compatible with $\omega_X$.  
In this section, we will define a Morse model for the $G$-equivariant Lagrangian Floer theory on $L$.

Let $X(N)=X\times_G \mu_N^{-1}(0)$ and set $X_G=\lim_{\to}X(N)$. Since $G$ acts freely on $\mu_N^{-1}(0)$, there is a canonical map $\pi \colon X(N)\to T^*BG(N)$, which is a fiber bundle with the fiber $X$. 
Let $\iota \colon X\to X(N)$ be inclusion  of a fiber. By the construction, $X(N)$ is endowed with a symplectic form $\omega_{X(N)}$ and a compatible almost complex structure $J_{X(N)}$ satisfying 
\begin{equation*}
 \iota^*\omega_{X(N)}=\omega_X, \quad J_{X(N)}|_{X(N-1)}=J_{X(N-1)}, \quad D\iota \circ J_X=J_{X(N)} \circ D\iota, \quad D\pi \circ J_{X(N)}=J_{T^*BG(N)} \circ D\pi. 
\end{equation*}

Let $L(N)=L\times _G EG(N)$ for $N = 0, 1, \dots$, which are finite dimensional approximations of the Borel construction $L_G=L\times _G EG$. We note that $L(N)\to BG(N)$ is a fiber bundle with the fiber $L$. 
Regarding $BG(N)$ as the zero section of $T^*BG(N)$, $L(N)$ is a Lagrangian submanifold of $X(N)$ for each $N$.
The diagram~\eqref{eq:EG(N)} induces a commutative diagram
\begin{equation}
\label{eq:inclusions}
\begin{tikzcd} 
L(0)=L \arrow[hookrightarrow]{r} \arrow{d}   & L(1) \arrow[hookrightarrow]{r} \arrow{d} & L(2) \arrow[hookrightarrow]{r} \arrow{d} & \cdots \\
X(0)=X  \arrow[hookrightarrow]{r}  & X(1)  \arrow[hookrightarrow]{r}   & X(2) \arrow[hookrightarrow]{r}  & \cdots 
\end{tikzcd}
\end{equation}
where the vertical arrows are embeddings of Lagrangian submanifolds. We note that the horizontal arrows in the bottom row of \eqref{eq:inclusions} depend on the choice of auxillary $G$-invariant metrics on $EG(N)$ compatible with restrictions. 

We note that $X(N)$ is neither convex at infinity nor geometrically bounded in general even if $X$ is. On the other hand, 
the following proposition shows that the image of any $J_{X(N)}$-holomorphic discs must be fully contained in a fiber of the map $\pi$ over $BG(N)$. It ensures compactness of the moduli spaces of pseudo-holomorphic discs bounded by $L(N)$. 

\begin{prop}[Effective disc classes]
	\label{prop:disc classes}
	Let $\pi \colon X(N) \to T^*BG(N)$ be the projection map. Then for every $J_{X(N)}$-holomorphic map $u \colon (\Sigma,\partial \Sigma)\to (X(N),L(N))$, the map $\pi\circ u \colon (\Sigma,\partial \Sigma)\to (X(N),L(N))\to (T^*BG(N), BG(N))$ is necessarily constant. In particular, we have a bijection of effective disc classes
	\begin{equation*}
	\label{eq:disc identification}
	\iota_N \colon H_2^{\mathrm{eff}}(X,L)\xrightarrow{\sim} H_2^{\mathrm{eff}}(X(N),L(N)).
	\end{equation*}
\end{prop}

\begin{proof}
	Suppose that there is a $J_{X(N)}$-holomorphic map $u \colon (\Sigma,\partial \Sigma)\to (X(N),L(N))$. 
	\begin{equation*}
\begin{tikzcd} 
(\Sigma, \partial \Sigma) \arrow[r,  "u"] \arrow[rd, "\pi \circ u" ']   & (X(N), L(N)) \arrow[d, "\pi"]  \\
\,    & (T^*BG(N), BG(N)) 
\end{tikzcd}
\end{equation*}
Then the composition $\pi\circ u$ is a $J_{T^*BG(N)}$-holomorphic disc in $T^*BG(N)$ bounded by the zero section $BG(N)$. Since $BG(N)$ is an exact Lagrangian in $T^*BG(N)$, it does not bound any non-constant pseudo-holomorphic discs. It implies that $\pi\circ u$ is necessarily constant and $\Image(u)$ is contained in a fiber of $\pi$ over $BG(N)$.
\end{proof}

Proposition~\ref{prop:disc classes} has the following corollaries that will play an important role in computing the disc potential of $L(N)$ later on.

\begin{corollary}[Maslov index]\label{cor_mas}
	The Maslov index of $\beta\in H_2^{\mathrm{eff}}(X,L)$ is equal to the Maslov index of $\iota_N(\beta) \in H_2^{\mathrm{eff}}(X(N),L(N))$. 
\end{corollary}

By abuse of notation, we denote both the disc classes $\beta\in  H_2^{\mathrm{eff}}(X,L)$ and $\iota_N(\beta)\in H_2^{\mathrm{eff}}(X(N),L(N))$ by $\beta$. We also denote their Maslov index by $\mu(\beta)$ and symplectic area by $\omega(\beta)$.

\begin{corollary}[Regularity]\label{cor_reg}
	\label{regularity}
	A $J_{X(N)}$-holomorphic disc $u \colon (\Sigma,\partial \Sigma)\to (X(N),L(N))$ is regular if it is regular as a disc in the corresponding fiber of $\pi$.
\end{corollary}

\begin{proof}
	Let $E:=u^*TX(N)$ and $F:=(\partial u)^*TL(N)$. We denote by $A^0(E,F)$ the space of smooth global sections of E with boundary values in $F$ and by $A^1(E)$ the space of smooth global $(0,1)$-forms with coefficient in $E$. Consider the two term elliptic complex
	\begin{equation}
	\label{eq:dbar}
	A^0(E,F)\xrightarrow{\bar{\partial}}A^1(E)
	\end{equation}
	where $\bar{\partial}$ is the linearized Cauchy-Riemann operator at $u$. Pulling the following exact sequences
	\[
	0\to TX\to TX(N)\to TT^*BG(N)\to 0
	\]
	and
	\[
	0\to TL\to TL(N)\to TBG(N)\to 0
	\]	
	via $u$, we choose splittings
	\[
	E=E_1\oplus E_2:=u^*TX\oplus (\pi\circ u)^*T(T^*BG(N))
	\]
	and
	\[
	F=F_1\oplus F_2:= (\partial u)^* TL\oplus (\pi\circ \partial u)^* TBG(N),
	\]
	where $F_i\subset E_1$ and $F_2\subset E_2$ are subbundles. This gives a decomposition of (\ref{eq:dbar})
	\[
	A^0(E_1,F_1)\oplus A^0(E_2,F_2) \xrightarrow{(\bar{\partial}_1, \bar{\partial}_2)}A^1(E_1)\oplus A^1(E_2).
	\]
	By Proposition \ref{prop:disc classes}, $\pi\circ u$ is a constant map and hence $\bar{\partial}_2$ is surjective. Therefore, $\bar{\partial}$ is surjective if and only if $\bar{\partial}_1$ is surjective. 
\end{proof}

\begin{corollary}
We have
\begin{equation*}
	\cM_{k+1}(\beta; L(N))= \cM_{k+1}(\beta;L)\times_G EG(N)
\end{equation*}
as topological spaces. 
\end{corollary}


\begin{remark}
If we have a $G$-equivariant Kuranishi structure and a $G$-equivariant transverse perturbation $\fs$ for $\cM_{k+1}(\beta;L)$ (see \cite{G-Kuranishi}), we can form the fiber product 
\[
 \cM_{k+1}(\beta;L)^{\fs} \times_G EG(N).
\]
This was used in \cite{DF17} to directly to define $G$-equivariant Floer theory as the limit of sequence of $A_{\infty}$-algebras for $L(N)$. On the other hand, our construction will be more flexible in the sense that we do not require the use of $G$-equivariant Kuranishi structures. Instead, we will choose Kuranishi structures and transverse perturbations for the moduli spaces $\cM_{k+1}(\beta; L(N))$ in a compatible manner so that the resulting $A_{\infty}$-algebras for $L(N)$ has a well-defined limit.
\end{remark}


Our heuristic definition of a $G$-equivariant Lagrangian Floer theory is the (ordinary) Lagrangian Floer theory of the Lagrangian submanifold $L_G\subset X_G$. However, since $L_G$ is infinite dimensional, we resort to using finite dimensional approximations. Clearly, arbitrary choices of Morse-Smale pairs $\{(f_N,\mathscr{V}_N)\}_{N\in\bN}$ and Kuranishi perturbations would not suffice for this purpose. We begin with choices of Kuranishi structures on the moduli spaces $\cM_{k+1}(\beta;L(N))$. 

Let us first recall that for an element $p=[\Sigma,\vec{z},u]\in\cM_{k+1}(\beta;L)$, a Kuranishi chart around $u$ is a quintuple $(V_p,E_p,\Gamma_p,\psi_p,s_p)$, where $\Gamma_p$ is the finite automorphism group of $(\Sigma,\vec{z},u)$ acting on the vector bundle $E_p \to V_p$, $V_p$ is a smooth manifold with corners parameterizing smooth stable maps close to $u$ , $s_p=\bar{\partial}_{J_X}$ is a $\Gamma$-equivariant smooth section of $E_p \to V_p$, and $\psi_p$ is a homeomorphism from $s^{-1}_p\{0\}/\Gamma_p$ to a neighborhood of $p$.  

We fix weakly submersive Kuranishi structures for the disc moduli $\cM_{k+1}(\beta;L)$ following \cite[Chapter ~7.1]{FOOO}. For $N\ge 1$, 
since the holomorphic discs with boundary on $L(N)$ are contained in the fibers of $X(N)\to T^*BG(N)$ over $BG(N)$, there are no obstruction in the base direction, and therefore we can choose Kuranishi structures to be of the following form: For every $N\ge 1$ and $q\in BG(N)$, we choose a contractible neighborhood $U_q$ of $q$ in $BG(N)$ and local trivializations of $X(N)$ and $L(N)$ over $U_q$,  compatible with the diagram \eqref{eq:inclusions}. This in turn gives a local trivialization $\cM_{k+1}(\beta;L)\times U_q$ of the fiber bundle $\cM_{k+1}(\beta;L(N))\to BG(N)$. Then, for an element $p$ in a fiber of $\cM_{k+1}(\beta;L)\times U_q\to U_q$, we can take the (weakly submersive) Kuranishi chart around $p$ to be 
\[
(V_p\times U^{k+1}_q,E_p\oplus T_{q}BG(N)^{k},\Gamma_p,(\psi'_p,id),\bar{\partial}_{J_{X(N)}}),
\]
where $\psi'_p$ is the pullback of $\psi_p$ via the projection $s^{-1}_p(0)/\Gamma_p\times U_q^{k+1}\to  s^{-1}_p(0)/\Gamma_p$. 
Then we have 
\begin{equation}\label{equ_equasKust}
\cM_{k+1}(\beta;L(N-1))=\cM_{k+1}(\beta;L(N))\times_{L(N)} L(N-1),
\end{equation}
as Kuranishi structures, where the fiber product is taken over the evaluation map $\ev_0$ at the $0$-th marked point $z_0$ and the inclusion map.  Moreover, we can orient the moduli spaces using a $G$-invariant relative spin structure such that the orientations are compatible with the above fiber product.

We now specify our choice of Morse-Smale pairs. The choice will ensure that the Morse models obtained from applying homological perturbations are compatible. 

\begin{defn} \label{def:morse-smale}
We call a sequence of Morse-Smale pairs $\{(f_N,\mathscr{V}_N)\}_{N\in\bN}$ \emph{admissible} if it satisfies the following:	\begin{enumerate}[label=\textnormal{(\arabic*)}]
	\item  	\label{assum:1}
		For each $N\in \bN$, there is an inclusion of critical point sets $\Crit(f_{N})\subset \Crit(f_{N+1})$. Under this identification, we have 
		\begin{enumerate}[label=(\roman*)]
			\item For $p\in \Crit(f_{N})$,  $W^u(f_{N+1};p)\times_{L(N+1)}L(N)=W^u(f_{N};p)$. 
			\item For $p\in \Crit(f_{N})$, the image of $W^s(f_{N};p)$ in $L(N)$  coincides with $W^s(f_{N+1};p)$ in $L(N+1)$.
			\item For $q\in L(N+1)\setminus L(N)$, we have $\Phi_t(q)\notin L(N)$ for all $t\ge 0$. This implies 
			\[
			W^u(f_{N+1};p)\times_{L(N+1)}L(N)=\emptyset
			\]
			for $p\in\Crit(f_{N+1})\setminus\Crit(f_{N})$.
		\end{enumerate}	
			
		\item \label{assum:2} 
		For each $\ell\ge 0$, there exists an integer $N(\ell)>0$ such that $|p|>\ell$ for all $N\ge N(\ell)$ and $p\in\Crit(f_N)\setminus\Crit(f_{N-1})$.	
		
		\item 
		The Morse function $f_N$ has a unique maximal point $\bunit_{L(N)}$ and the inclusion $\Crit(f_{N})\subset \Crit(f_{N+1})$ identifies $\bunit_{L(N)}$ with $\bunit_{L(N+1)}$. 
	\end{enumerate}
\end{defn}

This allows us to identify $CF^{\bullet}(L(N);\Lambda_0)$ with a subcomplex of $CF^{\bullet}(L(N+1);\Lambda_0)$. We will denote $\bunit_{L(N)}$, $\wunit_{L(N)}$ and $\gunit_{L(N)}$ simply by $\bunit$, $\wunit$, and $\gunit$. 			
We explain below how such a sequence of Morse-Smale pairs $\{(f_N,\mathscr{V}_N)\}_{N\in\bN}$ can be produced. This is inspired by the family Morse theory of Hutchings \cite{hutchings08}. 

\begin{prop} \label{prop:admissible}
There exists a finite dimensional approximation of $BG$ by smooth manifolds $BG(N)$ which has an admissible sequence of Morse-Smale pairs.
\end{prop}

\begin{proof}
We choose $E(U(k))$ to be the infinite Stiefel manifold $V_k(\C^{\infty})=\lim_{\to} V_k(\C^{k+N}) $ and $B(U(k))$ to be the infinite Grassmannian $Gr(k,\C^{\infty})=\lim_{\to} Gr(k,\C^{k+N})$. 
We have an embedding of $Gr(k,\C^{\infty})$ into the skew-Hermitian matrices on $\C^{\infty}$ by identifying $V\in Gr(k,\C^{\infty})$ with the orthogonal projection $P_V$ onto $V$. Let $A$ be the diagonal matrix with entries $\{1,2,3,\ldots\}$. The map 
\begin{equation} \label{eq:perfect_Gr}
f_{Gr(k,\C^{\infty})} \colon Gr(k,\C^{\infty})\to \R, \quad f_{Gr(k,\C^{\infty})}(V)= -\re (\tr(AP_V)),
\end{equation}
is a perfect Morse function (see \cite{nicolaescu94}). We note that $f_{Gr(k,\C^{\infty})}$ has the following properties:
\begin{enumerate}
	\item The restriction 
	\[
	f_{Gr(k,\C^{k+N})}=f_{Gr(k,\C^{\infty})}|_{Gr(k,\C^{k+N})}
	\]
	of $f_{Gr(k,\C^{\infty})}$ to each finite dimensional stratum is again a perfect Morse function.
	
	\item 	For all $p\in Gr(k,\C^{k+N})$ and $q\in Gr(k,\C^{k+N+1})\setminus Gr(k,\C^{k+N})$, we have 
	\[
	f_{Gr(k,\C^{k+N+1})}(q)\le f_{Gr(k,\C^{k+N})}(p).
	\]
	
	\item If a critical point $p$ of $f_{Gr(k,\C^{\infty})}$ is contained in $Gr(k,\C^{\infty})\setminus Gr(k,\C^{k+N})$, then $|p|> 2N$.
\end{enumerate}

We fix an embedding $G\to U(k)$ of Lie groups for some $k$, and choose   $EG=V_k(\C^{\infty})$ and $BG=V_k(\C^{\infty})/G$. This gives a fibration $\pi \colon BG\to Gr(k,\C^{\infty})$ with fibers the homogeneous space $U(k)/G$. Let $\{U_i\}_{i\in I}$ be an open cover for $Gr(k,\C^{\infty})$ such that each $U_i$ is contractible and $p\notin \overline{U}_i$ whenever $p\in \Crit(f_{Gr(k,\C^{\infty})})$ and $p\notin U_i$. Let $\{\rho_i\}_{i\in I}$ be a partition of unity subordinate to $\{U_i\}_{i\in I}$. Notice that $\rho_i$ are constant near the critical points of $f_{Gr(k,\C^{\infty})}$ due to our second assumption on the open cover. 

Let $f_{U(k)/G} \colon U(k)/G\to \R$ be a Morse function with a unique maximal point and define $F_{U(k)/G} \colon BG\to \R$ by
\[
F_{U(k)/G} = \sum_{i\in I} \rho_i f_{U(k)/G}.
\]
Here $\rho_i$ and $f_{U(k)/G}$ are understood as defined on the local trivializations $U(k)/G\times U_i$. 
For each $N\in\bN$, the restriction of $F_{U(k)/G,N}:=F_{U(k)/G}|_{BG(N)}$ to generic fibers of $BG(N)\to Gr(k,\C^{k+N})$ is a Morse function. In particular, the restriction of $F_{U(k)/G,N}$ to each fiber over a critical point of $f_{Gr(k,\C^{k+N})}$ agrees with $f_{U(k)/G}$.

We set 
\[
f_{BG(N)}=\pi^*f_{Gr(k,\C^{k+N})}+\epsilon F_{U(k)/G,N}, \quad \epsilon>0.
\]
Since $Gr(k,\C^{k+N})$ is compact, for sufficiently small $\epsilon$, $f_{BG(N)}$ has the critical point set 
\[
\Crit(f_{BG(N)})=  \Crit(f_{U(k)/G}) \times \Crit(f_{Gr(k,\C^{k+N})}) .
\]
Moreover, $f_{BG(N)}$ is a Morse function. The non-degeneracy of critical points follows from the fact that $\rho_i$ are constant near the critical points of $f_{Gr(k,\C^{\infty})}$. 

Now, let $V_{Gr(k,\C^{\infty})}$ be a negative pseudo-gradient vector field for $f_{Gr(k,\C^{\infty})}$ such that $(f_{Gr(k,\C^{\infty})},V_{Gr(k,\C^{\infty})})$ is a Morse-Smale pair, and let $V_{U(k)/G}$ be a fiberwise negative pseudo-gradient for $F_{U(k)/G}$. We set
\[
V_{BG}= V_{U(k)/G}+HV_{Gr(k,\C^{\infty})},
\] 
where $H$ denotes the horizontal lift with respect to an auxiliary connection for the fiber bundle $BG\to Gr(k,\C^{\infty})$. For each $N\in \bN$, we put $V_{BG(N)}=V_{BG}|_{BG(N)}$. Then, for a generic choice of $V_{U(k)/G}$, the pairs $\{(f_{BG(N)},V_{BG(N)})\}_{N\in \bN}$ are Morse-Smale.

Finally, let $f \colon L\to \R$ be a Morse function with a unique maximal point. By iterating the construction in the above paragraphs, we obtain Morse-Smale pairs $\{(f_N,\mathscr{V}_N)\}_{N\in \bN}$ with 
\begin{equation}\label{equ_prodcrit}
\Crit(f_{N})=\Crit(f)\times \Crit(f_{BG(N)}). 
\end{equation}
It follows from the properties of $f_{Gr(k,\C^{\infty})}$ that $\{(f_N,\mathscr{V}_N)\}_{N\in\bN}$ is an admissible collection in the sense of Definition \ref{def:morse-smale}.
\end{proof}



The next key proposition enables us to define the $A_{\infty}$-structure maps of $G$-equivariant Lagrangian Floer theory.

\begin{prop}
	\label{prop:extension}
	There exists a sequence of singular chain models $(C^{\bullet}(L(N);\Lambda_0)^+,\tilde{\fm}^{N,+})$ and a cochain map 
	\begin{equation} \label{eq:chain_projection}
	\pi_{N} \colon C^{\bullet}(L(N);\Lambda_0)^+ \to C^{\bullet}(L(N-1);\Lambda_0)^+
	\end{equation}
	satisfying the following property:
	\begin{equation}
	\label{eq:G_compatibility}
	\pi_{N}\left(\mathfrak{m}^{N}_{k,\beta}(P_1,\ldots,P_k)\right)=\fm^{N-1}_{k,\beta}(\pi_{N}(P_1),\ldots,\pi_{N}(P_k))
	\end{equation}
for $P_1,\ldots,P_k\in C^{\bullet}(L(N);\Lambda_0)^+$ and $\beta\in  H_2^{\mathrm{eff}}(X,L)$.
\end{prop}

\begin{proof}	
	We proceed by induction on $N$. When $N=0$, the statement of this proposition is void. Suppose that we have constructed $A_{\infty}$-algebras $(C^{\bullet}(L(N');\Lambda_0)^+,\tilde{\fm}^{N',+})$ and maps $\pi_{N'}$ for all $N' < N$ obeying~\eqref{eq:G_compatibility}. 
	
	We shall construct a complex $(C^{\bullet}(L(N);\Lambda_0)^+,\tilde{\fm}^{N,+})$ and a cochain map $\pi_N$ satisfying~\eqref{eq:G_compatibility} inductively on $g \ge -1$. More precisely, at the end of the induction process, we choose perturbations $\fs_{\fd,\beta,\ell,\vec{P}_{N}}$ with $\norm{(\fd,\beta)}=g$ and construct the set of singular chains $\mathcal{X}_{g}(L(N))$ in such a way that we have a map 
	\[
	\pi_{N,g} \colon \mathcal{X}_{g}(L(N))\to \mathcal{X}_{g}(L(N-1))
	\] 
	satisfying 
	\begin{equation} \label{eq:chain_map_proj}
	\pi_{N,g}(\partial P) =\partial 	\pi_{N,g}(P)
	\end{equation}
    for $P\in \mathcal{X}_{g}(L(N))$, and 
	\begin{equation} \label{eq:chain_compatibility}
	\pi_{N,g}\left((\ev_0)_*\left(\mathcal{M}_{\ell+1}(\beta;L(N);\vec{P}_{N})^{\fs_{\fd,\beta,\ell,\vec{P}_{N}}}\right) \right) = (\ev_0)_*\left(\mathcal{M}_{\ell+1}(\beta;L(N-1);\vec{P}_{N-1})^{\fs_{\fd,\beta,\ell,\vec{P}_{N-1}}}\right),
	\end{equation}
	for $\vec{P}_{N}=(P_{1,N},\ldots,P_{\ell,N})$, $P_{i,N}\in\mathcal{X}_{\fd(i)}(L(N))$, $\fd(i)< g$, $\vec{P}_{N-1}=(P_{1,N-1},\ldots,P_{\ell,N-1})$, and $P_{i,N-1}=\pi_{N, \fd(i)}(P_{i,N}) \in \mathcal{X}_{\fd(i)}(L(N-1))$. 
	
	For the base step $g=-1$, we start with a set of singular chains $\mathcal{X}_{-1}(L(N))$ satisfying 
	\[
	P\times_{L(N)} L(N-1)\in \mathcal{X}_{-1}(L(N-1)), 
	\]
	for all $P\in \mathcal{X}_{-1}(L(N))$. As an induction hypothesis, suppose that we have constructed $\fs_{\fd,\beta,\ell,\vec{P}_{N}}$ with $\norm{(\fd,\beta)}=g'$,  $\mathcal{X}_{g'}(L(N))$, and maps $\pi_{N,g'}$ satisfying \eqref{eq:chain_compatibility} for all $g'<g$. For  $\vec{P}_{N}=(P_{1,N},\ldots,P_{\ell,N})$, $P_{i,N}\in\mathcal{X}_{\fd(i)}(L(N))$, and $\beta\in  H_2^{\mathrm{eff}}(X,L)$ with $\norm{(\fd,\beta)}=g$, we have 
	\[
	\mathcal{M}_{\ell+1}(\beta;L(N);\vec{P}_N)\times_{L(N)} L(N-1)=\mathcal{M}_{\ell+1}(\beta;L(N-1);\vec{P}_{N-1})
	\]
	as compact subsets of $\mathcal{M}_{\ell+1}(\beta;L(N);\vec{P}_N)$. As in~\eqref{equ_equasKust}, we adorn Kuranishi structure on $\mathcal{M}_{\ell+1}(\beta;L(N))$ such that the above equality holds as Kuranishi structures.  
	
	We then choose a perturbation $\fs_{\fd,\beta,\ell,\vec{P}_{N}}$ by extending the perturbation $\fs_{\fd,\beta,\ell,\vec{P}_{N-1}}$ and the perturbation for the boundary strata $\partial\mathcal{M}_{\ell+1}(\beta;L(N);\vec{P}_{N})$. In order to define the map $\pi_{N,g}$, we add the following singular chains from  $\mathcal{M}_{\ell+1}(\beta;L(N);\vec{P}_{N})^{\fs_{\fd,\beta,\ell,\vec{P}_{N}}}$ to be elements of $\mathcal{X}_{g}(L(N))$: Set $d_n :=\dim_{\R}L(N)-\dim_{\R}L(N-1)$. For every simplex $\tau_{N-1}$ in $\mathcal{M}_{\ell+1}(\beta;L(N-1);\vec{P}_{N-1})^{\fs_{\fd,\beta,\ell,\vec{P}_{N-1}}}$, we have $\tau_N=[0,1]^{d_n}\times \tau_{N-1}$ in the tubular neighborhood of $\mathcal{M}_{\ell+1}(\beta;L(N-1);\vec{P}_{N-1})^{\fs_{\fd,\beta,\ell,\vec{P}_{N-1}}}$, such that the fiber product of $\tau_N \times_{L(N)} L(N-1)$ is of the form $(t_1,\ldots,t_{d_n})\times \tau_{N-1}$, where $(t_1,\ldots,t_{d_n})\in [0,1]^{d_n}$ is an interior point. We triangulate $\tau_N$ to regard it as a singular chain, then add $\tau_N$ and its faces to $\mathcal{X}_{g}(L(N))$. Note that we do not add the individual simplices on $\tau_N$. We then triangulate $\mathcal{M}_{\ell+1}(\beta;L(N);\vec{P}_{N})^{\fs_{\fd,\beta,\ell,\vec{P}_{N}}}$ extending the triangulation from the tubular neighborhood of $\mathcal{M}_{\ell+1}(\beta;L(N-1);\vec{P}_{N-1})^{\fs_{\fd,\beta,\ell,\vec{P}_{N-1}}}$ and the boundary strata, and add the remaining simplices to $\mathcal{X}_{g}(L(N))$. (We also extend perturbations for moduli spaces \eqref{eq:homotopy_moduli} which appear in the homotopy unit construction and add singular chains on it to $\mathcal{X}_{g}(L(N))$.) 
	
	For a singular chain $\sigma \in \mathcal{X}_{g}(L(N))$, we have either $\sigma\times_{L(N)} L(N-1)=\emptyset$ or $\sigma=[0,1]^{d_n}\times \tau$ and $\sigma\times_{L(N)} L(N-1)=\tau$ for some simplex $\tau\in \mathcal{X}_{g}(L(N-1))$. Thus, we can define the map $\pi_{N,g}$ by the fiber product 
	\[
		\pi_{N,g}(\sigma)= \sigma \times_{L(N)} L_{N-1}.
	\]
	By our construction, $\pi_{N,g}$ satisfies the properties \eqref{eq:chain_map_proj} and \eqref{eq:chain_compatibility}, and extends to a cochain map
	\[
	 \pi_{N,g} \colon C^{\bullet}_{(g)}(L(N);\Lambda_0)\to C^{\bullet}_{(g)}(L(N-1);\Lambda_0).
	\]
The induction on $g$ gives us a singular chain model $(C^{\bullet}(L(N);\Lambda_0)^+,\tilde{\fm}^{N,+})$ and a cochain map $\pi_N$ satisfying \eqref{eq:G_compatibility}. The map $\pi_N$ is defined by $\pi_{N,g}$ on the component $C^{\bullet}(L(N);\Lambda_0)$, and identity on the component $\Lambda_0 \cdot\bm{e}^+ \oplus \Lambda_0 \cdot\bm{f}$.
\end{proof}

We are now ready to define the $G$-equivariant Lagrangian Floer theory of $L$. The equivariant Floer complex $C^{\bullet}_G(L;\Lambda_0)$ is defined to be the inverse limit
\begin{equation}\label{eq_inverse_limit}
C^{\bullet}_G(L;\Lambda_0)=\lim_{\leftarrow} C^{\bullet}(L(N);\Lambda_0)^+
\end{equation}
along the cochain maps $\pi_N$ in~\eqref{eq:chain_projection}. By definition, an element in $C^{\bullet}_G(L;\Lambda_0)$ can be expressed as a sequence $(P_N)_{N\in \bN}\in \prod_{N\in \bN } C^{\bullet}(L(N);\Lambda_0)^+$ satisfying $\pi_N(P_N)=P_{N-1}$.
The $A_{\infty}$-structure maps $\tilde{\fm}^G=\{\tilde{\fm}^G_k\}$ on $C^{\bullet}_G(L;\Lambda_0)$ are defined as follows: For $(P_{N,1}), \ldots, (P_{N,k})\in C^{\bullet}_G(L;\Lambda_0)$, we have 
\[
\tilde{\fm}^G_k((P_{N,1}), \ldots, (P_{N,k}))= (\tilde{\fm}^{N,+}_k(P_{N,1},\ldots,P_{N,k}))_{N\in \bN }.
\]
Proposition \ref{prop:extension} ensures that the RHS is an element in $C^{\bullet}_G(L;\Lambda_0)$. It is straightforward to check that $(C^{\bullet}_G(L;\Lambda_0), \tilde{\fm}^G)$ is a unital $A_{\infty}$-algebra with the strict unit $(\bm{e}^+,\bm{e}^+,\ldots)$.

\begin{defn}
We refer to $(C^{\bullet}_G(L;\Lambda_0), \tilde{\fm}^G)$ as the \textit{G-equivariant singular chain model} of the pair $(L,G)$.
\end{defn}

We now define a $G$-equivariant Morse model, which is more amenable for the computations in the later sections. Let $\{(f_N,\mathscr{V}_N)\}_{N\in\bN}$ be an admissible sequence of Morse-Smale pairs in Definition~\ref{def:morse-smale}, and let $(CF^{\bullet}(L(N);\Lambda_0),\fm^N)$ be the unital Morse model obtain from the unital singular chain model $(C^{\bullet}(L(N);\Lambda_0)^+,\tilde{\fm}^{N,+})$ via homological perturbation. We have projection maps (by abuse of notation)
\[
\pi_N \colon CF^{\bullet}(L(N)) \to CF^{\bullet}(L(N-1)) 
\]
defined by $\pi_N(p)=p$ if $p$ is also a critical point of $f_{N-1}$ and $\pi_N(p)=0$, otherwise.
By setting 
\[
\mathcal{X}_{-1}(L(N))=\{\Delta_p \mid p\in\Crit(f_N)\}, \quad N\in\bN, 
\]
and including additional singular chains coming from the homological perturbation to $\mathcal{X}_{g}(L(N))$ in the inductive process of the proof of Proposition \ref{prop:extension}, the $A_{\infty}$-structure maps satisfy
\begin{equation} \label{eq:morse_extension}
\pi_N(\fm^N_k(p_1,\ldots,p_k))= \fm^{N-1}_k(  \pi_N(p_1),\ldots,\pi_N(p_k)), 
\end{equation}
for $p_1,\ldots,p_k\in  CF^{\bullet}(L(N);\Lambda_0)$. 

For the $G$-equivariant Morse model, instead of defining it to be the $A_{\infty}$-algebra structure on the inverse limit 
\[
\lim_{\leftarrow} CF^{\bullet}(L(N);\Lambda_0), 
\] 
we define $CF^{\bullet}_G(L;\Lambda_0)$ to be the submodule generated by elements $(p^N)_{N\in \bN}$ such that there exists $N_0\in \bN$ for which  $p^{N}=p^{N+1}$ for sufficiently large $N\ge N_0$. This allows us to identify an element in $CF^{\bullet}_G(L;\Lambda_0)$ with a finite $\Lambda_0$-linear combination of critical points of a Morse function $f_N$ for sufficiently large $N$. (We found this definition to be more convenient for applications, it is analogous to the choice of defining  $H^{\bullet}(\C\bP^{\infty};R)$ to be $R[\lambda]$ instead of $R[[\lambda]]$. )


The $A_{\infty}$-structure maps $\mathfrak{m}^G=\{\fm^G_k\}$ on $CF^{\bullet}_G(L;\Lambda_0)$ are defined in term of the maps  $\{\mathfrak{m}^G_{\Gamma}\}$ indexed by decorated rooted trees $\Gamma \in \bm{\Gamma}_{k+1}$ for $k\ge 0$.
For $p_1,\ldots,p_k\in  CF^{\bullet}_G(L;\Lambda_0)$, there exists a sufficiently large $N_0 \in \mathbb{N}$ such that $p_1,\ldots,p_k\in CF^{\bullet}(L(N_0);\Lambda_0)$. Setting
\begin{equation} \label{eq:output_deg}
\ell=\sum_{i=1}^k |p_i|+2-k-\mu(\Gamma),
\end{equation} 
we define $\mathfrak{m}^G_{\Gamma}$ by 
\begin{equation}
\label{eq:m^G_gamma}
\mathfrak{m}^G_{\Gamma}(p_1,\ldots,p_k)=\fm^{N(\ell)}_{\Gamma}(p_1,\ldots,p_k)
\end{equation}
where $N(\ell) (\ge N_0)$ comes from Definition~\ref{def:morse-smale}~\ref{assum:2}.
By Proposition~\ref{prop:extension} and \eqref{eq:morse_extension}, (\ref{eq:m^G_gamma}) is defined independent to the choice of $N(\ell)$. Moreover, the degree of \eqref{eq:m^G_gamma} is $\ell$ as defined in \eqref{eq:output_deg}. Finally, we define the maps $\mathfrak{m}^G_{k} \colon CF_G^{\bullet}(L;\Lambda_0)^{\otimes k}\to CF_{G}^{\bullet}(L;\Lambda_0)$ by
\begin{equation}
\label{eq:m^G_beta}
\mathfrak{m}^G_{k,\beta}=\sum_{\substack{\Gamma\in\bm{\Gamma}_{k+1},\beta=\sum \beta_{\eta(v)}}} \bT^{\omega(\beta)}\mathfrak{m}^G_{\Gamma},
\end{equation}
\begin{equation}
\label{eq:m^G}
\mathfrak{m}^G_{k}=\sum_{\beta} \mathfrak{m}^G_{k,\beta}.
\end{equation}
It is straightforward to see that $(CF^{\bullet}_G(L;\Lambda_0), \mathfrak{m}^G)$ is a unital $A_{\infty}$-algebra with a strict unit $\wunit$. 

\begin{defn} \label{def:G_Morse_model}
We refer to $(CF^{\bullet}_G(L;\Lambda_0), \mathfrak{m}^G)$ as the \emph{$G$-equivariant Morse model} of the pair $(L,G)$.
\end{defn}

\subsection{Equivariant parameters as homotopy partial units}
\label{sec:partial_units}

Consider a Lagrangian toric fiber $L$  of a compact semi-Fano toric manifold of the complex $d$ dimension. 
One of the main goals of this paper is to understand Givental's equivariant toric superpotential
\begin{equation}\label{equ_Giventalsuper}
W_{\lambda}=W^{\mathrm{toric}}+\sum_{i=1}^d \lambda_i x_i,
\end{equation}
as the disc potential of the $T$-equivariant Lagrangian Floer theory of $L$ constructed in Section~\ref{sec:G Morse model}. Here $W^{\mathrm{toric}}$ is the superpotential of Givental and Hori-Vafa \cite{Givental,HV,MS_big_book}. The terms $\lambda_1,\ldots,\lambda_d$ are the equivariant parameters generating the cohomology ring
\[
H^{\bullet}_T(\mathrm{pt})=H^{\bullet}(BT;\C)=\C[\lambda_1,\ldots,\lambda_d].
\]
Since $\lambda_1,\ldots,\lambda_d$ have cohomological degree $2$, the expression~\eqref{equ_Giventalsuper} of $W_{\lambda}$ suggests that the boundary deformations of curvature $\fm_0^T(1)$ are, a priori, \emph{obstructed}. For this reason, we shall construct in this section an alternative model which is homotopy equivalent to our equivariant Morse model and define its equivariant disc potential. Our construction in this section replies on the existence of a well-behaved perfect Morse function on $BG$ (e.g. \eqref{eq:perfect_Gr} for $G=U(k)$). For this reason, we will temporarily restrict the discussion to the case when $G$ is a product of unitary groups. 

To begin with, we fix a Morse function $f$ on $L$ with a unique maximum point $\bunit_L$. Following the proof of Proposition \ref{prop:admissible}, we can choose an admissible sequence of Morse-Smale pairs $\{(f_N,\mathscr{V}_N)\}_{N\in \bN}$ such that $f_N$ is of the form
\begin{equation}\label{eq_fN}
f_N=\pi_N^{*}\varphi_N+\phi_N
\end{equation}
where 
\begin{itemize}
\item $\varphi_N$ is a perfect Morse function on $BG(N)$. 
\item $\phi_N$ is a (generically) fiberwise Morse function for $L(N)\to BG(N)$.  In particular the restriction of $\phi_N$ to each fiber over a critical point of $\varphi_N$ agrees with $\epsilon f$ for some $\epsilon>0$. 
\end{itemize}
Since the choice~\eqref{eq_fN} satisfies~\eqref{equ_prodcrit}, we have
\begin{equation*}\label{equ_fgML0}
C^{\bullet}(f_N;\Lambda_0)=C^{\bullet}(f;\Lambda_0)\otimes_{\Lambda_0} H^{\bullet}(BG(N);\Lambda_0).
\end{equation*}
We then enlarge $C^{\bullet}(f_N;\Lambda_0)$ fiberwise over $\Crit(\varphi_N)$ to
\[
CF^{\bullet}(L(N);\Lambda_0)^{\dagger}=CF^{\bullet}(L;\Lambda_0)\otimes_{\Lambda_0} H^{\bullet}(BG(N);\Lambda_0),
\]
where
\[
CF^{\bullet}(L;\Lambda_0)=C^{\bullet}(f;\Lambda_0)\oplus \Lambda_0 \cdot \wunit_L \oplus  \Lambda_0 \cdot \gunit_L
\]
as in \eqref{equ_CFLlam0}. 

In this section, we will construct a unital $A_\infty$-algebra structure $\mathfrak{m}^{G,\dagger}$ on
\begin{equation}\label{equcffldag}
CF^{\bullet}_G(L;\Lambda_0)^{\dagger}=CF^{\bullet}(L;\Lambda_0)\otimes_{\Lambda_0} H^{\bullet}_G(\mathrm{pt};\Lambda_0),
\end{equation}
such that the $A_\infty$-structure maps are  $H^{\bullet}_G(\mathrm{pt};\Lambda_0)$-multilinear. 

We begin with the singular chain model from which our Morse model is derived from via homological perturbation. Let us set
\[
C^{\bullet}(L(N);\Lambda_0)^{\dagger}=C^{\bullet}(L(N);\Lambda_0)\oplus \left(\bigoplus_{\lambda\in \Crit (\varphi_N)} (\Lambda_0 \cdot \wlambda\oplus \Lambda_0 \cdot \glambda)\right).
\]
Here, the singular chains in $\chi_{-1}(L(N))\subset C^{\bullet}(L(N);\Lambda_0)$ are the unstable chains of the critical points of $f_N$. We denote by $\blambda$ the unstable chain of the critical point $(\bunit_L,\lambda)$. The generators $\wlambda$ and $\glambda$ are of degrees $|\wlambda|=|\blambda|$ and $|\glambda|=|\blambda|-1$. We note that $\blambda$ and $\wlambda$ are of even degrees while $\glambda$ is of odd degree.

By using the idea of the homotopy unit construction in Section~\ref{sec:Morse model}, we construct the following interpolating perturbations, which we use to define a unital $A_{\infty}$-algebra $(C^{\bullet}(L(N);\Lambda_0)^{\dagger}, \tilde{\fm}^{N,\dagger})$ with additional properties.

Let $\vec{a}=(a_1,\ldots,a_{|\vec{a}|})$ be an ordered subset of $\{1,\ldots,k\}$ satisfying $a_1< \ldots < a_{|\vec{a}|}$.
For a $(k-|\vec{a}|)$-tuple of singular chains $\vec{P}=(P_1,\ldots,P_{k-|\vec{a}|})$,
let $\vec{P}^{\dagger}$ be the $k$-tuple obtained by inserting $\blambda$ into the $\vec{a}$-th places of $\vec{P}$ (by abuse of notation, $\blambda$'s inserted can be distinct). We choose a perturbation on $[0,1]^{|\vec{a}|}\times \cM_{k+1}(\beta;L(N);\vec{P}^{\dagger})$ as follows: For a splitting $\vec{a}^1\coprod \vec{a}^2=\vec{a}$ of $\vec{a}$, 
we denote by $\vec{P}'$ the $(k-|\vec{a}^1|)$-tuple given by removing $\blambda$ from the $\vec{a}^1$-th places of $\vec{P}^{\dagger}$. Let $(t_1,\ldots,t_{|\vec{a}|})$ be coordinates on $[0,1]^{|\vec{a}|}$. If $t_i=0$ for all $i\in\vec{a}^1$, we consider the chosen perturbation in Proposition~\ref{prop:extension} on the moduli space obtained by just inserting $\blambda$ into the $\vec{a}^1$-th place. On the other hand, If $t_i=1$ for all $i\in\vec{a}^1$, we take a perturbation pulled back via $\forget_{\vec{a}_1} \colon \cM_{k+1}(\beta;L(N);\vec{P}^{\dagger})\to \cM_{k+1-|\vec{a}^1|}(\beta;L(N);\vec{P}')$. Finally, we take a perturbation $\fs^{\dagger}$ on $[0,1]^{|\vec{a}|}\times \cM_{k+1}(\beta;L(N);\vec{P}^{\dagger})$ transversal to the zero-section  interpolating between them. We note that the unstable submanifold $W^u(f_N,\blambda)$ is the restriction of the fiber bundle $L_N\to BG(N)$ over the unstable submanifold  $W^u(\varphi_N,\lambda)$ in $BG(N)$. Inserting $\blambda$ imposes no constraint in the fiber direction. Therefore it makes sense to pullback perturbations from  $\cM_{k+1-|\vec{a}^1|}(\beta;L(N);\vec{P}')$ to $\cM_{k+1}(\beta;L(N);\vec{P}^{\dagger})$. 

For $\lambda, \lambda' \in \Crit(\varphi_N)$, we denote by $\lambda\cup\lambda'\in \Crit(\varphi_N)$ the critical point representing the cup product of $\lambda$ and $\lambda'$ in $H^{\bullet}(BG(N))$. We choose the perturbations for moduli spaces of the form
\begin{equation} \label{eq:cup_product}
[0,1]\times \cM_{3}(\beta_0; \blambda, \Delta_{(p,\lambda')}), \qquad [0,1]\times \cM_{3}(\beta_0;\Delta_{(p,\lambda')},\blambda),
\end{equation}
to be interpolating between perturbations for $\cM_{3}(\beta_0; \blambda, \Delta_{(p,\lambda')})$ (resp. $\cM_{3}(\beta_0;\Delta_{(p,\lambda')},\blambda)$) and $\Delta_{(p,\lambda\cup \lambda')}$ (resp. $(-1)^{|p|} \Delta_{(p,\lambda\cup \lambda')}$).

The $A_{\infty}$-structure maps $\tilde{\fm}^{N,\dagger}=\{\tilde{\fm}^{N,\dagger}_k\}_{k\ge 0}$ are defined as follows:
\begin{itemize}


\item The restriction of $\tilde{\fm}^{N,\dagger}$ to $C^{\bullet}(L(N);\Lambda_0)^+$ agrees with $\tilde{\fm}^{N,+}$. In particular, this means if $\lambda\in \Crit(\varphi_N)$ is the maximum point, then $\blambda=\bm{e}$ is the homotopy unit, $\wlambda=\bm{e}^+$ is the strict unit, and $\glambda=\bm{f}$ is the homotopy between them.

\item We set $\tilde{\fm}^{N,\dagger}_{1}(\wlambda)=0$ and $\tilde{\fm}^{N,\dagger}_{1,\beta_0}(\glambda)=\wlambda-\blambda$.

\item The $\tilde{\fm}^{N,\dagger}_{k,\beta}$ maps for $(k,\beta) \ne (1,\beta_0)$ with inputs $\glambda$ inserted into $\vec{a}$-th place of $\vec{P}$ are defined by
\begin{equation}
	\label{eq:m++lambda}
	(\ev_0)_*\left(\left([0,1]^{|\vec{a}|}\times \cM_{k+1}(\beta;L(N);\vec{P}^{\dagger})\right)^{\fs^{\dagger}}/\sim\right)
\end{equation}
where $\sim$ is again the equivalence relation collapsing fibers of $\forget_{\vec{a}_1}$. We note that the zero locus of a pullback multisection is a degenerate singular chain which becomes zero in the quotient. 
\end{itemize}

By our construction, $\tilde{\fm}^{N,\dagger}$ has the following additional properties:
\begin{enumerate}
\item $\tilde{\fm}^{N,\dagger}_{k}(\ldots,\wlambda,\ldots)=0$ for $k\ge 3$.
\item We have 
	$\tilde{\fm}^{N,\dagger}_{1}(\glambda)=\wlambda-\blambda+\bm{h}_{\lambda}$, 
where
\begin{equation}
	\label{eq:fh-lambda}
	\bm{h}_{\lambda}=\sum_{\beta\ne \beta_0} \bm{T}^{\omega_X(\beta)} (\ev_0)_* \left(([0,1]\times \cM_2(\beta;L(N);\blambda))^{\fs^{\dagger}}/\sim \right),
\end{equation}
and $\bm{h}_{\lambda} \equiv 0 \mod \Lambda_+\cdot C^{\bullet}(L(N);\Lambda_0)^{\dagger}$.
\item $\tilde{\fm}^{N,\dagger}_{2,\beta}(\wlambda,P)=\tilde{\fm}^{N,\dagger}_{2,\beta}(P,\wlambda)=0$ for $\beta \ne \beta_0$.

\item $\tilde{\fm}^{N,\dagger}_{2,\beta_0}(\wlambda, \Delta_{(p,\lambda')})= \Delta_{(p,\lambda\cup \lambda')}= (-1)^{|p|} \tilde{\fm}^{N,\dagger}_{2,\beta_0} (\Delta_{(p,\lambda')},\wlambda)$.
\end{enumerate}

We note that properties (1), (2) and (3) above are consequences of \eqref{eq:m++lambda} similar to properties of the original homotopy unit, and (4) follows from the choices of perturbations for the moduli spaces in \eqref{eq:cup_product}.

We apply homological perturbation to obtain a Morse model $(CF^{\bullet}(L(N);\Lambda_0)^{\dagger},\fm^{N,\dagger})$ from the singular chain model $(C^{\bullet}(L(N);\Lambda_0)^{\dagger},\tilde{\fm}^{N,\dagger})$ using the strong contraction (Definition \ref{def:strong_contraction}) constructed in Section \ref{sec:Morse model}. For the homological perturbation to preserve properties (1), (3), (4) above, we require the homotopy operator $G^+$ to satisfy 
\begin{equation} \label{eq:G_and_lambda}
G^+(\tilde{\fm}^{N,\dagger}_{2,\beta_0}(\wlambda, P))=  \tilde{\fm}^{N,\dagger}_{2,\beta_0}(\wlambda, G^+(P)), 
\end{equation} 
\begin{equation}  \label{eq:G_and_lambda_2}
G^+(\tilde{\fm}^{N,\dagger}_{2,\beta_0}(P,\wlambda))=  \tilde{\fm}^{N,\dagger}_{2,\beta_0}(G^+(P),\wlambda).
\end{equation} 
\eqref{eq:G_and_lambda}, \eqref{eq:G_and_lambda_2} together with $(G^+)^2=0$ imply that
\begin{equation} \label{eq:perserve_partial_unit}
	G^{\dagger}(\tilde{\fm}^{N,\dagger}_{2,\beta_0}(\wlambda,G^{\dagger}(P))=G^{\dagger}(\tilde{\fm}^{N,\dagger}_{2,\beta_0}(G^{\dagger}(P),\wlambda))=0.
\end{equation} 
\eqref{eq:G_and_lambda}, \eqref{eq:G_and_lambda_2} and \eqref{eq:perserve_partial_unit} ensure that 
\[
\fm^{N,\dagger}_{\Gamma}(\ldots,\wlambda,\ldots)=0,
\]
unless $\Gamma$ is the tree with two inputs and only one interior vertex decorated by the constant disc class. 

We note that the singular chains in both sides of \eqref{eq:G_and_lambda} and \eqref{eq:G_and_lambda_2} are determined by the perturbations chosen for the moduli spaces,
\begin{equation*}
 \cM_{3}(\beta_0; \blambda, P), \quad   \cM_{3}(\beta_0;\blambda,G^+(P)), \quad
 \cM_{3}(\beta_0; P,\blambda), \quad   \cM_{3}(\beta_0;G^+(P),\blambda).
\end{equation*}
The perturbation for a moduli space $\cM_{3}(\beta_0; P,Q)$ is given by composing the input evaluation maps $\ev_1$ and $\ev_2$ with a small diffeomorphism of $L_N$ such that the fiber product with singular chains $P$ and $Q$ are transverse. By choosing the same diffeomorphisms for  $\cM_{3}(\beta_0; \blambda, P)$ and $\cM_{3}(\beta_0;\blambda,G^+(P))$ (resp.  $\cM_{3}(\beta_0; P,\blambda)$ and $\cM_{3}(\beta_0;G^+(P),\blambda)$), the homotopy operator $G^+$ satisfies \eqref{eq:G_and_lambda} and \eqref{eq:G_and_lambda_2}. 

Finally, following the inductive procedure in Section \ref{sec:G Morse model}., we obtain an $A_{\infty}$-algebra $(CF^{\bullet}_G(L;\Lambda_0)^{\dagger},\mathfrak{m}^{G,\dagger})$.
We will abuse notation and identify elements of $CF^{\bullet}_G(L;\Lambda_0)^{\dagger}$ with their image in $C^{\bullet}(L(N);\Lambda_0)^{\dagger}$. 
For $\lambda\in H^{\bullet}_G(\mathrm{pt};\Lambda_0)$, we put 
\[
\blambda=\bunit_L\otimes \lambda, \, \wlambda=\wunit_L\otimes \lambda,  \mbox{ and } \,  \glambda=\gunit_L\otimes \lambda.
\]
In particular, when $\lambda=1$, we have
\[
\bunit=\bunit_L\otimes 1, \, \wunit=\wunit_L\otimes 1,  \mbox{ and } \,  
\gunit=\gunit_L\otimes 1.
\]
By our construction, $(CF^{\bullet}_G(L;\Lambda_0)^{\dagger}, \mathfrak{m}^{G,\dagger})$ has the following properties: 
\begin{enumerate}
\item The restriction of $\mathfrak{m}^{G,\dagger}$ to $CF^{\bullet}_G(L;\Lambda_0)$ coincides with $\mathfrak{m}^{G}$. In particular, $\bunit$ is the homotopy unit, $\wunit$ is the strict unit, and $\gunit$ is the homotopy between them.

\item For $k \ne 2$,
\begin{equation}
	\label{eq:dagger_2}
	\mathfrak{m}^{G,\dagger}_k(\ldots,\wlambda,\ldots)=0.
\end{equation}

\item For $X=x\otimes \lambda'\in CF^{\bullet}_G(L;\Lambda_0)^{\dagger}$, 
let's put
\begin{equation*}
\lambda\cdot X=x \otimes (\lambda\cup\lambda').
\end{equation*}
Then we have
\begin{equation}
\label{eq:dagger_1}
\mathfrak{m}^{G,\dagger}_2(\wlambda,{X})=\lambda \cdot X =(-1)^{|X|}\mathfrak{m}^{G,\dagger}_{2}({X},\wlambda)
\end{equation}

\end{enumerate}
When the minimal Maslov index of $L$ is nonnegative, the elements $\gunit$ and $\glambda$ satisfy
\begin{equation}
\label{eq:coh_unit}
\mathfrak{m}^{G,\dagger}_1(\gunit)=\wunit-(1-h)\bunit,
\end{equation}
for some $h\in \Lambda_+$, and
\begin{equation}
\label{eq:glambda}
\mathfrak{m}^{G,\dagger}_1(\glambda)=\mathfrak{m}^{G,\dagger}_1(\mathfrak{m}^{G,\dagger}_2(\gunit,\wlambda))=\wlambda-(1-h)\blambda,
\end{equation}
where the last inequality follows from the $A_{\infty}$-identity.

The equation $\eqref{eq:coh_unit}$ implies $[\bunit] = [\wunit]/(1-h) \in HF^0_G(L;\Lambda_0)$ (in the weakly unobstructed case so that the equivariant Floer cohomology is well-defined). Since $h\in \Lambda_+$, $\bunit$ is a cohomological unit. This is important, for instance when we consider quasi-isomorphisms of objects in the Fukaya category.

Since the properties~\eqref{eq:dagger_1} and~\eqref{eq:dagger_2} that $\wlambda$ satisfies are similar to those of a strict unit, we call $\wlambda$ a \emph{partial unit} and $\blambda$ a \emph{homotopy partial unit}. We formalize this in the following definition.

\begin{defn}[Partial unit] 
	\label{def:partial_unit}
	Let $(A,\fm)$ be an $A_{\infty}$-algebra over $\Lambda_0$. An element $\lambda^+\in A$ is called a \emph{partial unit} if
	\[
	\fm_k(\ldots, \lambda^+,\ldots)=0
	\]
	for $k\ne 2$ and
	\[
	\fm_2(x,\lambda^+)=x\bigcdot \lambda^+, \quad \fm_2(\lambda^+,x)=\lambda^+\bigcdot x
	\]
	for an associative algebra structure $(A,\bigcdot)$. We say that an element $\lambda \in A$ is a \emph{homotopy partial unit} if there exists an element $\mu\in A$ such that 
	\[
	\fm_1(\mu)=\lambda^+ + c\lambda
	\]
	for some nonzero element $c\in \Lambda_0$.
\end{defn}

The following theorem shows that the $A_{\infty}$-structure maps of  
$(CF^{\bullet}_G(L;\Lambda_0)^{\dagger}, \mathfrak{m}^{G,\dagger})$ are $H^{\bullet}_G(\mathrm{pt};\Lambda_0)$-multilinear.




\begin{theorem} \label{thm:pull-out-lambda}
Assume that $L$ has non-negative minimal Maslov index. Let $X_1,\ldots,X_{k}\in CF^{\bullet}_G(L;\Lambda_0)^{\dagger}$.  Let us fix a basis $\{x_1,x_2,\ldots\}$ and $\{\lambda_1,\lambda_2,\ldots\}$ for $CF^{\bullet}(L)$ and $H^{\bullet}(BG)$, respectively.
For $\ell\in\{1,\ldots,k\}$,  we can write $X_{\ell}=\sum a_{ij} x_i\otimes \lambda_j$, where $a_{ij}\in \Lambda_0$ and $a_{ij}=0$ for all except finitely many $i$ and $j$. Then, we have
\begin{equation*}
\mathfrak{m}^{G,\dagger}_k(X_1,\ldots,X_k)=(-1)^{\ell}\sum 
a_{ij} \lambda_j\cdot \mathfrak{m}^{G,\dagger}_k(X_1,\ldots,X_{\ell-1},x_i\otimes 1,X_{\ell+1},\ldots,X_k).
\end{equation*}	
\end{theorem}

\begin{proof}
From~\eqref{eq:dagger_1}, it follows that 
\[
\mathfrak{m}^{G,\dagger}_k(X_1,\ldots,X_{\ell-1},x_i\otimes \lambda_j,X_{\ell+1},\ldots,X_k)=\mathfrak{m}^{G,\dagger}_k(X_1,\ldots,X_{\ell-1},\mathfrak{m}^{G,\dagger}_2(\wlambda,x_i\otimes 1),X_{\ell+1},\ldots,X_k).
\]
Then, by using~\eqref{eq:dagger_2} and $A_{\infty}$-identity, one obtains
\begin{align*}
&\mathfrak{m}^{G,\dagger}_k(X_1,\ldots,X_{\ell-1},\mathfrak{m}^{G,\dagger}_2(\wlambda,x_i\otimes 1),X_{\ell+1},\ldots,X_k)\\&=(-1)^{\ell-1} \mathfrak{m}^{G,\dagger}_k(\mathfrak{m}^{G,\dagger}_2(\wlambda,X_1)\ldots,X_{\ell-1},x_i\otimes 1,X_{\ell+1},\ldots,X_k)\\&=(-1)^{\ell}\mathfrak{m}^{G,\dagger}_2(\wlambda,\mathfrak{m}^{G,\dagger}_k(X_1,\ldots,X_{\ell-1},x_i\otimes 1,X_{\ell+1},\ldots,X_k))\\&=(-1)^{\ell}\lambda_j\cdot \mathfrak{m}^{G,\dagger}_k(X_1,\ldots,X_{\ell-1},x_i\otimes 1,X_{\ell+1},\ldots,X_k)
\end{align*}
where the last equality follows from~\eqref{eq:dagger_1}.
\end{proof}


\subsection{Disc potentials in equivariant Lagrangian Floer theory}

We now introduce an equivariant disc potential for the equivariant Morse model $(CF^{\bullet}_G(L;\Lambda_0)^{\dagger}, \mathfrak{m}^{G,\dagger})$ \emph{over the graded coefficient ring} $H^{\bullet}_G(\mathrm{pt};\Lambda_0)$. 

For an element $b\in CF^{\bullet}_G(L;\Lambda_0)^{\dagger}$, we define the \emph{equivariant weak Maurer-Cartan equation} by
\begin{equation}
\label{eq:equiv_MC}
\mathfrak{m}^{G,\dagger,b}_0(1)=\mathfrak{m}^{G,\dagger}_0(1)+\mathfrak{m}^{G,\dagger}_1(b)+\mathfrak{m}^{G,\dagger}_2(b,b)+\ldots  \in H^{\bullet}_G(\mathrm{pt};\Lambda_0) \cdot \wunit.
\end{equation}
We denote by $\mathcal{MC}_G(L)$ the space of odd degree solutions of \eqref{eq:equiv_MC}. An element $b\in \mathcal{MC}_G(L)$ is called a \emph{weak Maurer-Cartan element over $H^{\bullet}_G(\mathrm{pt};\Lambda_0)$}. We say that $(CF^{\bullet}_G(L;\Lambda_0)^{\dagger}, \mathfrak{m}^{G,\dagger})$ is \emph{weakly unobstructed} if $\mathcal{MC}_G(L)$ is nonempty.



\begin{defn}
The deformation of $(L,G)$ by $b$ (or simply $(L,G,b)$) is called \emph{unobstructed} (resp. \emph{weakly unobstructed}) if $\mathfrak{m}^{G,\dagger,b}(1)=0$ (resp. $b\in \mathcal{MC}_G(L)$). In particular, if $b=0$, we simply call $(L,G)$ \emph{unobstructed}.   
\end{defn}

Similar to Lemma \ref{lem:unobstr}, we have

\begin{lemma}
\label{lem:G_unobstr}
For $b\in CF^{1}_G(L;\Lambda_{+})$, suppose that
\[
\mathfrak{m}^{G,\dagger,b}_0(1)= \left( W(b) +\sum_{\deg \lambda =2} \phi_{\lambda}(b)\lambda \right) \bunit, 
\]
for some $W(b)$, $\phi_{\lambda}(b)\in \Lambda_0$, and the minimal Maslov index of $L$ is nonnegative. 
Then there exists $b^{\dagger}\in CF^{odd}_G(L;\Lambda_0)^{\dagger}$ such that 
\[
\mathfrak{m}^{G,\dagger,b^{\dagger}}_0(1)=\left( W^{\wT}(b)+\sum_{\deg \lambda =2} \phi_{\lambda}^{\wT}(b) \lambda \right) \wunit,
\] 
for some $W^{\wT}(b)$, $\phi_{\lambda}^{\wT}(b) \in \Lambda_0$.
In particular, $(CF^{\bullet}_G(L;\Lambda_0)^{\dagger}, \mathfrak{m}^{G,\dagger})$ is weakly unobstructed. 
Furthermore, if the minimal Maslov index of $L$ is assumed to be at least two, then $W^{\wT}=W(b)$ and $\phi_{\lambda}^{\wT}(b)=\phi_{\lambda}(b)$.
\end{lemma}

\begin{proof}
	
Let us tentatively set
\[
b^{\dagger}=b+s(b)\left(W(b)+\sum_{\deg \lambda=2} \phi_{\lambda}(b)\lambda \right)\cdot \gunit,
\]
for some $s(b)\in \Lambda_0$ to be determined. We have 
\begin{align*}
\mathfrak{m}^{G,\dagger,b^{\dagger}}_0(1) &=\sum_{k\ge 0} \mathfrak{m}^{G,\dagger}_k(b,\ldots,b)+\sum_{k\ge 1} \mathfrak{m}^{G,\dagger}_k(b,\ldots,b,\Delta b,b,\ldots,b)\\
&= \mathfrak{m}^{G,\dagger,b}_0(1)+\mathfrak{m}^{G,\dagger,b}_1(\Delta b),
\end{align*}
where $\Delta b=b^{\dagger}-b$. The reason for the first equality is the following: $\Delta b$ is a multiple of $\gunit$, and the latter is in degree $-1$. By the dimension formula \eqref{eq:dim_formula}, the terms with more than one $\gunit$ as inputs vanish since the outputs have degrees at most $(-2)$.  

Next, using the assumption that the minimal Maslov index of $L$ is nonnegative, we have
\[
\mathfrak{m}^{G,\dagger,b}_1(\gunit)=\wunit - (1-h(b))\cdot \bunit,
\]
for some $h(b)\in \Lambda_+$, by the same argument as in the proof of Lemma \ref{lem:unobstr}. This computes $\mathfrak{m}^{G,\dagger,b}_1(\Delta b)$ by $H^{\bullet}_G(\pt,\Lambda_0)$-linearity of the $A_{\infty}$-maps proved in Theorem \ref{thm:pull-out-lambda}. Combining with the first assumption of the Lemma, we have
\begin{align*}
\mathfrak{m}^{G,\dagger,b^{\dagger}}_0(1)=  \left( W(b) +\sum_{\deg \lambda =2} \phi_{\lambda}(b)\lambda \right) \bunit + s(b)\left(W(b)+\sum_{\deg \lambda=2} \phi_{\lambda}(b)\lambda \right) \cdot (\wunit - (1-h(b)) \bunit). 
\end{align*}

By setting $s=\cfrac{1}{1-h(b)}$, we get 
\begin{equation*}
\mathfrak{m}^{G,\dagger,b^{\dagger}}_0(1)= \left( \cfrac{W(b)}{1-h(b)}+\sum_{\deg \lambda =2} \cfrac{\phi_{\lambda}(b)}{1-h(b)}\lambda\right) \wunit =: \left(W^{\wT}(b)+\sum_{\deg \lambda =2} \phi_{\lambda}^{\wT}(b) \lambda \right) \wunit.
\end{equation*}
Furthermore, if the minimal Maslov index of $L$ is at least two, then $h(b)=0$, and therefore $W^{\wT}(b)=W(b)$, $\phi_{\lambda}^{\wT}(b)=\phi_{\lambda}(b)$.
\end{proof}	

\begin{corollary}
\label{cor:L_L_G}
In the setting of Lemma \ref{lem:G_unobstr}, if $b\in C^1(f;\Lambda_0)$ and $\fm_0^b(1)\in \Lambda_0 \cdot \bunit$, then the equivariant Morse model $(CF^{\bullet}_G(L;\Lambda_0)^{\dagger}, \mathfrak{m}^{G,\dagger})$ is weakly unobstructed.
\end{corollary}

\begin{proof}
If $b\in C^1(f;\Lambda_0)$ and $\fm_0^b(1)\in \Lambda_0 \cdot \bunit$, then $\mathfrak{m}^{G,\dagger,b}_0(1)$ is of the form 
\[
\mathfrak{m}^{G,\dagger,b}_0(1)=\left( W(b) +\sum_{\deg \lambda =2} \phi_{\lambda}(b)\lambda\right) \bunit.
\]
\end{proof}


We remark that even if the minimal Maslov index of $L$ is $2$ and $(L,b)$ is unobstructed, one can only expect $(L,G,b)$ to be weakly unobstructed in general. This is due to the possibility of the constant disc class  contributing to degree $2$ equivariant parameters. 

Now, let $\mathcal{MC}^{\blT}_G(L)$ be the space of elements $b\in CF^{1}_G(L;\Lambda_{+})$ which satisfy
\[
\mathfrak{m}^{G,\dagger,b}_0(1)=W(b) \bunit+\sum_{\deg \lambda =2} \phi_{\lambda}(b)\blambda
\]
for some $W(b), \phi_{\lambda}(b) \in \Lambda_0$. By the Lemma \ref{lem:G_unobstr}, for each such $b$, there exists $b^{\dagger}\in CF^{odd}_G(L;\Lambda_0)^{\dagger}$ such that 
\[
\mathfrak{m}^{G,\dagger,b^{\dagger}}_0(1)=\left(W^{\wT}(b)+\sum_{\deg \lambda =2}  \phi_{\lambda}^{\wT}(b) \lambda \right)\cdot \wunit,
\] 
Then, we may view $W^L_G(b):=W^{\wT}(b)+\sum_{\deg \lambda =2}  \phi_{\lambda}^{\wT}(b) \lambda$ as a map from $\mathcal{MC}^{\blT}_G(L)$ to $H^{\bullet}_G(\mathrm{pt};\Lambda_0)$.

\begin{defn} \label{def:G_disc_potential}
	We will call $W^L_G:=W^{\wT} \colon \mathcal{MC}^{\blT}_G(L) \to H^{\bullet}_G(\mathrm{pt};\Lambda_0)$ the \emph{equivariant disc potential} of $(L,G)$. 
\end{defn}


\section{$T$-equivariant disc potentials of toric manifolds} \label{sec:toric}
	In this section, we study the equivariant Morse model in the case of a torus $T\cong (\bS^1)^{\ell}$ acting on a closed, connected, relative spin Lagrangian submanifold $L\cong T\times P$ of product type of a symplectic $T$-manifold $X$, such that $T$ acts freely on the first factor of $L$ and trivially on the second factor. In the case $L\cong T$ is a regular moment map fiber of a semi-projective, semi-Fano toric manifold, we recover the equivariant toric superpotential $W_{\lambda}$ as the $T$-equivariant disc potential of $L$.

\subsection{Morse theory on the approximation spaces}\label{sec:Morsethyappr}

  For explicit computations of the equivariant disc potentials, we begin by fixing our choice of admissible Morse-Smale pairs $\{(f_N,\mathscr{V}_N)\}_{N\in \bN}$.

		Our choices of the universal bundle $ET$ and the classifying space $BT$ are $ET=(\bS^{\infty})^{\ell}$ and $BT=(\C\bP^{\infty})^{\ell}$. We also have the finite dimensional approximations $ET(N)= (\bS^{2N+1})^{\ell}$ and $BT(N)= (\C\bP^{N})^{\ell}$. Let $T(N)$ be the finite dimensional approximation $T\times_T (\bS^{2N+1})^{\ell}$ of $T_T$.
		Since the $T$ acts trivially on $P$, we have $L_{T}=T_{T}\times (P\times (\C\bP^{\infty})^{\ell})$ and $L(N)=T(N)\times (P\times (\C\bP^{N})^{\ell})$. 

	For $N\ge 1$, let $[z_{i,0},\ldots,z_{i,N}]$ be the homogeneous coordinates on the $i$-th component of $BT(N)=(\C\bP^N)^{\ell}$. Let $\pi_N \colon L(N) \to (\C\bP^N)^{\ell}$ be the projection map. Let $\{U_{j_1\ldots j_{\ell}}\}$ be the open cover of $(\C\bP^N)^{\ell}$ defined by
	$$
	U_{j_1\ldots j_{\ell}}=\left\{[z_{i,0},\ldots,z_{i,N}]_{i=1,\ldots,{\ell}}\in (\C\bP^N)^{\ell} \mid z_{i,j_i}\ne 0\right\}.
	$$
	Set $\widetilde{U}_{j_1\ldots j_{\ell}}=\pi_N^{-1}(U_{j_1\ldots j_{\ell}})\cong L\times (\C^N)^{\ell}$. We will be working with the atlas $\{\widetilde{U}_{j_1\ldots j_{\ell}}\}$ for $L(N)$ with local coordinates
	\begin{equation*}
	\left(\left(\theta_1^{(j_1\ldots j_{\ell})},\ldots,\theta_{\ell}^{(j_1\ldots j_{\ell})}, \vec{p}^{(j_1\ldots j_{\ell})} \right),\left(z_{i,0}^{(j_1\ldots j_{\ell})},\ldots,\widehat{z_{i,j_i}^{(j_1\ldots j_{\ell})}=1},\ldots,z_{i,N}^{(j_1\ldots j_{\ell})}\right)_{i=1,\ldots,{\ell}}\right)
	\end{equation*}
	on $\widetilde{U}_{j_1\ldots j_{\ell}}$, where $\theta_i^{(j_1\ldots j_{\ell})}\in [0,2\pi)$ are angular coordinates on $(\bS^1)^{\ell}$, $\vec{p}^{(j_1\ldots j_{\ell})}$ is any coordinate system on $P$, and the term under $``\widehat{\quad}"$ is omitted. Set $\vartheta_{i,j}^{(j_1\ldots j_{\ell})}=\mathrm{Arg}\left(z^{(j_1\ldots j_{\ell})}_{i,j}\right)$. The transition map $\widetilde{U}_{j_1\ldots j_{\ell}}$ to  $\widetilde{U}_{j'_1\ldots j'_{\ell}}$ is given by
	\begin{multline*}
	\left(\left(\theta_1^{(j_1\ldots j_{\ell})},\ldots,\theta_{\ell}^{(j_1\ldots j_{\ell})},\vec{p}^{(j_1\ldots j_{\ell})}\right),\left(z_{i,0}^{(j_1\ldots j_{\ell})},\ldots,\widehat{z_{i,j_i}^{(j_1\ldots j_{\ell})}=1},\ldots,z_{i,N}^{(j_1\ldots j_{\ell})}\right)\right)\mapsto\\
	\left(\left(\theta_1^{(j_1\ldots j_{\ell})}+\vartheta_{1,j'_1}^{(j_1\ldots j_{\ell})},\ldots,\theta_{\ell}^{(j_1\ldots j_{\ell})}+\vartheta_{{\ell},j'_{\ell}}^{(j_1\ldots j_{\ell})},\vec{p}^{(j_1\ldots j_{\ell})}\right),\left(\cfrac{z_{i,0}^{(j_1\ldots j_{\ell})}}{z_{i,j'_i}^{(j_N\ldots j_{\ell})}},\ldots,\widehat{\cfrac{z_{i,j'_i}^{(j_1\ldots j_{\ell})}}{z_{i,j'_i}^{(j_1\ldots j_d)}}=1}, \ldots,\cfrac{z_{i,j_i}^{(j_1\ldots j_{\ell})}}{z_{i,j'_i}^{(j_1\ldots j_{\ell})}}\ldots,\cfrac{z_{i,N}^{(j_1\ldots j_{\ell})}}{z_{i,j'_i}^{(j_1\ldots j_{\ell})}}\right)\right).
	\end{multline*}
		
	We fix the inclusion $(\C\bP^N)^{\ell}\into (\C\bP^{N+1})^{\ell}$ to be
	\begin{equation}
	\label{inclusion}
	\left([z_{i,0},\ldots,z_{i,N}]\right)\mapsto \left([z_{i,0},\ldots,z_{i,N},0]\right),
	\end{equation}
	which in turn, fixes the inclusions $X(N)\into X(N+1)$ and $L(N)\into L(N+1)$.
	
	For $U(1)=\bS^1$, the perfect Morse function \eqref{eq:perfect_Gr} on the infinite complex Grassmanian specializes to 
	\[
	f_{\C\bP^{\infty}} \colon \C\bP^{\infty}\to \R, \quad f_{\C\bP^{\infty}}([z_{0},z_{1},\ldots]) = \cfrac{-\sum_{k=1}^{\infty} k|z_{k}|^2} {\sum_{k=0}^{\infty} |z_{k}|^2},
	\]
	from which we obtain a perfect Morse function on $BT$
	\begin{equation} \label{eq:CP_infty}
	\varphi \colon (\C\bP^{\infty})^{\ell} \to \R, \quad \varphi([z_{i,0},z_{i,1},\ldots])=-\sum_{i=1}^{\ell}\cfrac{\sum_{k=1}^{\infty} k|z_{i,k}|^2} {\sum_{k=0}^{\infty} |z_{i,k}|^2}.
	\end{equation}
	The critical points of $\varphi$ are of the form
	\[
	\Crit(\varphi)= \left\{ ([0,\ldots,0,z_{i,j_i}\ne 0,0,\ldots]_{i=1,\ldots,\ell}) \mid (j_1,\ldots, j_{\ell})\in \Z_{\ge 0}^{\ell} \right\}.
	\]
	with degrees given by $\sum_{i=1}^{\ell} 2j_i$. We denote the degree $2$ critical points with $j_i=1$, and $j_k=0$ for $k\ne i$ by $\lambda_i$.  We also denote the critical point where $\varphi$ attains the maximum by $\bm{1}_{BT}$, i.e.
	\[
	\bm{1}_{BT}=([z_{i,0}\ne 0,0,0,\ldots]_{i=1,\ldots,\ell}).
	\]
	
	We set $\varphi_N=\varphi|_{BT(N)}$. Note that $\varphi_N$ is a perfect Morse function on $(\C\bP^{N})^{\ell}$. We will abuse notation and denote the degree two critical points and the maximum of $\varphi_N$ again by $\bm{1}_{BT}$ and $\lambda_i$, respectively. Note that we have
	\[
	H^{\bullet}(BT(N);\Z)=H^{\bullet}((\C\bP^{N})^{\ell};\Z)=\Z[\lambda_1,\ldots,\lambda_{\ell}]/\langle \lambda_1^{N+1},\ldots,\lambda_{\ell}^{N+1}\rangle,
	\]
	and
	\[
	H^{\bullet}_{T}(\pt;\Z)=\Z[\lambda_1,\ldots,\lambda_{\ell}].
	\]
	
	On the other hand, let $f_P$ be a Morse function on $P$ with a unique maximum $\bm{1}_P$, and let  $f_{T}: (\bS^1)^{\ell}\to\R$ be the perfect Morse function 
	\[
	f_{T}(\theta_1,\ldots,\theta_{\ell})=\sum_{i=1}^{\ell}\cos(\theta_i).
	\]
    The critical points of $f_{T}$ are of the form
    \[	
	\Crit\left(f_{T}\right)=\left\{(\theta_1,\ldots,\theta_{\ell})\in L \mid \theta_i=0 \text{ or } \theta_i=\pi\right\}.
    \]   
    Let $f=(f_{T},f_P)$ be the Morse function on $L$. We denote by $\bm{1}_{L}$ the critical point where $f$ attains the maximum, 
    {i.e.} $\bm{1}_{L}=((0,\ldots,0),\bm{1}_P)$. We also denote by $X_1,\ldots,X_{\ell}$ the degree $1$ critical points, and $X_i\wedge X_j$, $1\le i<j\le \ell$ the degree $2$ critical points of $f$ of the following forms
    \begin{equation*} \label{eq:X_i}
    X_i=((0,\ldots,\theta_i=\pi,\ldots,0),\bm{1}_P), 
    \end{equation*}
    and
    \[
    X_i\wedge X_j=((0,\ldots,\theta_i=\pi,\ldots,\theta_j=\pi,\ldots,0),\bm{1}_P).
    \]

	Let $\{D_{j}\}$ be the open cover of $\C\bP^{\infty}$ given by
	\begin{equation*}
	\label{polydisc}
	D_{j}=\left\{\left(z_{0}^{(j)},\ldots,z_{N}^{(j)}\right)\in U_{j} \mid \left|z_{k}^{(j)}\right|<2 \text{ for all } k\right\}.
	\end{equation*}
	Let $\{\rho^{(j)}\}$ be a smooth partition of unity subordinate to $\{D_{j}\}$, and denote by $\rho^{(j)}_{i}$ the post-composition of $\rho^{(j)}$ with the projection from $L_{T}$ to the $i$-th component of $(\C\bP^{\infty})^{\ell}$. We define $\rho^{(j_1\ldots j_{\ell})} \colon L_{T}\to \R$ by 
	\[
	\rho^{(j_1\ldots j_{\ell})}=\prod_{i=1}^{\ell} \rho^{(j_i)}_{i}.
	\]
	Let $\rho^{(j_1\ldots j_{\ell})}_N=\rho^{(j_1\ldots j_{\ell})}|_{BT(N)}$, then $\{\rho^{(j_1\ldots j_{\ell})}_N\}$ is a smooth partition of unity on $L(N)$.

	Let $\phi^{(j_1\ldots j_{\ell})}_N \colon \widetilde{U}_{j_1\ldots j_{\ell}}\to\R$ be the fiber-wise Morse function
	\begin{equation*}
	\phi^{(j_1\ldots j_{\ell})}_N\left(\left(\theta_1^{(j_1\ldots j_{\ell})},\ldots,\theta^{(j_1\ldots j_{\ell})}_{\ell},\vec{p}^{(j_1\ldots j_{\ell})}\right),\left(z_{i,0}^{(j_1\ldots j_{\ell})},\ldots,z_{i,N}^{(j_1\ldots j_{\ell})}\right)\right)=f\left(\theta_1^{(j_1\ldots j_{\ell})},\ldots,\theta^{(j_1\ldots j_{\ell})}_{\ell},\vec{p}^{(j_1\ldots j_{\ell})}\right),
	\end{equation*}
	and set
	\begin{equation*}
	\phi_N=\sum_{(j_1\ldots j_{\ell})} \rho^{(j_1\ldots j_{\ell})}_N  \phi^{(j_1\ldots j_{\ell})}_N. 
	\end{equation*}
	Notice that $\phi_N=\phi^{(j_1\ldots j_{\ell})}_N$ for some $(j_1\ldots j_{\ell})$ in a neighborhood of $\pi_N^{-1}(\lambda)$ for  $\lambda\in\Crit(\varphi_N)$. Then, the function $f_N \colon L(N)\to\R$ defined by
	\begin{equation*}
	f_N:=\epsilon \phi_N+\pi_N^*\varphi_N, \quad \epsilon>0, 
	\end{equation*}
	is a Morse function for sufficiently small $\epsilon$. $f_N$ has the critical point set
	\begin{equation*}
	\mathrm{Crit}(f_N)=\left\{(p,\lambda)|p\in\mathrm{Crit}\left(\phi_{N}|_{\pi_N^{-1}(\lambda)}\right), 
	\lambda\in\mathrm{Crit}(\varphi_N)\right\}=\mathrm{Crit}(f)\times\mathrm{Crit}(\varphi_N).
	\end{equation*}
	The degree of $(p,\lambda)\in\mathrm{Crit}(f_N)$ is given by $|(p,\lambda)|=|p|+|\lambda|$,
	where $|p|$ and $|\lambda|$ are the degrees of $p$ and $\lambda$ as critical points of $f$ and $\varphi_N$, respectively.
	
	Let $V_P$ be a negative pseudo-gradient for $f_P$. Let $V_N$ be the vector field given by
	\begin{align*}
	V_N|_{\widetilde{U}_{j_1\ldots j_{\ell}}}=&4\sum_{i=1}^{\ell} \sum_{j\ne j_i} (j-j_i)\left|z_{i,j}^{(j_1\ldots j_{\ell})}\right|^2\cfrac{\partial}{\partial \left|z_{i,j}^{(j_1\ldots j_{\ell})}\right|^2}\\
	&+\sum_{i=1}^{\ell} \left(\rho^{(j_i)}_{N,i}\sin\left(\theta_i^{(j_1\ldots j_{\ell})}\right)+\sum_{j\ne j_i} \rho^{(j)}_{N,i} \sin\left(\theta_i^{(j_1\ldots j_{\ell})}+\vartheta_{i,j}^{(j_1\ldots j_{\ell})}\right)\right) \cfrac{\partial}{\partial \theta_i^{(j_1\ldots j_{\ell})}}.
	\end{align*}
	Here $\rho^{(j)}_{N,i}=\rho^{(j)}_{i}|_{BT(N)}$. Then
	\[
	\mathscr{V}_N=V_N+V_P
	\]
	is a negative pseudo-gradient for $f_N$. For $N=0$, we set $f_0=f$, and $\mathscr{V}_0=\sum_{i=1}^{\ell} \sin(\theta_i) \cfrac{\partial}{\partial \theta_i}+V_P$.
	
	\begin{prop} 
	For generic choices of $V_P$, the pair $(f_N,\mathscr{V}_N)$ is a Morse-Smale pair. Moreover, the sequence $\{(f_N,\mathscr{V}_N) \}_{N\in \bN}$ is admissible in the sense of Definition \ref{def:morse-smale}.
	\end{prop}

	For $p,q\in\mathrm{Crit}(f_N)$, let $\mathcal{M}(p,q)$ be the moduli space of gradient flow lines from $p$ to $q$, modulo reparametrizations.   Let $(C^{\bullet}(f_N;\Z),\delta)$ to be the Morse cochain complex
	\[
	C^{\bullet}(f_N;\Z):=\bigoplus_{p\in \mathrm{Crit}(f_N)} \Z\cdot p
	\]
	We have
	\begin{equation*}
	C^{\bullet}(f_N;\Z)=C^{\bullet}(f;\Z)\otimes \Z[\lambda_1,\ldots,\lambda_{\ell}]/\langle \lambda_1^{N+1},\ldots,\lambda_{\ell}^{N+1}\rangle.
	\end{equation*}
   For convenience, we will write the elements of $\Crit(f_N)$ as $p\otimes \lambda$, where $p\in\Crit(f)$ and $\lambda\in\Crit(\varphi_N)$. The following result will be used in the computation of the equivariant disc potentials.

	\begin{prop}
		\label{prop:uniqueflowline}
		For $N\ge 1$ and $X_1,\ldots,X_{\ell}$, we have
		\[
		\delta(X_i\otimes\bm{1}_{BT})=\bm{1}_{L}\otimes \lambda_i.
		\]
		In particular, there exists a unique gradient flow line from $X_i\otimes \bm{1}_{BT}$ to $\bm{1}_{L}\otimes \lambda_i$.
	\end{prop}
	
	\begin{proof}
	The possible outputs of $\delta(X_i\otimes\bm{1}_{BT})$ are of degree $2$ critical points of the form $(X_{j}\wedge X_k)\otimes\bm{1}_{BT}$, and $\bm{1}_L\otimes\lambda_j$. It is easy to see that $\mathcal{M}(X_i\otimes\bm{1}_{BT}, (X_{j}\wedge X_k)\otimes\bm{1}_{BT})=\emptyset$ if both $j\ne i$ and $k\ne i$, and $\mathcal{M}(X_i\otimes\bm{1}_{BT}, (X_{j}\wedge X_k)\otimes\bm{1}_{BT})$ consists of two points of opposite orientation if either $j=i$ or $k=i$. 
		
	Suppose $\Phi \colon \R\to L(N)$ is a flow line from $X_i\otimes\bm{1}_{BT}$ to $\bm{1}_{L}\otimes \lambda_j$. Its projection $\pi_N\circ \Phi \colon \R\to (\C\bP^N)^{\ell}$ is a flow line from $\bm{1}_{BT}$ to $\lambda_j$ for the vector field $V$ on $(\C\bP^N)^{\ell}$ given by
		\begin{equation}\label{equ_descriptionvector}
		V|_{U_{j_1\ldots j_{\ell}}}=4\sum_{i=1}^{\ell} \sum_{j\ne j_i} (j-j_i)|z_{i,j}^{(j_1\ldots j_d)}|^2 \cfrac{\partial}{\partial |z_{i,j}^{(j_1\ldots j_{\ell})}|^2},
		\end{equation}
        whose image is contained in $U_{0\ldots 0}$. 
		This means the image of $\Phi$ is contained in $\tilde{U}_{0\ldots 0}$, and we can therefore write 
		\begin{equation}\label{equ_componentsofflow}
		\Phi(t)=\left(\left(\theta_1^{(0\ldots0)}(t),\ldots,\theta_{\ell}^{(0\ldots 0)}(t)\right),\left(z_{\mu,1}^{(0\ldots0)}(t),\ldots,z_{\mu,N}^{(0\ldots0)}(t)\right)_{\mu=1,\ldots,{\ell}}\right),
		\end{equation}
		in term of the coordinates on $\tilde{U}_{0\ldots 0}$. 
		
		We analyze the components in~\eqref{equ_componentsofflow}. As $\pi_N\circ \Phi \colon \R\to (\C\bP^N)^{\ell}$ is a flow line from $\bm{1}_{BT}$ to $\lambda_j$ of the vector field $V$ in~\eqref{equ_descriptionvector}, the component $z_{\mu,\nu}^{(0\ldots 0)}(t)$ is equal to $0$ for $\mu\ne j$ and $\nu\ne 1$, and the component $z_{j,1}^{(0\ldots 0)}(t)$ is equal to $e^{2t+i\vartheta}$ for some $\vartheta\in [0,2\pi)$. Since $\Phi$ is a flow line of $\mathscr{V}_N$ from $X_i\otimes\bm{1}_{BT}$ to $\bm{1}_{L}\otimes \lambda_j$, we have $\theta^{(0\ldots 0)}_{\mu}(t)=0$ for $\mu\ne i$, and  $\theta_i^{(0\ldots 0)}(t)$ satisfies 
		\[
		\lim_{t\to -\infty} \theta_i^{(0\ldots 0)}(t)=\pi,
		\]
		and
		\[\lim_{t\to +\infty} \theta_i^{(0\ldots 0)}(t)+\delta_{ij}\vartheta=0,
		\]
		in addition to 
		\[
		\cfrac{d \theta_i^{(0\ldots 0)}(t)}{dt}=a(t)\sin\left(\theta_i^{(0\ldots 0)}(t)\right)+b(t) \sin\left(\theta_i^{(0\ldots 0)}(t)+\vartheta\right).
		\]
		Here $a(t)=\rho^{(0)}_{1}|_{U_0}\left(e^{2t+i\vartheta}\right)$ and $b(t)=\rho^{(1)}_{1}|_{U_0}\left(e^{2t+i\vartheta}\right)$, and $\{\rho^{(0)}_{1}, \rho^{(1)}_{1}\}$ is the partition of unity for $\C\bP^1$. 
		
		For existence of a flow line when $i=j$, we note that the flow line with $\theta_i^{(0\ldots 0)}(t)=\pi$ and $\vartheta=0$ has the desired asymptotics. As for uniqueness, assume without loss of generality that $\Phi(0)$ is in a neighborhood of $X_i\otimes\bm{1}_{BT}$ such that $\rho^{(0)}_{N}(t)=1$ and $\rho^{(1)}_{N}(t)=0$. In this neighborhood, we have explicitly 
		\[
		\theta_i^{(0\ldots 0)}(t)=2\cot^{-1}\left(e^{-t}\cot\left(\cfrac{\theta_i^{(0\ldots 0)}(0)}{2}  \right)\right).
		\]
		It is easy to see that 
		\[
		\lim_{t\to -\infty} \theta_i^{(0\ldots 0)}(t)=
		\begin{cases}
		\pi &\mbox{if $\theta_i^{(0\ldots 0)}(0)=\pi$} \\
		0 &\mbox{otherwise.}
		\end{cases}
		\]
		Solving for $$\cfrac{d \theta_i^{(0\ldots 0)}(t)}{dt}=0$$ gives
		\[
		\tan \theta_i^{(0\ldots 0)}(t)=\cfrac{b(t)\sin(\vartheta)}{a(t)+b(t)\cos(\vartheta)}.
		\]
		This means we have
		\[
		\lim_{t\to +\infty} \theta_i^{(0\ldots 0)}(t)+\delta_{ij}\vartheta= 
		\begin{cases}
		0 &\mbox{if $\vartheta=0$ and $i=j$} \\
		\pi &\mbox{otherwise.}
		\end{cases}
		\]
	 This also implies that there exists no flow line when $i\ne j$.
	\end{proof}

\subsection{Computing $T$-equivariant disc potentials}\label{equ_computingteqdisc}

	By applying the construction in Section~\ref{sec:G Morse model} and \ref{sec:partial_units} with the choice of Morse-Smale pairs $\{(f_N,\mathscr{V}_N)\}$ made in Section~\ref{sec:Morsethyappr}, one obtains the equivariant Morse model $(CF^{\bullet}_{T} (L;\Lambda_0)^{\dagger}, \mathfrak{m}^{T,\dagger})$ 
    associated to the pair $(L,T)$. Note that we have
    \[
    CF^{\bullet}_{T} (L;\Lambda_0)^{\dagger}= CF^{\bullet}(L;\Lambda_0)\otimes_ {\Lambda_0} \Lambda_0[\lambda_1,\ldots,\lambda_{\ell}],
    \]
    where 
    \[
    CF^{\bullet}(L;\Lambda_0)=C^{\bullet}(f;\Lambda_0)\oplus \Lambda_0 \cdot \wunit_L \oplus \Lambda_0 \cdot \gunit_L
    \]
    is the unital Morse complex associated to $f \colon L\to \R$. In this section, we  compute the $T$-equivariant disc potential of $L$ assuming the minimal Maslov index of $L$ is nonnegative.

	For simplicity of notations, we will suppress $``\blT"$ and denote the unique maximum $\bunit_L$ on $L$ by $\bm{1}^{\vphantom{\blacktriangledown}}_L$. Let $\{ X_{1},\ldots,X_{\ell},Y_1,\ldots,Y_{\nu} \}$ be a basis of $C^{1}(f;\Q)$ with $\{X_i\}$ and $\{Y_j\}$ degree one critical points of $f_T$ and $f_P$ respectively. We will abuse notations and write $X_i:=X_i\otimes \bm{1}_{BT}$ and $Y_i:=Y_i\otimes \bm{1}_{BT}$. We also put  $\bm{\lambda}_i:=\bm{1}_L\otimes \lambda_i$, and $\bm{1}:=\bm{1}_L\otimes \bm{1}_{BT}$.

	Let $b=\sum_{i=1}^{\ell} x_iX_i+\sum_{i=1}^{\nu} y_i Y_i$, $(x_i, y_i)\in\Lambda_{+}^2$. 
We consider the boundary deformation of $\mathfrak{m}^{T,\dagger}_{0}$ by $b$
\[
		\mathfrak{m}^{T,\dagger, b}_{0} (1) =  \mathfrak{m}^{T, b}_0(1) = \mathfrak{m}^{T}_{0} (1) +\mathfrak{m}^{T}_{1}(b)+\mathfrak{m}^{T}_{2}(b,b)+\cdots.
\]
	The first equality above follows  from the fact that the restriction of  $\mathfrak{m}^{T, \dagger}$ to $CF^{\bullet}_{T}(L;\Lambda_0)$ agrees with $\mathfrak{m}^{T}$.
	
	We compute the obstruction $\mathfrak{m}_0^{T,b}(1)$ by counting pearly trees in $L_T \subset X_T$\footnote{To be more precise, we should perform the calculation on the finite dimensional approximations $L(N)\subset X(N)$ (see~\eqref{eq:m^G_gamma}, \eqref{eq:m^G_beta} and \eqref{eq:m^G}). By abuse of notations, the $A_\infty$-structure maps on $L_T$ are regarded as the $A_\infty$-structure maps on the approximation spaces.}. Since $m^T_{k,\beta}(b,\ldots,b)$ is of degree $2-\mu(\beta)$, it vanishes unless $\mu(\beta)=2$ or $0$. 

	The outputs of $\mathfrak{m}_0^{T,b}(1)$ have degree either $0$ or $2$ depending on the Maslov indices of the contributing disc classes. 
	The possible outputs are of the following forms$\colon$
	\begin{enumerate}
		\item (Degree zero) $\bm{1}$, 
		\item (Degree two) $p\otimes \one_{BT}$, where $p\in\Crit(f)$ is a degree $2$ critical point,
		\item (Degree two) $\bm{\lambda}_i$. 
	\end{enumerate}
    Notice that all the critical points above are contained in $\Crit(f_1)$. Thus, $\mathfrak{m}^{T, b}_0(1)$ can be computed by counting pearly trees in $L(1)$.
		
\begin{prop} \label{prop:gen-m0}
	Suppose that $T\cong (\mathbb{S}^1)^\ell$ acts on $(X, \omega_X)$ preserving $\omega_X$. 
	Let $L\subset X$ be a $T$-invariant closed Lagrangian submanifold of product type  $L\cong T \times P$ such that $T$ acts freely on $(\mathbb{S}^1)^\ell$ and acts trivially on $P$. Suppose the minimal Maslov index of $L$ is nonnegative, then
	\begin{equation}\label{equation_m0tbexp}
		\mathfrak{m}_0^{T,b}(1) = \fm_{\vphantom{T}0}^b(1) \otimes \bm{1}_{BT} + \sum_{i=1}^\ell h_i(\vec{x},\vec{y})  \bm{\lambda}_i
	\end{equation}		
where $\fm$ is the $A_{\infty}$-structure on  $CF^{\bullet}(L;\Lambda_0)$, and $h_i(\vec{x},\vec{y}) \in \Lambda_0$. Moreover, if $\fm_0^b(1) = W(\vec{x},\vec{y})  \one_{L}$ for some $W(\vec{x},\vec{y})\in \Lambda_+$, then $(CF^{\bullet}_{T} (L;\Lambda_0)^{\dagger}, \mathfrak{m}^{T,\dagger})$ is weakly unobstructed.		
	\end{prop}

\begin{proof}
By the discussion above, $\mathfrak{m}_0^{T,b}(1)$ can be expressed as 
\[
	\mathfrak{m}_0^{T,b}(1)= g_{\bm{1}}(b)\cdot \bm{1}+ \sum g_{p}(b)\cdot p\otimes \one_{BT}+ g_{\lambda_i}(b)\cdot \bm{\lambda}_i,
\]	
where the coefficients $g_{\bm{1}}(b)$, $g_{p}(b)$ and $g_{\lambda_i}(b)$ are elements of $\Lambda_0$ depending on $b$. 
We note that both $\bm{1}$ and $p\otimes \one_{BT}$ are critical points of the Morse function $f$ on $L$. By the definition of the equivariant Morse model (Definition \ref{def:G_Morse_model}), it can be computed using the $A_{\infty}$-algebra  $(CF^{\bullet}(L;\Lambda_0), \fm)$. In particular, by using the degree condition, we classify on the possible configurations on $\Gamma$ in $\mathfrak{m}_0^{T,b}(1)$.
 Moreover, we have
\[
 g_{\bm{1}}(b)\cdot \bm{1}+ \sum g_{p}(b)\cdot p\otimes \one_{BT} =  \fm_{\vphantom{T}0}^b(1) \otimes \bm{1}_{BT} . 
\]
Since $b=\sum_{i=1}^{\ell} x_iX_i+\sum_{i=1}^{\nu} y_i Y_i$ for $(x_i, y_i)\in\Lambda_{+}^2$ by definition, we can rewrite $g_{\lambda_i}(b)$ as a function $h_i(\vec{x},\vec{y})$ in variables $x_i$ and $y_i$. In summary, $\mathfrak{m}_0^{T,b}(1)$ has the expression \eqref{equation_m0tbexp}. 


 The last assertion follows from Corollary \ref{cor:L_L_G}, namely, if $\fm_0^b(1) =W(\vec{x},\vec{y})\one_{L}$, then	
	\begin{equation}\label{equation_m0tbexp2}
	\mathfrak{m}_0^{T,b}(1) =\left( W(\vec{x},\vec{y})  +\sum_{i=1}^{\ell} h_i(\vec{x},\vec{y}) \lambda_i\right) \bm{1}.
	\end{equation}
	\end{proof}
	
	In the following, we compute $\fm_0^b(1)$ and $h_i(\vec{x},\vec{y})$ explicitly under additional assumptions. 
	We begin by simplifying $h_i(\vec{x},\vec{y})$ under the condition that the minimal Maslov index of $L$ is at least $2$.
	
		\begin{lemma} \label{lem:h_i}
		In the situation of Proposition \ref{prop:gen-m0}, if in addition, the minimal Maslov index of $L$ is at least $2$, then $h_i(\vec{x},\vec{y}) =x_i$ in~\eqref{equation_m0tbexp}.

	\end{lemma}
	\begin{proof}
	The terms $h_i(\vec{x},\vec{y})$ are contributed by disc classes of Maslov index $0$ by degree reason. 
    Since $L$ has minimal Maslov index at least two, the only contribution comes from the trivial disc class $\beta_0$, hence the computation of the terms $\fm^T_k(b,\ldots,b)$ reduces to $\fm^T_{\Gamma}(b,\ldots,b)$, where $\Gamma\in \bm{\Gamma}_{k+1}$ is a stable planar rooted tree with all interior vertices decorated by $\beta_0$. In other words, the configurations we are counting are Morse flow trees in $L_1$.
	
	 By Proposition \ref{prop:uniqueflowline}, we have $\fm^T_1(X_i) = \bm{\lambda}_i$. We will show momentarily that the coefficients of $\bm{\lambda}_i$ in the terms $\mathfrak{m}_{k}^T(X_{\nu_1},\ldots,X_{\nu_k})$ are zero for $k\ge 2$ and $(\nu_1,\ldots,\nu_k) \in \{1,\ldots,\ell\}^k$. First, let us consider the case when $\nu_j\ne i$ for some $j$. In the proof of Proposition \ref{prop:uniqueflowline}, we have showed that there exists no flow line from $X_{i_j}$ to $\bm{\lambda}_i$. This means for generic small perturbations, the moduli spaces $\cM_{\Gamma}(f_1;X_{\nu_1},\ldots,X_{\nu_k};\bm{\lambda_i})$ are empty. 
	 
     Now, we consider $\fm^T_k(X_{i},\ldots,X_{i})$, $k\ge 2$. Since rotations on the $i$-th circle factor of $L$ commute with the structure group of $L_1\to \C\bP^1$, we have a global $\bS^1$-action on $L_1$ rotating the $i$-th circle factor of the fibers $L$. To achieve transversality for the moduli spaces $\cM_{\Gamma}(f_1;X_{i},\ldots,X_{i};\bm{\lambda_i})$, we can perturb the the flow lines from the first $k-1$ inputs using the $\bS^1$-action. We have the unique flow line $\gamma$ from $X_{i}$ to $\bm{\lambda}_i$ as in Proposition \ref{prop:uniqueflowline}. For a perturbed flow line from the first $k-1$ inputs to intersect with $\gamma$ (in order to form a flow tree), its projection in $BT(1)$ must coincide with $\pi(\gamma)$. However, over $\pi(\gamma)$, the flow lines from the first $k-1$ inputs are simply the rotations of $\gamma$ by $\bS^1$, which do not intersect with $\gamma$. Thus, $\cM_{\Gamma}(f_1;X_{i},\ldots,X_{i};\bm{\lambda_i})$ are empty for generic small perturbations.
     
     Similarly, since the fiber bundle $P_{T}=P\times BT$ is trivial, there exists no flow line from $Y_j$ to $\bm{\lambda}_i$ for any $i,j$. Therefore the coefficients of $\bm{\lambda}_i$ in the terms $\fm^T_k(b,\ldots,Y_j,\ldots,b)$ are zero. 
		\end{proof}
	
		Thus, in the setting of Lemma \ref{lem:h_i}, we have
		$$ \mathfrak{m}_0^{T,b}(1) = \fm_0^b (1) \otimes \one_{BT} + \sum_{i=1}^\ell x^i \bm{\lambda}_i. $$
		
		In \cite{FOOO-T}, it was shown that the moment map fibers of a compact semi-Fano toric manifold are weakly unobstructed in the de Rham model.  The following is the analogous statement in our setting. 
		
		\begin{lemma}
		\label{lem:toric_m_0}
			Let $L$ be a regular moment map fiber of a compact semi-Fano toric manifold $X$. Then $\fm_0^b(1)\in \Lambda_0= W^L\cdot \one_L$ for some $W^L(b)\in \Lambda_0$. 
		\end{lemma}
	
		\begin{proof}
		The only possible outputs of $\fm_0^b(1)$ are multiples of $\one_L$ and degree two critical points of $f$ (which are of the form $X_i\wedge X_j$), contributed by Maslov index zero and two disc classes, respectively. Since $X$ is semi-Fano, $L$ has minimal Maslov index $2$. This means only Morse flow trees contribute to the degree two critical points. We have $\fm_1(X_i)=0$, since $f$ is a perfect Morse function on $L$. For $\fm_k(X_{\nu_1},\ldots,X_{\nu_k})$ and $k\ge 2$, if there are two repeated inputs $X_{\nu_i}=X_{\nu_j}$, then by perturbing unstable hypertori of $X_{\nu_i}$ using the $\bS^1$-action rotating the $\nu_i$-th circle factor of $L$, the perturbed hypertori do not intersect and hence $\fm_k(X_{\nu_1},\ldots,X_{\nu_k})=0$. In the case of distinct inputs, we have 
		\[
		\sum_{\sigma\in S_k} \fm_k\left(X_{\nu_{\sigma(1)}},,\ldots,X_{\nu_{\sigma(k)}}\right)=0, \quad k\ge2,
		\]
		due to the orientations on the corresponding moduli spaces.
		
		\end{proof}
		
		\begin{remark}
			\label{rmk:perturbation}
		Before proceeding, a remark is in order addressing the perturbations used in the proof of Lemma \ref{lem:h_i}, and \ref{lem:toric_m_0}. We do not perturb the input singular chains when choosing perturbation for a fiber product $\cM_{k+1}(\beta;L(N);\vec{P})$ since the singular chains are fixed during the inductive construction and doing so would destroy the $A_{\infty}$-structure. On the other hand, since the Kuranishi structure on the disc moduli are weakly submersive, we can realize the perturbation of input singular chains by perturbing the respective evaluation maps.
		\end{remark}		
		
		Combining Lemma \ref{lem:h_i} and Lemma \ref{lem:toric_m_0},  to show that the equivariant toric superpotentials \eqref{equ_Giventalsuper} coincide with the equivariant disc potentials we defined in Definition \ref{def:G_disc_potential}, it remains to show that $W^L$ agrees with the Givental-Hori-Vafa superpotential $W^{\mathrm{toric}}$. The equality holds in the de Rham model due to divisor axiom (see \cite{Fukaya10}, Section 11 in \cite{FOOO-T}). On the other hand, in the singular chain model, we do not have divisor axiom, and the expression of $W^L$ depends on choice of perturbations. We will show below that we can choose certain perturbations such that we have the equality $W^L=W^{\mathrm{toric}}$.
		
		Recall that the holomorphic disc classes are generated by the basic disc classes $\beta_i$ for $1\le i\le m$ \cite{CO}.  Moreover, by \cite{FOOO-T}, a stable disc class of Maslov index two must be of the form $\beta_i+\alpha$ for some effective curve class $\alpha$ with $c_1(\alpha)=0$.  Let $n_1(\beta_i+\alpha)$ be the degree of the virtual fundamental class $\ev_*[\cM_1(\beta_i+\alpha)]$.  In the Morse model, it is given by the intersection number of $\ev_*[\cM_1(\beta_i+\alpha)]$ with the maximal point $\one_L$.

	\begin{theorem}[Equivariant toric superpotential]
		\label{thm:superpotential}
		Let $X$ be a compact semi-Fano toric manifold with $\dim_{\C} X=d$, and let $v_1,\ldots,v_m\in \Z^d$ be the primitive generators of the one-dimensional cones of the fan $\Sigma$ defining $X$. Let $L\subset X$ be a regular moment map fiber. Then, 
		\[
		\mathfrak{m}^{T,b}_{0}(1) = W^L (x_1,\ldots,x_d) \bm{1}+\sum_{i=1}^d x_i \bm{\lambda}_i=\left(W^L (x_1,\ldots,x_d)+\sum_{i=1}^d \lambda_i x_i \right)\one:=W^L_T(x_1,\ldots,x_d)  \one,
		\]
		where
		\begin{equation}
		\label{eq:W_disc}
		W^L (x_1,\ldots,x_d) = \sum_{i=1}^m \left(\sum_{\alpha: c_1(\alpha)=0} n_1(\beta_i + \alpha) \bT^{\omega_X(\alpha)}\right) \bT^{\omega_X(\beta_i)} \exp (v_i \cdot (x_1,\ldots,x_d)).
		\end{equation}		
		In particular in the Fano case,	
		\[
		\mathfrak{m}^{T,b}_{0}=\left(W^{\mathrm{toric}}(e^{x_1},\ldots,e^{x_d}) +\sum_{i=1}^d \lambda_i x_i \right)\one,
		\]
		where $W^{\mathrm{toric}}$ is the superpotential of Givental and Hori-Vafa \cite{Givental,HV,MS_big_book}.
	\end{theorem}
	
	\begin{proof}
	Recall from Proposition \ref{prop:gen-m0} that $W^L\cdot \one_L=\fm^b_0(1)$. Since $L$ has minimal Maslov index $2$, by degree reason, a stable rooted tree $\Gamma\in\bm{\Gamma}_{k+1}$ contributing to the term $\fm_k(b,\ldots,b)$, $k\ge 1$, must have exactly one interior vertex decorated by a Maslov index $2$ disc class and the remaining interior vertices are decorated by the constant disc class $\beta_0$. 
	
	Let us abuse notation and denote by $X_j$ the unstable hypertorus of the degree one critical point $X_j$ of the perfect Morse function $f=\sum_{i=1}^{d}\cos(\theta_i)$ on $(\bS^1)^d$. We can choose an identification $L\cong (\bS^1)^d$ such that the boundaries of the finitely many holomorphic discs in basic disc classes $\beta_1,\ldots,\beta_m$, passing through the maximal point $\one_L$ do not intersect with $X_\mu \cap X_\nu$ for any $\mu\not=\nu$. This means the only configurations $\Gamma\in\bm{\Gamma}_{k+1}$, contributing to $\fm_k(b,\ldots,b)$ have exactly one interior vertex, which are decorated by a Maslov index $2$ disc class $\beta_i+\alpha$. It remains to consider contributions from the moduli spaces of the form $\cM_{k+1}(\beta_i+\alpha;L;X_{\mu_1},\ldots,X_{\mu_k} 	)$, $(\mu_1,\ldots,\mu_k)\in \{1,\ldots,d\}^k$.	
	
	We note that $\cM_{k+1}(\beta_i+\alpha;L;X_{\mu_1},\ldots,X_{\mu_k})^{\fs}=\emptyset$ for generic perturbations $\fs$ unless $\partial \beta_i$ intersects all the input hypertori $X_\mu$. Therefore, it suffices to assume that this is the case. Let us denote by $r_j$ the number of times $X_j$ appears in the inputs, and by $m_j$ is the intersection multiplicity of $\partial\beta_i\pitchfork X_j$. In order to obtain the expression \eqref{eq:W_disc}, we choose the following Kuranishi structure and perturbation for  $\cM_{k+1}(\beta_i+\alpha;L;X_{\mu_1},\ldots,X_{\mu_k})$.
	
	Let $\fs=\{\fs_p\}$ be the multisection chosen for $\cM_1(\beta_i+\alpha;L)$, which gives rise to the open Gromov-Witten invariant $n_1(\beta_i+\alpha)$. (We note that $n_1(\beta_i+\alpha)$ is in fact independent of the choice of a transversal multisection due to the absence of disc bubbles.) For $k\ge 1$, let $\forget \colon \cM_{k+1}(\beta_i+\alpha;L)\to\cM_1(\beta_i+\alpha;L)$ be the forgetful map which forgets the $k$ input marked points and stabilizes. To relate the fiber product $\cM_{k+1}(\beta_i+\alpha;L;X_{\mu_1},\ldots,X_{\mu_k})$ with $n_1(\beta_i+\alpha)$, we first pullback the Kuranishi structure and multisection from $\cM_1(\beta_i+\alpha;L)$ to $\cM_{k+1}(\beta_i+\alpha;L)$ via $\forget$. Let $p^+=[\Sigma,\vec{z},u]\in \cM_{k+1}(\beta_i+\alpha;L)$ such that $\forget(p^+)=p\in\cM_1(\beta_i+\alpha;L)$. Let $V^+_{p^+}$, $E^+_{p^+}$, and $\fs_{p^+}^+:V^+_{p^+}\to (E^+_{p^+})^{\ell}/S_{\ell}$ be pullback of the Kuranishi neighborhood, obstruction bundle, and multisection at $p$ via $\forget$ (see \cite[Lemma 7.3.8]{FOOO} for the precise definitions). To make the Kuranishi structure on $\cM_{k+1}(\beta_i+\alpha;L)$ weakly submersive, we enlarge Kuranishi neighborhood $V_{p^+}$ and the obstruction bundle $E_{p^+}$ at $p^+$ to be of the form
	\begin{enumerate}
	    \item $E_{p^+}=E^+_{p^+}\bigoplus_{i=1,\ldots,k} T_{u(z_i)}L$.
		
    \item $V_{p^+}=V^+_{p^+}\times U_{p^+}$, where $U_{p^+}\subset \bigoplus_{i=1,\ldots,k} T_{u(z_i)}$ is an open subset containing $0$.
	\end{enumerate}
Then, $\fs^+_{p^+}$ naturally induces a multisection $\fs'_{p^+}: V_{p^+}\to (E_{p^+})^{\ell}/S_{\ell}$ which is characterized as follows. Let $\widetilde{\fs_{p^+}} \colon V_{p^+}\to (E_{p^+})^{\ell}$ be the lift of $\fs_{p^+}$ and write
\[
\widetilde{\fs_{p^+}}=\left((s^1_{p^+,1},s^{2,1}_{p^+},\ldots,s^{2,k}_{p^+}),\ldots,(s^1_{p^+,\ell},s^{2,1}_{p^+},\ldots,s^{2,k}_{p^+}) \right).
\]
Here, the restriction of $(s^1_{p^+,1},\ldots,s^1_{p^+,\ell}):V_{p^+}\to (E^+_{p^+})^{\ell}$ to $V^+_{p^+}\times\{0\}$ agrees with the lift of $\fs_{p^+}^+$ and is constant along the $U_{p^+}$-direction; $(s^{2,1}_{p^+},\ldots,s^{2,k}_{p^+}):V_{p^+}\to \bigoplus_{i=1,\ldots,k} T_{u(z_i)}$ is the tautological section given by the embedding $U_{p^+}\subset \bigoplus_{i=1,\ldots,k} T_{u(z_i)}$.

Now, let $u_1,\ldots,u_k$ be distinct, generic, small vectors in the direction of $\partial \beta_i$. (Since $TL$ is trivial, we can view $u_1,\ldots,u_k$ as constant vector fields on $L$.)
We define a new multisection $\fs'=\{\fs'_{p^+}\}$ by 
\[
\widetilde{\fs'_{p^+}} \colon V_{p^+}\to E_{p^+}^{\ell\cdot k!},
\] 
\[
\widetilde{\fs'_{p^+}}=\left((s^1_{p^+,1},s^{2,1}_{p^+}+u_{\sigma(1)},\ldots,s^{2,k}_{p^+}+u_{\sigma(k)}),\ldots,(s^1_{p^+,\ell},s^{2,1}_{p^+}+u_{\sigma(1)},\ldots,s^{2,k}_{p^+}+u_{\sigma(k)})\right)_{\sigma \in S_k},
\]
and $\fs'_{p^+}=\widetilde{\fs'_{p^+}}/S_{\ell\cdot k!}$. Then, $\fs'$ induces a transversal multisection on the fiber product $\cM_{k+1}(\beta_i+\alpha;L;X_{\mu_1},\ldots,X_{\mu_k})$ whose zero locus is given by 
\[
\cM_{k+1}(\beta_i+\alpha;L)^{\fs'}\times_{L^k} \prod_{\nu=1,\ldots k} X_{\mu_\nu}=\coprod_{1,\ldots,\prod_{j=1}^d m_j^{r_j}} \fs^{-1}(0)
\]
weighted by $1/k!$. Its intersection with $\one_L$ gives the coefficient
	$\frac{n_1(\beta_i+\alpha)\prod_{j=1}^d m_j^{r_j}}{k!}$. (See Figure \ref{fig_hypertori} for an illustration of such perturbations.) Summing over all the possible inputs of $k$ hypertori and the $\frac{k!}{\prod_{j=1}^d (r_j!)}$ arrangements among them, we obtain the term $n_1(\beta_i+\alpha) \exp (v_i \cdot (x_1,\ldots,x_d))$.
\end{proof}	

\begin{figure}[h]
	\begin{center}
		\includegraphics[scale=0.65]{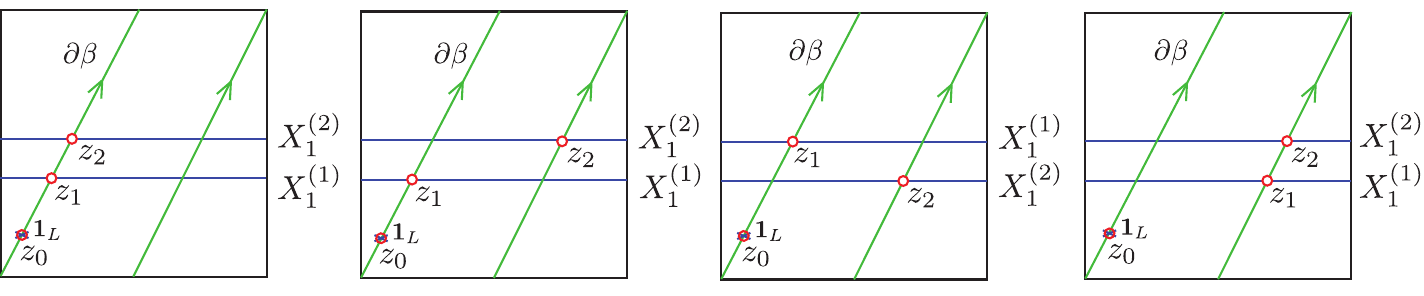}
		\caption{All possible configurations contributing to $\fm_{2,\beta}(X_1,X_1)$ in the case $\partial \beta\cdot X_1=2$, after requiring $z_0$ to intersect $\one_L$. Here $X^{(i)}_1$ denotes the composition of $X_1$ with the time-$1$ flow of $v_i$.}
		\label{fig_hypertori}
	\end{center}
\end{figure}

	\begin{remark}
		In the above theorem, we have chosen the perturbations for the fiber products to be the average of all configurations of the perturbed unstable hypertori of $X_i$.  Indeed, we can choose other perturbations, which would give different expressions of the (non-equivariant) disc potential.
		
		Take $\bP^1$ as an example.  Let $L$ be the equator.  Let us perturb $X_1$ in the counterclockwise order with respect to the left hemisphere.  Then the non-equivariant disc potential will read
		\[
		 W^L = \bT^{A/2} \left( \frac{1}{1-x_1} +(1-x_1) \right)
		\]
		instead of the well-known expression $\bT^{A/2} \left( e^{x_1} + e^{-x_1} \right)$.
	\end{remark}

	\begin{figure}[h]
	\begin{center}
		\includegraphics[scale=0.37]{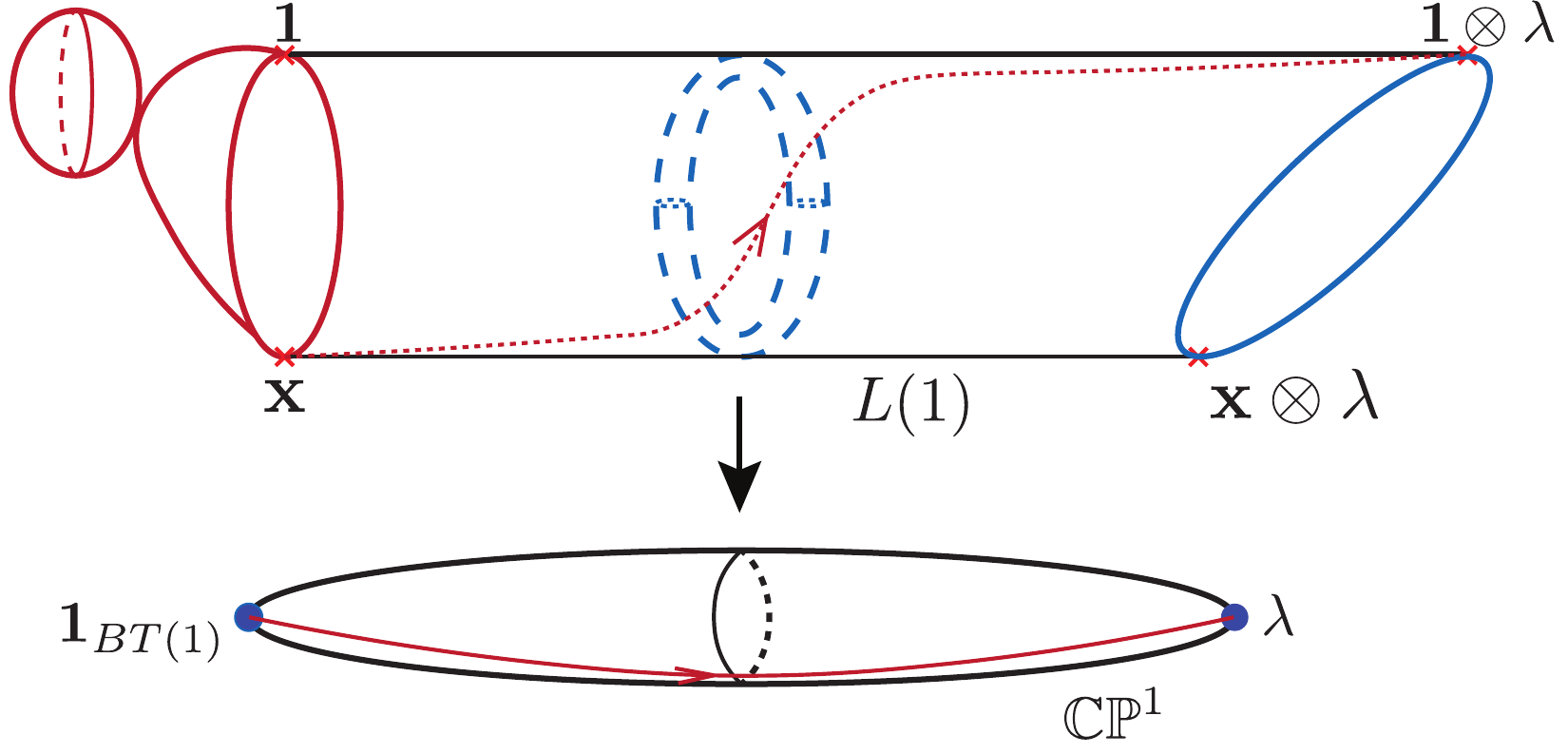}
		\caption{The geometric inputs of $W^{L}_T$.  The non-constant discs (with sphere bubbles) contribute to $\bunit$, and the flow line contributes to $\blambda$.}
		\label{fig_L_T}
	\end{center}
	\end{figure}	

	In the Fano case, $n_1(\beta_i+\alpha)=0$ whenever $\alpha \not=0$ by dimension reason.  Moreover $n_1(\beta_i)=1$ \cite{CO}.  More generally, the non-equivariant disc potential for compact semi-Fano toric manifold has been computed by \cite{CLLT} using Seidel representations.  The coefficients are given by the inverse mirror map.  We recall it in the following theorem.
	
	\begin{theorem} [\cite{CLLT}] \label{thm:inverse mirror map}
		Let $X$ be a compact semi-Fano toric manifold. Let $v_1,\ldots,v_m$ be the primitive generators of the one dimensional cones of the fan defining $X$. We denote by $D_i$ the toric prime divisor corresponding to $v_i$. Then
		\[
		\sum_{\alpha: c_1(\alpha)=0} n_1(\beta_i + \alpha) \bT^{\omega_X(\alpha)} = \exp(g_i(\check{q}(q))),
		\]
		where
		\begin{equation}\label{eqn:funcn_g}
		g_i(\check{q}):=\sum_{C}\frac{(-1)^{(D_i\cdot C)}(-(D_i\cdot C)-1)!}{\prod_{j\neq i} (D_j\cdot C)!}\check{q}^{C}, \quad \check{q}^{C}:=\bT^{\omega_X(C)} ,
		\end{equation}
		and the summation is over all effective curve classes $C\in H_2^{\mathrm{eff}}(X)$ satisfying
		\[
		-K_X\cdot C=0, D_i\cdot C<0 \text{ and } D_j\cdot C \geq 0 \text{ for all } j\neq i
		\]
		and $\check{q}=\check{q}(q)$ is the inverse of the mirror map $q=q(\check{q})$.
	\end{theorem}

	A similar result also holds for toric semi-Fano Gorenstein orbifolds \cite{CCLLTtoorbit}.  However we stick with the manifold case for simplicity.  Moreover, compactness of $X$ can be replaced by requiring $X$ to be semi-projective, so that the disc moduli spaces are compact.  A toric manifold $X$ is said to be semi-projective if it has a torus-fixed point, and the natural map $X\to \Spec{H^0(X,\mathcal{O}_X)}$ is projective \cite[Section 7.2]{CLS_toricbook}. By combining Theorem \ref{thm:superpotential} and \ref{thm:inverse mirror map}, we get the following:
	
	\begin{corollary} \label{cor:tor}
		For a toric fiber of a semi-projective and semi-Fano toric manifold, the equivariant disc potential equals to
		\[
		 W_T^L = \sum_{i=1}^m  \exp(g_i(\check{q}(q))) \bT^{\omega_X({\beta_i})} \exp (v_i \cdot (x_1,\ldots,x_d))+\sum_{i=1}^d \lambda_i x_i,
		\]
		where $g_i(\check{q}(q))$ is given by the inverse mirror map in \eqref{eqn:funcn_g}.
	\end{corollary}

	

\section{$\bS^1$-equivariant disc potential for the immersed two-sphere} \label{sec:imm}

In this section, we study the $\bS^1$-equivariant Floer theory of an immersed Lagrangian two-sphere $L_0$ which has a single nodal point. This will serve as a prototypical example for the authors' subsequent work \cite{HKLZ19} with Hong on the $\bS^1$-equivariant Floer theory of immersed fibers of Lagrangian torus fibrations in toric Calabi-Yau manifolds.


	\subsection{Equivariant Floer theory for Lagrangian immersions}
	\label{sec:Morse-Bott}
	In this section we describe a generalization of the equivariant Morse model to immersed Lagrangians with clean self-intersections. 
	
	In \cite{AJ}, Akaho and Joyce developed Lagrangian Floer theory for immersed Lagrangians $L$ with transverse self-intersection (which is a nodal point), generalizing the work of Fukaya-Oh-Ohta-Ono \cite{FOOO} in the embedded case.
	Their construction can be further generalized to the case where $L$ is an immersed Lagrangian with clean self-intersection as we describe below. See \cite{Sch-clean, CW15, Fuk_gauge} for the development of Floer theory of clean intersections in different settings. 
	
Let $(X,\omega_X)$ be as in Section \ref{sec:simplicial} and let $L$ be a closed, connected, relatively spin, immersed Lagrangian submanifold with clean self-intersections. We denote by $\iota:\tilde{L}\to X$ an immersion with image $L$ and by $\cI$ the self-intersection. As in Akaho-Joyce \cite{AJ}, the inverse image of $\cI$ under the immersion $\iota$ is assumed to be the disjoint union $\cI^-\coprod \cI^+\subset \tilde{L}$ of \emph{two branches}, each of which is diffeomorphic to $\cI$.  Then the following fiber product
\[
\tilde{L} \times_L \tilde{L}=\coprod_{i=-1,0,1} R_i,
\]
consists of the diagonal component $R_0$, and the two \emph{immersed sectors}
	\[
	\begin{split}
	R_{1}=\{(p_{-}, p_+)\in \tilde{L} \times \tilde{L}  \mid p_-\in\cI^-, p_+\in\cI^+, \iota(p_-)=\iota(p_+)  \},\\
	R_{-1}=\{(p_{+}, p_{-})\in \tilde{L} \times \tilde{L} \mid p_+\in\cI^+, p_-\in\cI^-, \iota(p_+)=\iota(p_-) \},
	\end{split}
	\]
	 We have canonical isomorphisms $R_0\cong\tilde{L}$ and $R_{-1}\cong R_{1}\cong \cI$. We also have the involution $\sigma \colon R_{-1}\coprod R_{1} \to R_{-1}\coprod R_{1}$ swapping the two immersed sectors, i.e, $\sigma(p_{-},p_{+})=(p_{+},p_{-})$.
	 
	 	Let $J_X$ be a compatible almost complex structure of $(X,\omega_X)$. For a map $\alpha \colon \{0,\ldots, k\}\to \{-1,0,1\}$, we consider quintuples $(\Sigma,\vec{z},u,\tilde{u},l)$ where
	 \begin{itemize}
	 	\item $\Sigma$ is a prestable genus $0$ bordered Riemann surface,
	 	\item $\vec{z}=(z_0,\ldots,z_k)$ are distinct counter-clockwise ordered smooth points on $\partial\Sigma$, 
	 	\item $u \colon (\Sigma,\partial \Sigma)\to (X,L)$ is a $J$-holomorphic map with $(\Sigma,\vec{z},u)$ stable,
	 	\item $\tilde{u} \colon \bS^1 \setminus \{\zeta_i:=l^{-1}(z_i):\alpha(i)\ne 0\} \to \tilde{L}$ is a local lift of $u|_{\partial\Sigma}$, i.e., 
	 	\[
	 	\iota\circ \tilde{u}=u\circ l \mbox{ and } \left(\lim_{\theta\to 0^-}\tilde{u}(e^{\mathbf{i}\theta} \zeta_i),\lim_{\theta \to 0^+}\tilde{u}(e^{\mathbf{i}\theta} \zeta_i)\right)\in R_{\alpha(i)},
	 	\]
	 	where $\alpha(i)\ne 0$ and $\mathbf{i}:=\sqrt{-1}$. 
	 	\item $l \colon \bS^1 \to \partial \Sigma$ is an orientation preserving continuous map (unique up to a reparametrization) characterized by that the inverse image of a smooth point is a point and the inverse image of a singular point consists of two points. 
	 \end{itemize}
	 Let $[\Sigma,\vec{z},u,\tilde{u},l]$ be the equivalence class of $(\Sigma,\vec{z},u,\tilde{u},l)$ given by the automorphisms. 
	 For $\beta\in H_2(X,L;\Z)$ and $\alpha$ a map as described above, we denote by $\cM_{k+1}(\alpha,\beta)$ the moduli space of (equivalence classes of) such quintuples $[\Sigma,\vec{z},u,\tilde{u},l]$ satisfying $u_*([\Sigma])=\beta$. The moduli spaces come with the following evaluation maps $\ev_i \colon \cM_{k+1}(\alpha,\beta) \to \tilde{L} \times_L \tilde{L}$ defined by
	 \[
	 \ev_i([\Sigma,\vec{z},u,\tilde{u},l])=\begin{cases} \tilde{u}(z_i)\in R_0 & \alpha(i)=0 \\
	 \left(\lim_{\theta\to 0^-}\tilde{u}(e^{\mathbf{i}\theta} \zeta_i),\lim_{\theta \to 0^+}\tilde{u}(e^{\mathbf{i}\theta} \zeta_i)\right)\in R_{\alpha(i)} & \alpha(i)\ne 0,
	 \end{cases}
	 \]
	 at the input marked points $i=1,\ldots, k$, and 
	 \[
	 \ev_0([\Sigma,\vec{z},u,\tilde{u},l])=\begin{cases} \tilde{u}(z_0)\in R_0 & \alpha(0)=0 \\
	 \sigma\left(\lim_{\theta\to 0^-}\tilde{u}(e^{\mathbf{i}\theta} \zeta_0),\lim_{\theta \to 0^+}\tilde{u}(e^{\mathbf{i}\theta} \zeta_0)\right)\in R_{-\alpha(0)} & \alpha(0)\ne 0,
	 \end{cases}
	 \]
	 at the output marked point. 
	 
	 For the convenience of writing, we will call an element of $\cM_{k+1}(\alpha,\beta)$ a \emph{stable polygon} if $\alpha(i)\ne 0$ for some $i\in\{0,\ldots,k\}$. In this case, the \emph{corners} of a polygon are the boundary marked points $z_i$ with $\alpha(i)\ne 0$. If $\alpha(i)=0$ for all $i$, we will simply refer to an element of $\cM_{k+1}(\alpha,\beta)$ as a \emph{stable disc}. Note that $\alpha$ dictates the branch jumps at the boundary marked points.

	The Kuranishi structures on $\cM_{k+1}(\alpha,\beta;L)$ are chosen to be weakly submersive. 
For $\vec{P}=(P_1,\ldots, P_k), P_1,\ldots,P_k\in S^{\bullet}(\tilde{L}\times_{L}\tilde{L};\Q)$, we denote by $\cM_{k+1}(\alpha,\beta,\vec{P})$ the fiber product 
	\[
	\cM_{k+1}(\alpha,\beta,\vec{P})=\cM_{k+1}(\alpha,\beta,L) \times_{(\tilde{L} \times_L \tilde{L})^k} \vec{P}
	\]
	in the sense of Kuranishi structures. We can then construct an $A_{\infty}$-structure $\tilde{\fm}$ on a countably generated subcomplex $C^{\bullet}(L;\Lambda_0)\subset S^{\bullet}(\tilde{L}\times_{L}\tilde{L};\Lambda_0)$ as in Section \ref{sec:simplicial}. 
    We define the maps $\tilde{\fm}_{k,\beta} \colon C^{\bullet}(L;\Lambda_0)^{\otimes k}\to C^{\bullet}(L;\Lambda_0)\subset S^{\bullet}(\tilde{L}\times_{L}\tilde{L};\Lambda_0)$ by
	\[
	\tilde{\fm}_{1,\beta_0}(P)=(-1)^n \partial P,
	\]
	and by
	\begin{equation*}
	\tilde{\fm}_{k,(\alpha,\beta)}(P_1,\ldots,P_k)=(\ev_0)_*\left(\mathcal{M}_{k+1}(\alpha,\beta;\vec{P})^{\fs}\right),
	\end{equation*}
	\begin{equation*}
	\tilde{\fm}_{k,\beta}(P_1,\ldots,P_k)=\sum_{\alpha } \tilde{\fm}_{k,(\alpha,\beta)}(P_1,\ldots,P_k)
	\end{equation*}
	for $(k,\beta)\ne (1,\beta_0)$. Notice that $\fm_{k,(\alpha,\beta)}(P_1,\ldots,P_k)=\emptyset$ unless $P_i$ is a singular chain on $R_{\alpha(i)}$. The $A_{\infty}$-structure map $\tilde{\fm}_k$ is given by
	\begin{equation*}
	\tilde{\fm}_k(P_1,\ldots,P_k)=\sum_{\beta} \bT^{\omega_X(\beta)}\tilde{\fm}_{k,\beta}(P_1,\ldots,P_k).
	\end{equation*}
	
	Similarly, we have a unital $A_{\infty}$-algebra $(C^{\bullet}(L;\Lambda_0)^+, \tilde{\fm}^+)$ that is homotopy equivalent to $(C^{\bullet}(L;\Lambda_0), \tilde{\fm})$ obtained via the homotopy unit construction.
	
	\begin{remark}
	We choose for each connected component $\cI_i\subset \cI$, a base point $p_i$, and a path of Lagrangian subspaces in $T_{p_i}X$ connecting $d\iota (T_{p_{i,-}}L)$ and $d\iota (T_{p_{i,+}}L)$, $\iota(p_{i,-})$=$\iota(p_{i,+})=p_i$. Then, the dimensions of the moduli spaces and the grading of singular chains on the immersed sectors depend on the choices of paths in the same manner as in \cite{AJ}. On the other hand, for our application in Section \ref{subsec_compubs1}, the immersed Lagrangian will be graded by a holomorphic volume form and these will be determined uniquely once paths of graded Lagrangian subspaces are chosen.  
	\end{remark}

Let $f \colon \tilde{L}\times_L \tilde{L}\to \R$ be a Morse function with a unique maximum point $\bunit$ on the diagonal component $\tilde{L}$. Let $C^{\bullet}(f;\Lambda_0)$ be the cochain complex generated by the critical points of $f$
\[
C^{\bullet}(f;\Lambda_0)=\bigoplus_{p\in \Crit(f)} \Lambda_0 \cdot p,
\]
and set
\[
CF^{\bullet}(L;\Lambda_0)=C^{\bullet}(f;\Lambda_0)\oplus \Lambda_0 \cdot \wunit \oplus \Lambda_0 \cdot \gunit. 
\]
Let
\[
\mathcal{X}_{-1}(L)=\{\Delta_p \mid p\in\Crit(f)\},
\]
where $\Delta_p$ is defined by \eqref{eq:delta_p}. Let $C^{\bullet}_{(-1)}(L;\Lambda_0)\subset S^{\bullet}(\tilde{L}\times_{L}\tilde{L};\Lambda_0)$ be the subcomplex generated elements in $\mathcal{X}_{-1}(L)$, and set
\[
C^{\bullet}_{(-1)}(L;\Lambda_0)^+=C^{\bullet}_{(-1)}(L;\Lambda_0)\oplus \Lambda_0  \cdot \bm{e}^+ \oplus \Lambda_0 \cdot \bm{f}. 
\]
We again identify $CF^{\bullet}(L;\Lambda_0)$ with $C^{\bullet}_{(-1)}(L;\Lambda_0)^+$ via the assignment $p\mapsto \Delta_p$ and equip $CF^{\bullet}(L;\Lambda_0)$ with a unital $A_{\infty}$-structure $\fm$ via homological peturbation as in Section \ref{sec:Morse model}. 

The $A_{\infty}$-structure map $\fm_k$ counts stable pearly trees with interior vertices decorated by stable polygons $(\alpha,\beta)$. Each edge of a tree corresponds to a negative gradient flow line of the Morse function $f$ (in a connected component of $\tilde{L}\times_{L}\tilde{L}$). A flow line connected to the $i^{\mathrm{th}}$ input marked point of a polygon $(\alpha,\beta)$ must be contained $R_{\alpha(i)}$, for otherwise the pearly tree does not exist. See Figure \ref{fig_comparison_1} for some (non-) examples.

For a pair $(L_1,L_2)$ of closed, connected, relatively spin, embedded Lagrangian submanifolds intersecting cleanly, the union $L=L_1\cup L_2$ is an immersed Lagrangian with clean self-intersections with $\tilde{L}=L_1\coprod L_2$. We choose the splitting $\iota^{-1}(\cI)=\cI^-\coprod \cI^+$ so that $\cI^-\subset L_1$ and $\cI^+\subset L_2$.

We can define a pearl complex $(CF^{\bullet}(L_1,L_2;\Lambda_0),\fm_1^{L_1,L_2})$ for the Lagrangian intersection Floer theory of $(L_1,L_2)$ as follows: 
\[
CF^{\bullet}(L_1,L_2;\Lambda_0)=\bigoplus_{p\in \Crit(f|_{R_{1}}) } \Lambda_0 \cdot p
\]
is the subcomplex of $CF^{\bullet}(L;\Lambda_0)$ generated by critical points of $f$ in $R_1$, and $\fm_1^{L_1,L_2}$ counts stable pearly trees $\Gamma\in\bm{\Gamma}_2$ with input and output vertices in $R_1$.




For the equivariant Morse model, we consider a $G$-equivariant Lagrangian immersion $\iota \colon \tilde{L} \to X$ with image $L=\iota(\tilde{L})\subset X$ a $G$-invariant immersed Lagrangian $L$.  
We can choose an admissible collection of Morse-Smale pairs $\{(f_N,\mathscr{V}_N)\}_{N\in \bN}$ on the finite dimensional approximations of $\tilde{L}_G \times_{L_G} \tilde{L}_G$ and proceed to define $A_{\infty}$-algebras $(CF^{\bullet}_G(L;\Lambda_0),\fm^{G})$ and $(CF^{\bullet}_G(L;\Lambda_0)^{\dagger},\fm^{G,\dagger})$ following Section \ref{sec:Morse model} and \ref{sec:partial_units}. This gives rise to interesting equivariant disc potentials even in the case when $L$ is unobstructed as we will see in Section~\ref{subsec_compubs1}.

	\begin{figure}[h]
	\begin{center}
		\includegraphics[scale=0.7]{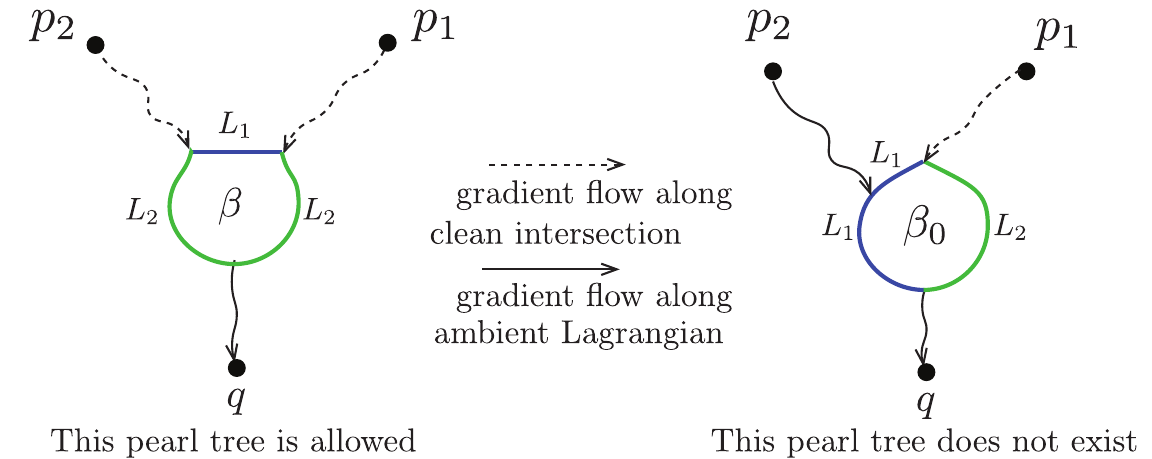}
		\caption{$L$ can be seen locally as a union of two cleanly intersecting Lagrangians $L_1$ and $L_2$. The pearly tree shown on the right does not exist due to inconsistency of Lagrangian boundary labels.}
		\label{fig_comparison_1}
	\end{center}
	\end{figure}

\subsection{$\bS^1$-equivariant disc potential via the gluing method}\label{subsec_compubs1}
First, we make the following observation which is important for finding isomorphisms between objects in the Fukaya category.

\begin{lemma} \label{lemma:b}
	Let $L_1,L_2$ be two graded Lagrangians in the Fukaya category that cleanly intersect with each other, and $\alpha \in CF^0(L_1,L_2)$.  Suppose that $CF^j(L_1,L_2)=0$ for all $j<0$.  For any $b'_i \in CF^1(L_i,\Lambda_+)$ and $w_i \in \Lambda_0$, we put $b_i=b'_i+w_i\cdot \gunit_{L_i}$. Then ,
	\begin{equation} \label{eq:m_1_b'}
	\mathfrak{m}^{b_1', b_2'}_1(\alpha) = \mathfrak{m}^{b_1, b_2}_1(\alpha).
	\end{equation}
	
	Similarly, if further $CF^j(L_2,L_1)=0$ for all $j<0$ and $\beta \in CF^0(L_2,L_1)$, then 
	\begin{align} \label{eq:m_2_b'}
	\mathfrak{m}^{b_1',b_2',b_1'}_2(\beta,\alpha) &= \mathfrak{m}^{b_1, b_2,b_1}_2(\beta,\alpha), \nonumber \\ 
	\mathfrak{m}^{b_2',b_1',b_2'}_2(\alpha,\beta) &= \mathfrak{m}^{b_2, b_1,b_2}_2(\alpha,\beta).
	\end{align}
\end{lemma}

\begin{proof}
For \eqref{eq:m_1_b'}, since the degree of $\gunit_{L_i}$ is $(-1)$, the extra terms
\[
\mathfrak{m}^{b_1', b_2'}_1(\alpha) - \mathfrak{m}^{b_1, b_2}_1(\alpha)=\sum_{\substack{k\ge 2 \\ k_1+k_2+1=k}} \fm_{k}(\underbrace{b_1',\ldots,b_1'}_{k_1},\alpha,\underbrace{b_2',\ldots,b_2'}_{k_2})-\fm_{k}(\underbrace{b_1,\ldots,b_1}_{k_1},\alpha,\underbrace {b_2,\ldots,b_2}_{k_2})
\]	
have degrees $2-k + (-1)\cdot (k-1) = 3-2k < 0$  in $CF^{\bullet}(L_1,L_2)$ and hence must vanish. Similarly for \eqref{eq:m_2_b'}, the extra terms in 	
\[
\mathfrak{m}^{b_1',b_2',b_1'}_2(\beta,\alpha) - \mathfrak{m}^{b_1, b_2,b_1}_2(\beta,\alpha)\in CF^{\bullet}(L_1;\Lambda_0)
\]
and 
\[
\mathfrak{m}^{b_2',b_1',b_2'}_2(\alpha,\beta) - \mathfrak{m}^{b_2, b_1,b_2}_2(\alpha,\beta)\in CF^{\bullet}(L_2;\Lambda_0)
\]
have degrees at most $(-2)$ and hence vanish for degree reason.
\end{proof}

In the setting of Section \ref{sec:partial_units}, the above lemma has a direct $G$-equivariant analog, by replacing $\mathfrak{m}_k$ with $\mathfrak{m}^G_k$ and setting $w_i = w_i^0 + \sum_{\deg \lambda=2} \phi_\lambda \cdot \lambda \in H^{\bullet}_G(\pt;\Lambda_0)$, as a consequence of Proposition \ref{thm:pull-out-lambda}.  

\begin{defn}
An \textit{isomorphism} between $(L_1,b_1)$ and $(L_2,b_2)$, where $b_i$ are weak bounding cochains, consists of $\alpha \in CF^0(L_1,L_2)$ and $\beta \in CF^0(L_2,L_1)$ satisfying 
\begin{align} \label{eq:iso_eq}
&\mathfrak{m}^{b_1, b_2}_1(\alpha) = \mathfrak{m}^{b_2, b_1}_1(\beta)= 0, \nonumber\\ 
&\mathfrak{m}^{b_1, b_2,b_1}_2(\beta,\alpha) =  \wunit_{L_1} + \mathfrak{m}^{b_1}_1 (\gamma_1), \nonumber \\ 
&\mathfrak{m}^{b_2, b_1,b_2}_2(\alpha,\beta) =  \wunit_{L_2} + \mathfrak{m}^{b_2}_1 (\gamma_2),
\end{align}
for some 
$\gamma_i \in CF^{\bullet}(L_i;\Lambda_0)$, $i=1,2$.
\end{defn}
Since $\mathfrak{m}^{b_i}_1(\gunit_{L_i}) = \wunit_{L_i} -  (1-  h(b_i) )\bunit_{L_i}$ for some $h(b_i) \in \Lambda_+$, by Lemma \ref{lemma:b}, it suffices to check that 
\[
\mathfrak{m}^{b_1', b_2'}_1(\alpha)=0, \quad \mathfrak{m}^{b_2', b_1'}_1(\beta)=0, \quad  \mathfrak{m}^{b_1', b_2',b_1'}_2(\beta,\alpha) = \bunit_{L_1}, \quad  \mathfrak{m}^{b_2', b_1',b_2'}_2(\alpha,\beta) = \bunit_{L_2}.
\] 

The existence of an isomorphism between $(L_1,b_1)$ and $(L_2,b_2)$ implies that they are quasi-isomorphic as objects in the Fukaya category (see e.g. \cite[Theorem~4.2]{CHL18}). 

Similarly, an isomorphism between $(L_1,G,b_1)$ and $(L_2,G,b_2)$ consists of $\alpha\in CF^{0}_G(L_1,L_2)$ and $\beta\in CF^0_G(L_2,L_1)$ such that the obvious analog of equations \eqref{eq:iso_eq} hold.

Let us now consider the following immersed two-sphere in 
\[
X := \{ (a, b, z) \in \C^2 \times \C^\times \mid ab = 1 + z \}.
\] 
We equip $X$ with the symplectic form inherited from the standard symplectic form on $\C^3$.
Let $\Pi \colon X \to \C^\times$ be the projection to the $z$-component. 
Regarding $\mathbb{S}^1$ as the unit circle in the complex plane $\C$, the space $X$ has a fiberwise Hamiltonian $\mathbb{S}^1$-action given by
$\eta \cdot (a, b, z) \mapsto \left(\eta \cdot a, \eta^{-1} \cdot b, z\right).$
One obtains the special Lagrangian torus fibration 
\begin{equation}\label{equ_specialLagtorusfib}
(|a|^2-|b|^2, |z|)
\end{equation} with respect to the holomorphic volume form 
\[
\Omega_X = \cfrac{da \wedge db}{z}.
\]
See \cite{Harvey-Lawson,Goldstein, Gross-eg}.
In particular, a singular fiber which will be denoted by $L_0$ occurs over the point $(|a|^2-|b|^2 = 0, |z| = 1)$ is an immersed two-sphere which intersects itself transversally at the point $(0, 0, -1)$. 

Let $\iota \colon \widetilde{L_0}=\bS^2\to X$ be a Lagrangian immersion with image the immersed two-sphere $L_0$.  In this case, the self-intersection $r\in L_0$ is a single nodal point fixed by the $\mathbb{S}^1$-action, and the two branches of the immersed loci are given by $U=(p,q)$, $V=(q,p) \in \bS^2\times \bS^2$, $p\ne q$, and $\iota(p)=\iota(q)=r$.

We have finite dimensional approximations of the fiber product 
\[
(\bS^2)_{\bS^1}\times_{(L_0)_{\bS^1}} (\bS^2)_{\bS^1}\cong \C\bP^{\infty} \coprod (\bS^2)_{\bS^1} \coprod \C\bP^{\infty} .
\]
Let $f_{\infty}: (\bS^2)_{\bS^1}\times_{(L_0)_{\bS^1}} (\bS^2)_{\bS^1} \to \R$ be a Morse function such that its restriction to each of the $\C\bP^{\infty}$-components is the perfect Morse function $\varphi \colon B\bS^1=\C\bP^{\infty}\to \R$ (see \eqref{eq:CP_infty}), and the restriction to the diagonal component $(\bS^2)_{\bS^1}$ is the sum of a horizontal lift of the perfect Morse function $\varphi$ and function on the total space $(\bS^2)_{\bS^1}\to \C\bP^{\infty}$ which restricts to a perfect Morse function $f_{\bS^2} \colon \bS^2 \to \R$ on generic fibers. We require $\Crit(f_{\bS^2})$ to be away from $p,q\in\bS^2$.

We will denote the critical points of $f_{\infty}|_{(\bS^2)_{\bS^1}}$ by $\one_{\bS^2} \otimes \lambda^{\ell}$ and $\pt_{\bS^2} \otimes \lambda^{\ell}$, $\ell\geq 0$, where $\one_{\bS^2}$ and $\pt_{\bS^2}$ are the maximal and minimal critical points of $\phi$, respectively. We will also denote the critical points of $f_{\infty}$ on the non-diagonal components by $U \otimes \lambda^{\ell}$ and $V \otimes \lambda^{\ell}$, $\ell\ge 0$, respectively.

\begin{figure}[h]
	\begin{center}
		\includegraphics[scale=0.35]{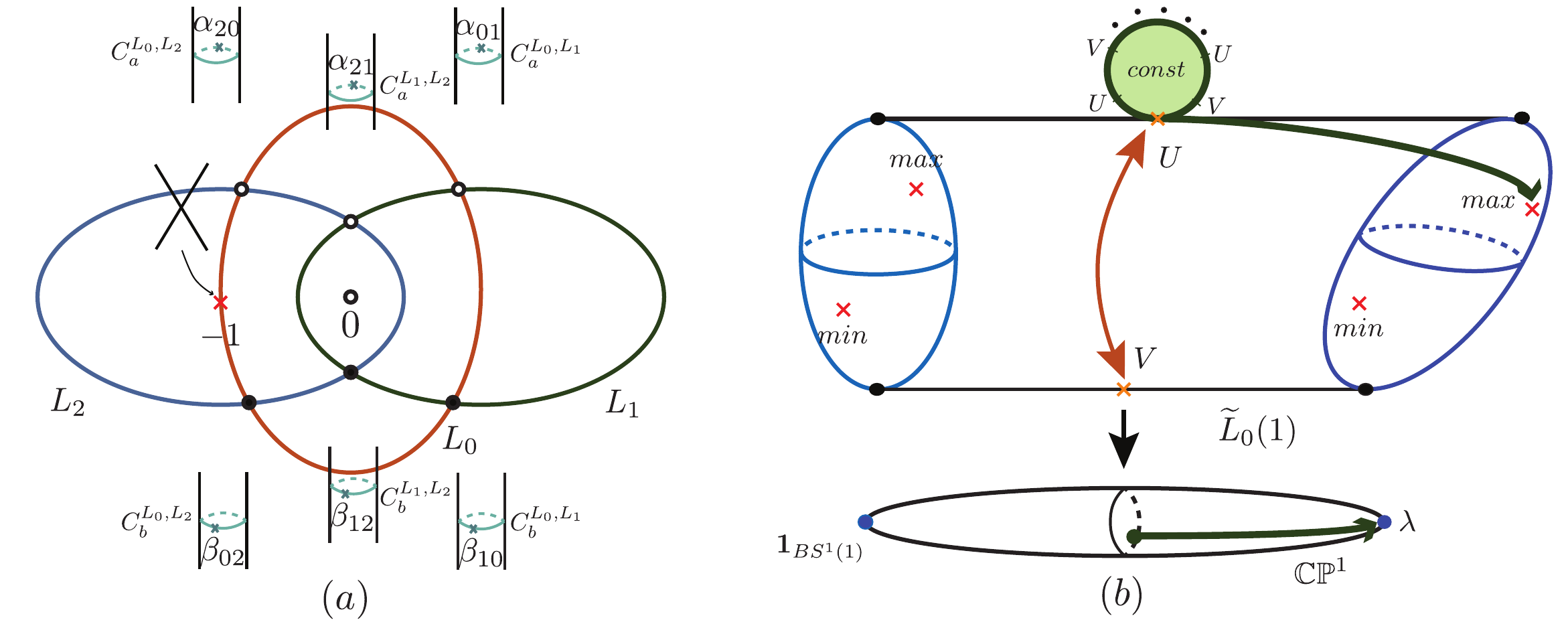}
		\caption{Comparison of three Lagrangians and constant disc contribution.}
		\label{fig_comparison}
	\end{center}
\end{figure}

Following Section~\ref{sec:partial_units} and \ref{sec:Morse-Bott}, we obtain an $A_\infty$-algebra $(CF_{\bS^1}^\bullet(L_0;\Lambda_0)^{\dagger},\fm^{\bS^1,\dagger})$ associated to $(L_0,\bS^1)$. We have
\[
CF_{\bS^1}^\bullet(L_0;\Lambda_0)^{\dagger}=CF^\bullet(L_0;\Lambda_0)\otimes_{\Lambda_0} \Lambda_0[\lambda],
\]
where $CF^\bullet(L_0;\Lambda_0)$ is the unital Morse complex associated to $f_{\infty}$. 
By abuse of notations, degree one immersed generators (graded with respect to the $\bS^1$-invariant holomorphic volume form $\Omega_{X}$) will simply be denoted by $U:=U\otimes 1$, and $V:=V\otimes 1$.
We now consider the formal deformations $b_0=uU+vV$ of $(L_0,\bS^1)$ for $(u,v)\in \Lambda^2_0$ with $\val (uv) >0$.

\begin{lemma}
	The formal deformations of $(L_0,\bS^1)$ by  $b_0=uU+vV$ are weakly unobstructed, i.e., $\mathfrak{m}_0^{\bS^1,\dagger, b_0}(1)\in H^{\bullet}_{\bS^1}(\pt;\Lambda_0)\cdot \one$.
\end{lemma}

\begin{proof}
	Since $L_0$ has minimal Maslov index two, $\mathfrak{m}_0^{\bS^1,\dagger, b_0}(1)$ can be expressed as 
	\[
	\mathfrak{m}_0^{\bS^1,\dagger, b_0}(1) =\mathfrak{m}_0^{b_0}(1) \otimes 1+ W(u,v) \cdot  \one_{\bS^2} \otimes \lambda
	\]
	by the same argument in the proof Proposition~\ref{prop:gen-m0}.
	Although $L_0$ does not bound any non-constant holomorphic discs, extra care is necessary for $\mathfrak{m}_0^{b_0}(1) $ as constant polygons might bubble off at the immersed loci. 
	We consider the anti-symplectic involution $\sigma \colon X \to X$ given by $(a,b, z) \mapsto (\overline{b}, \overline{a}, \overline{z})$, which preserves $L_0$ and swaps two immersed generators $U$ and $V$.
	In \cite{HKL}, by taking orientations respecting the involution, the constant polygon contributions are canceled out. 
	Thus $(L_0,b_0=uU+vV)$ is unobstructed, that is, $\mathfrak{m}_0^{b_0}(1)=0$. Therefore, 
	\[
	\mathfrak{m}_0^{\bS^1,\dagger, b_0}(1) = W(u,v) \cdot  \one_{\bS^2} \otimes \lambda.
	\]
\end{proof}

More precisely, the $\one_{\bS^2} \otimes \lambda$ appeared above is $\bunit_{\bS^2} \otimes \lambda=\blambda$ in the notation of Section \ref{sec:partial_units}.  By Lemma \ref{lem:G_unobstr}, the above $b_0$ can be modified (by adding terms involving $\gunit$) to $b_0^{\dagger}$ such that $\mathfrak{m}_0^{\bS^1,\dagger, b_0^{\dagger}}(1)= W^\wT(u,v) \cdot   \wlambda$.  By abuse of notations, we shall simply replace $b_0$ with $b_0^{\dagger}$ and denote $\wunit_{\bS^2}$ by $\one_{\bS^2}$, $\wunit=\wunit_{\bS^2}\otimes 1$ by $\one$, and $W^\wT$ by $W$.

\begin{remark}	
	This line of thoughts is similar to the work in \cite{HKL}.
	In this example $L_0$ does not bound non-constant holomorphic discs by maximal principle.  However, there are still highly interesting non-trivial contributions coming from constant polygons with corners at $U$ and $V$.  Note that for each constant polygon, $U$ and $V$ corners must occur in pairs (in order to return to the original branch).  This is the reason that the formal deformations of $(L_0,\bS^1)$ by $b_0=uU+vV$ are convergent for $(u,v)\in \Lambda_0^2$ with $\val (uv) >0$. 
\end{remark}

In order to compute $W(u,v)$, we compare $(L_0,\bS^1)$ with the $\bS^1$-equivariant Floer theory of Chekanov and Clifford tori $L_{1}$ and $L_{2}$ respectively.  
For the special Lagrangian torus fibration~\eqref{equ_specialLagtorusfib}, the Lagrangian torus over a point in the chamber $|z| < 1$ (resp. $|z| > 1$) is called a \emph{Chekanov} (resp. \emph{Clifford}) torus.
Since any pair of distinct fibers in~\eqref{equ_specialLagtorusfib} does \emph{not} intersect, they cannot be isomorphic. 
We apply Lagrangian isotopies to Chekanov torus and Clifford torus without intersecting the wall $z = 0$ to make them intersect as in Figure~\ref{fig_comparison} (a). 
Those isotoped Lagrangians are still $\Z$-graded and called Chekanov torus and Clifford torus respectively. 


The equivariant theory for the Lagrangian tori $L_{i}$ for $i=1,2$, is understood in the same way as in the toric case in Section~\ref{sec:toric}, except that in this local case they do not bound any non-constant holomorphic disc.  

We equip $L_1$ and $L_2$ with the \emph{non-trivial} spin structure along the $\bS^1$-orbit direction.  This will introduce extra systematic contributions to the orientation bundle of the moduli space of strips bounded by $L_1$ and $L_2$ and other Lagrangians.  The reason for doing this will be explained below.

We fix a perfect Morse function on each $L_i$ for $i=1,2$, such that the (compactified) unstable submanifolds of the degree one generators $X_i,Y_i$ are the hypertori chosen in \cite[Section~3.3]{HKL}.  The unstable submanifold of $X_i$ is transverse to the $\bS^1$-orbits in $L_i$.  We also have a Morse function on $(L_i)_{\bS^1}$ as in Section \ref{sec:toric}, whose critical points are labeled by $\eta \otimes \lambda^{\ell}$, $\ell\ge 0$ where $\eta$ is one of the following
\[
 \one_{L_i},\quad X_i, \quad Y_i, \quad X_i\wedge Y_i. 
\] 

We consider the formal deformations of  $L_i$ for $i=1,2$, by $b_i=x_i X_i+y_i Y_i$, $x_i,y_i\in \Lambda_+$. (Here $X_i:=X_i\otimes 1$ and  $Y_i:=Y_i\otimes 1$, by abuse of notations.) By Proposition~\ref{prop:gen-m0} and Lemma \ref{lem:h_i}, the equivariant disc potential of $(L_i,\bS^1)$ is given by
\[
W^{L_i}_{\bS^1} =  \lambda x_i.
\]
Since $L_i$ does not bound any non-constant holomorphic discs (particularly Maslov index zero discs), we have $W^{L^{(i)}}_{\bS^1} =(W^{L^{(i)}}_{\bS^1})^{\wT}$.

Each pair of Lagrangians $L_i$ and $L_j$, $0\le i<j\le 2$, intersect cleanly at two circles $C^{L_i,L_j}_\alpha$ and $C^{L_i,L_j}_\beta$ (see Figure \ref{fig_comparison}) on which $\bS^1$ acts freely. We fix a perfect Morse function on their clean intersections and denote by $CF^{\bullet}(L_i,L_j),$ and $CF^{\bullet}(L_j,L_i)$ the corresponding Lagrangian Floer intersection complexes. 

Let 
\begin{equation} \label{eq:iso012}
\alpha_{ij} \in CF^{0}(L_i,L_j), \quad \beta_{ji} \in CF^{0}(L_j,L_i)
\end{equation}
be the generators corresponding to the maximal critical points on $C^{L_i,L_j}_\alpha$ and $C^{L_i,L_j}_\beta$, respectively. Then, in the non-equivariant setting, we have the following gluing formula between the formal deformation space of the Lagrangians:

\begin{lemma}{\cite[Theorem~3.7]{HKL}}
The three formally deformed Lagrangians $(L_0, b_0)$, $(L_1, b_1)$ and $(L_2, b_2)$ are isomorphic, with the pairwise isomorphisms provided by \eqref{eq:iso012}, if and only if
	\begin{equation}\label{equ_isorel}
	\begin{cases}
	uv = 1-\exp (x_i) &\mbox{for $i = 1, 2$}, \\
	u = \exp (y_1), \\
	v = \exp (-y_2).
	\end{cases}
	\end{equation}
(The minus sign in front of $\exp(x_i)$ is due to the non-trivial spin structure we chose for $L_i$.)
\end{lemma}

We note that the isomorphism between $(L_1, b_1)$ and $(L_2, b_2)$ is given by wall-crossing, i.e., we have $\exp (x_1)=\exp (x_2)$ and $\exp (y_1)=\exp (y_2)(1-\exp (x_1))$. This was used in the proof of \cite[Theorem~3.7]{HKL} to deduce that $uv = 1-\exp (x_i)$ for $i=1,2$.

We now turn to the gluing formula in the equivariant setting. For $i,j\in\{0,1,2\}$ and $i\ne j$, let $CF^{\bullet}_{\bS^1}(L_i,L_j)$ be the $\bS^1$-equivariant Lagrangian Floer intersection complex with 
\[
CF^{\bullet}_{\bS^1}(L_i,L_j)=CF^{\bullet}(L_i,L_j)\otimes_{\Lambda_0} \Lambda_0[\lambda].
\]

\begin{lemma}
$(L_0,\bS^1,b_0)$, $(L_1,\bS^1,b_1)$ and $(L_2,\bS^1,b_2)$ are isomorphic, with the pairwise isomorphisms provided by 
\begin{equation} \label{eq:iso_s1}
\alpha_{ij}\otimes 1 \in CF^{0}_{\bS^1}(L_i,L_j), \quad \beta_{ji}\otimes 1 \in CF^{0}_{\bS^1}(L_j,L_i),
\end{equation}
if and only if $\eqref{equ_isorel}$ holds.
\end{lemma}

\begin{proof}
This is simply due to degree reason. Since the equivariant parameter $\lambda$ has  degree two, we have
	\[
\begin{cases}
(\fm_1^{\bS^1,\dagger})^{b_i,b_j}(\alpha_{ij}\otimes 1)= \fm_1^{b_i,b_j}(\alpha_{ij}), \\
(\fm_1^{\bS^1,\dagger})^{b_j,b_i}(\beta_{ji}\otimes 1)= \fm_1^{b_j,b_i}(\beta_{ji}), \\
(\fm_2^{\bS^1,\dagger})^{b_i,b_j,b_i}(\beta_{ji}\otimes 1, \alpha_{ij}\otimes 1)=\fm_2^{b_i,b_j,b_i}(\beta_{ji}, \alpha_{ij}), \\
(\fm_2^{\bS^1,\dagger})^{b_j,b_i,b_j}(\alpha_{ij}\otimes 1, \beta_{ij}\otimes 1)=\fm_2^{b_j,b_i,b_j}(\alpha_{ij}, \beta_{ji}) .
\end{cases}
\]
This implies that the gluing formula remains the same in the equivariant setting. Note that we have implicitly identified $CF^{\bullet}(L_i;\Lambda_0)$ with $CF^{\bullet}(L_i;\Lambda_0)\otimes 1\subset CF^{\bullet}_{\bS^1}(L_i;\Lambda_0)$ in the above equations.
\end{proof}

Since the disc potentials are compatible with the gluing formulas (see \cite[Theorem~4.7]{CHL18}), we can compute $W^{L_0}_{\bS^1}$ using the above lemma and the fact that $W^{L_1}_{\bS^1} = \lambda x_1$.

\begin{theorem} \label{thm:equiv-imm}
	The $\bS^1$-equivariant potential $W^{L_0}_{\bS^1}$ of the immersed two-sphere $L_0$ is given by
	\[
	W^{L_0}_{\bS^1} =\lambda W(u,v)=\lambda \log (1-uv) = -\lambda \sum_{j=1}^\infty \frac{(uv)^j}{j}.
	\]
\end{theorem}

\begin{proof}
We can show this by either comparing $L_0$ with $L_1$ or $L_2$. Suppose we pick the former. By the $A_{\infty}$-relation, we have
\[
((\fm_1^{\bS^1,\dagger})^{b_1,b_0})^2(\alpha_{10}\otimes 1)=\lambda W(u,v) (\fm_2^{\bS^1,\dagger})^{b_0,b_0,b_1}(\alpha_{10}\otimes 1,\one_{\bS^2}\otimes 1)- \lambda x_1 (\fm_2^{\bS^1,\dagger})^{b_0,b_1,b_1}(\one_{L_1}\otimes 1,\alpha_{10}\otimes 1).
\]	
\eqref{equ_isorel} forces the (LHS) to vanish, and thus
	\[
\lambda W(u,v) (\fm_2^{\bS^1,\dagger})^{b_0,b_0,b_1}(\alpha_{10}\otimes 1,\one_{\bS^2}\otimes 1)=\lambda x_1 (\fm_2^{\bS^1,\dagger})^{b_0,b_1,b_1}(\one_{L_1}\otimes 1,\alpha_{10}\otimes 1).
	\]
	On the other hand, we have
	\[
	(\fm_2^{\bS^1,\dagger})^{b_0,b_0,b_1}(\alpha_{10}\otimes 1,\one_{\bS^2}\otimes 1)
	= (\fm_2^{\bS^1,\dagger	})^{b_0,b_1,b_1}(\one_{L_1}\otimes 1,\alpha_{10}\otimes 1) = \alpha_{10}\otimes 1.
	\]
	Therefore, 
	\[
	W(u,v) = x_1.
	\]
	Since $\exp (x_1) = 1-uv$, we get
	\[
	W(u,v)=   \log (1-uv) = - \sum_{j=1}^\infty \frac{(uv)^j}{j}.
	\]
\end{proof}

\begin{remark}
	Note that if we did not take the non-trivial spin structure, then the gluing formula reads $\exp (x_1) = -1+uv$.  However, the LHS lies in $1 + \Lambda_+$, while the RHS lies in $-1 + \Lambda_+$. Since the equation has no solution, the formally deformed Lagrangians $(L_1,\bS^1,b_1)$ equipped with the trivial spin structure cannot be isomorphic to $(L_0,\bS^1,b)$. Moreover, $(L_1,\bS^1,b_1)$ equipped with the two different spin structures are not quasi-isomorphic as objects of the Fukaya category. This is reason why we equipped $L_1$ with the non-trivial spin structure in the first place.
\end{remark}

\bibliographystyle{amsalpha}
	\bibliography{geometry}	
\end{document}